\documentclass[12pt,reqno]{amsart}
\usepackage[backref]{hyperref}

\usepackage{mathrsfs}
\usepackage{amsmath}
\usepackage{latexsym}
\usepackage{amssymb}
\usepackage{amsfonts}
\usepackage{longtable}
\usepackage{enumerate}
\usepackage{colortbl}
\usepackage{hyperref}
\usepackage{color}

\newtheorem{theorem}{Theorem}[section]%
\newtheorem{lem}[theorem]{Lemma}
\newtheorem{cor}[theorem]{Corollary}%
\newtheorem{prop}[theorem]{Proposition}%
\newtheorem{lemma}[theorem]{Lemma}
\newtheorem{definition}[theorem]{Definition}%
\newtheorem{hypothesis}[theorem]{Hypotheses}%
\newtheorem{proposition}[theorem]{Proposition}%
\newtheorem{remark}[theorem]{Remark}%

\renewcommand\a{\alpha}
\renewcommand\b{\beta}
\newcommand\g{\gamma}
\renewcommand\d{\delta}

\renewcommand\r{\rho}
 \newcommand\s{\sigma}

\newcommand\ld{\lambda}

\newcommand\Om{{\it \Omega}}
\newcommand\Del{{\it \Delta}}
\newcommand\Sig{{\it\Sigma}}

  \newcommand\ov{\overline V}

\renewcommand\o{\omega}

\renewcommand\lg{\langle}
\newcommand\rg{\rangle}
\newcommand\nd{\mathrel{\bigm|\kern-.7em/}}

\newcommand\f{\noindent}

\newcommand\PSL{{\rm PSL}}
\newcommand\PSp{{\rm PSp}}
\newcommand\PSU{{\rm PSU}}

\newcommand\SL{{\rm SL}}
\newcommand\GL{{\rm GL}}
\newcommand\PGL{{\rm PGL}}

\newcommand\Aut{\hbox{\rm Aut\,}}

\newcommand\soc{\hbox{\rm soc}}

\newcommand\GF{{\rm GF}}

\renewcommand\mod{\hbox{\rm mod }}

\newcommand\PG{\hbox{\rm PG}}

\newcommand\M{{\rm M}}

\newcommand\Ga{{\it\Gamma}}
\newcommand\Sym{{\rm Sym}}
\newcommand\Ome{{\it \Omega}}
\newcommand\K{{\bf K}}

\newcommand\calB{{\mathcal B}}

\newcommand\ZZ{{\mathbb Z}}
\newcommand\Z{{\bf Z}}

\newcommand\C{{\bf C}}
\newcommand\N{{\bf N}}
\renewcommand\O{{\bf O}}

\newcommand\A{{\rm A}}
\renewcommand\S{{\rm S}}

\renewcommand\proof{\noindent{\it Proof.\ \ }}

\newcommand\AGammaL{{\rm A\Gamma L}}
\newcommand\PL{{\rm P\Gamma L} }
\newcommand\GammaL{{\rm \Gamma L}}
\newcommand\PO{{\rm P\Omega} }

\renewcommand\ov{\overline}
\renewcommand\l{\langle}
\renewcommand\r{\rangle}

\newcommand\mz{{\mathbb Z}}
\renewcommand\qed{\hskip10pt $\Box$\vspace{3mm}}

\def\ge{\geqslant}
\def\geq{\geqslant}

\def\leq{\leqslant}

\begin{document}

\title[Homogeneous graphs]{ The finite $k$-set homogeneous graphs}

\author[Li]{Cai Heng Li}
\address{Cai Heng Li\\
SUSTech International Center for Mathematics, and Department of Mathematics\\
Southern University of Science and Technology\\
Shenzhen 518055, Guangdong, P.R. China}
\email{lich@sustech.edu.cn}

\author[Yin]{Fu-Gang Yin}
\address{Fu-Gang Yin\\
School of Mathematics and Statistics\\
Beijing Jiaotong University\\
Beijing 100044, P.R. China}
\email{18118010@bjtu.edu.cn}

\author[Zhou]{Jin-Xin Zhou}
\address{Jin-Xin Zhou\\
School of Mathematics and Statistics\\
Beijing Jiaotong University\\
Beijing 100044, P.R. China}
\email{jxzhou@bjtu.edu.cn}

\date\today

\begin{abstract}
A classification is given of finite $k$-set-homogeneous graphs for $k\geqslant 2$, leading to a striking result that each finite $k$-set-homogeneous graph is $k$-homogeneous. It shows that $3$-set-homogeneous graphs are rare, consisting of the following graphs and their complements: $\C_5$, $\K_n\square\K_n$, $n\K_m$, the Schl\"{a}fli graph of order 27, the Higman-Sims graph, the MaLaughlin graph, {affine polar graphs, and elliptic orthogonal graphs}. As an ingredient for the proof, it is shown that all orbitals in a primitive permutation group of rank $4$ are self-paired, except for $\PSU_3(3)$ acting on 36 points.
\bigskip

\end{abstract}

\maketitle

\section{Introduction}

The main purpose of this paper is to solve the long-standing open problem:

\vskip0.1in
\f{\bf Problem A.} {\rm {Classify the finite graphs in which, for a given positive integer $k$, any two isomorphic induced subgraphs of order at most $k$ are equivalent under the automorphism group.}}
\vskip0.1in

Before proceeding, we give some background on this topic.
The strongest possible symmetry condition that one can impose {on a graph is {\it homogeneity}.}
A graph is called {\em homogeneous} if every isomorphism between any pair of isomorphic induced subgraphs extends to an automorphism of the graph. The study of graphs with {certain homogeneity-type properties} was initiated by Sheehan~{\cite{Sheehan1972,Sheehan1973,Sheehan1973a,Sheehan1974}},
and Gardiner~\cite{Gardiner1976} in the 1970s.
{It turns out that} the only finite homogeneous graphs are the following `trivial graphs':
\[\C_5,\ \K_3\square\K_3,\ n\K_m,\ \K_{n[m]}.\]
A graph $\Ga$ is called {\em set-homogeneous} if isomorphic subgraphs of $\Ga$ are equivalent under $\Aut(\Ga)$.
{In 1978, Ronse~\cite{Ronse}} proved that every finite set-homogeneous graph is homogeneous.
{Slightly weaker versions of homogeneity have been studied in the literature} (see Definition~\ref{Def-homog} for instance) and also {have been extended} to other mathematical structures. Over the past 50 years, { a fair amount of work has been done on objects with certain homogeneity-type properties. The reader may refer to the excellent survey paper of Macpherson~\cite{Macpherson-survey} for an overview.}

{One of these weaker conditions is as follows.}

\begin{definition}\label{Def-homog}
{\rm
Let $\Ga=(V,E)$ be a graph, $G\leqslant\Aut(\Ga)$, and  $k$ a positive integer.
\begin{itemize}
\item[(1)] If each isomorphism between any pair of isomorphic induced subgraphs of $\Ga$ of order at most $k$
extends to an automorphism of $\Ga$ { contained} in $G$, then $\Ga$ is called a {\it $(G, k)$-homogeneous} graph.
(The {\it order} of a graph is the number of vertices.)

\item[(2)] If any two isomorphic induced subgraphs of $\Ga$ of order at most $k$ are equivalent under $G$, then $\Ga$ is called a {\it $(G,k)$-set-homogeneous} graph.

\item[(3)] {When $G=\Aut(\Ga)$, a $(G, k)$-homogeneous or $(G, k)$-set-homogeneous graph is simply called {\em $k$-homogeneous} or {\em $k$-set-homogeneous}, respectively.}
\end{itemize}
}
\end{definition}

A major theme in the area has been the following classification problem.

\vskip0.1in
\f{\bf Problem B.}
{\rm Classify $k$-homogeneous graphs and $k$-set-homogeneous graphs.}
\vskip0.1in


This paper completely solves the problem, and particularly shows that a finite graph is $k$-homogeneous if and only if it is $k$-set-homogeneous.

{The study of Problem~B has lasted for over half a century, dating back to Sheehan and Gardiner's classification of finite homogeneous graphs in the 1970s mentioned above.} Later, Lachlan and Woodrow~\cite{Lachlan-Woodrow} obtained a classification of infinite countable homogeneous graphs. For more results regarding homogeneous graphs, we refer the reader to \cite{Enomoto} and \cite{Gray,Gray-Macpherson,GMPR,HH,Henson}.

By definition, $1$-homogeneous graphs and $1$-set-homogeneous graphs are exactly vertex-transitive graphs.
It is infeasible to obtain { an explicit list of vertex-transitive graphs.}
However, { an explicit classification has been achieved for}
finite $k$-homogeneous graphs with $k\geqslant 2$ in a collection of papers. Cameron~\cite{Cameron-6-transitive} proved that the list of finite $5$-homogeneous graphs is the same as that for homogeneous graphs;
Buczak~\cite{Buczak} classified finite $4$-homogeneous graphs; and Cameron and Macpherson~\cite{Cameron-3-hom} classified finite $3$-homogeneous graphs.
{We observe} that a graph $\Ga$ is $2$-homogeneous if and only if $\Ga$ and its complement $\ov\Ga$ are both arc-transitive. {Thus, a finite $2$-homogeneous graph is a complete multipartite graph or its complement, or an orbital graph of a primitive permutation group of rank $3$.
Note that primitive permutation groups of rank $3$ have been classified in \cite{Bannai,Kantor,Liebeck-affine,Liebeck-Saxl}.}

The main purpose of this paper is to solve Problem~B by proving that every $k$-set-homogeneous graph is $k$-homogeneous for $k\geqslant 2$, { so that $k$-set-homogeneous graphs are explicitly classified, as described in the previous paragraph.}
To state {our main theorem}, we fix some graph notation. Let $m$ and $n$ be positive integers. We denote the complete graph with $n$ vertices by $\K_n$, the cycle with $n$ vertices by $\C_n$, the complete multipartite graph with $n$ parts of
size $m$ by $\K_{n[m]}$, and its complement (the disjoint union of $n$ copies of $\K_m$) by $n\K_m$. Let $\K_n\square\K_n$ and $\K_n\times\K_n$ be the Cartesian product and direct product of two $\K_n$'s, respectively. For a graph $\Ga$, we denote by $\ov\Ga$ the {\it complement} of $\Ga$. Note that $\overline{\K_n\square\K_n}=\K_n\times\K_n$.

\begin{theorem}\label{th-main}
For each positive integer $k$, a finite graph is $k$-set-homogeneous if and only if it is $k$-homogeneous.
Moreover, { if $\Ga$ is a $k$-homogeneous graph with $k\geq 2$, then} one of the following holds:
\begin{itemize}
\item [{\rm(1)}] $k\geqslant 5$, and $\Ga$ is one of { $\C_5 (\cong\overline{\C_5})$, $\K_3\square\K_3 (\cong \overline{\K_3\square\K_3})$, $n\K_m$, and $\K_{n[m]}$; furthermore, $\Ga$ is homogeneous.}

{ \item[(2)] $\Ga$ is either the Schl\"{a}fli graph on $27$ vertices or its complement (the two non-trivial orbital graphs of $\mathrm{PSU}_4(2)$ acting on the isotropic lines); furthermore, $\Ga$ is $4$-homogeneous but not $5$-homogeneous.

\item[(3)] $\Ga$ is either $\K_n\square\K_n$ with $n\geqslant4$, or a graph from Theorem~\ref{th:3sethomo}(1)--(4), or the complement of any such graph; furthermore, $\Ga$ is $3$-homogeneous but not $4$-homogeneous.

\item[(4)] $\Ga$ is an orbital graph of a primitive permutation group of rank $3$ which is not covered by parts (1)--(3); furthermore, $\Ga$ is $2$-homogeneous but not $3$-homogeneous. }


\end{itemize}
\end{theorem}

{\noindent{\bf Remark on Theorem~\ref{th-main}.}\ For $k\geq2$, all $k$-homogeneous graphs have already been classified in the literature, { as mentioned above}, and so all graphs in Theorem~\ref{th-main}~(1)--(4) are known.}\medskip

By definition, a $k$-homogeneous graph is $k$-set-homogeneous.
The converse of this statement is clearly true for $k=1$.
{In the case where $k\ge4$,}
{Devillers~\cite{Devillers-2002paper}} proved that every finite $5$-set-homogeneous graph is set-homogeneous, and that a $4$-set-homogeneous graph is either homogeneous or one of the Schl\"{a}fli graph on $27$ vertices and its complement, which are $4$-homogeneous (see, for example, \cite[Proposition~3.2.1]{Devillers}).
{So, to prove Theorem~\ref{th-main}, it suffices to show that every $k$-set-homogeneous graph is also $k$-homogeneous for $k=2$ and 3.}

Before stating the results, we introduce some families of graphs.

\begin{definition}\label{Def-graph-families}
{\rm \begin{itemize}
\item[\rm (1)]  Let ${\rm GF}(q)$ be a field of order a prime power $q$ such that $q\equiv 1\ (\mod 4)$, and let $\o$ be a fixed primitive { element} of ${\rm GF}(q)$.
Let $S=\{\o^2, \o^4, \ldots, \o^{q-3}, \o^{q-1}\}$. The {\it Paley graph} of order $q$, denoted by {\rm Paley}$(q)$, is the graph with vertex set ${\rm GF}(q)$ { where} two vertices $x,y$ are adjacent if and only if $x-y\in S$.

\item[\rm (2)]  For a positive integer $n$, the {\it Hamming graph}, denoted by $\mathrm{H}(2,n)$, is the graph with vertex set {$\{(i, j) \mid 1\leq i, j\leq n\}$}, and for every vertex $(i, j)$, its neighbour set is
\[
\{ (i,j')\mid j' \in \{1,2,\dots,n \}\setminus \{j\}\} \cup \{ (i',j)\mid i' \in \{1,2,\dots,n \}\setminus \{i\}\}.
\]
Note that $\mathrm{H}(2,n)$ is just the Cartesian product $\K_n\square\K_n$.

\item[\rm (3)]  For $\epsilon \in \{+,-\}$, a positive integer $m$ and a prime power $q$, the {\it affine polar graph}, denoted by $\mathrm{VO}_{2m}^{\epsilon}(q)$, is a graph with vertex set $\GF(q)^{2m}$ equipped with a nondegenerate quadratic form $Q$ of type $\epsilon$, and two vertices $\alpha$ and $\alpha'$ being adjacent if and only if $Q(\alpha-\alpha')=0$ (see \cite[\textsection{3.3.1}]{BrouwerSRG}). Since a nondegenerate $2$-dimensional orthogonal vector space of type $-$ contains no singular vector, it follows that $\mathrm{VO}_{2}^{-}(q)\cong q^2\K_1$. As remarked in~\cite[\textsection{3.3.1}]{BrouwerSRG}, we see that $\mathrm{VO}_{2}^{+}(q)\cong \K_q\square\K_q=\mathrm{H}(2,q)$.

\item[\rm (4)] For a prime power $q$, the {\it elliptic orthogonal graph}, denoted by $\mathrm{\Gamma}(\mathrm{O}^{-}_6(q))$ is a graph with vertex set consisting of  all singular $1$-subspaces of the orthogonal space on {$\GF(q)^{6}$ of type $-$},
    and two vertices $\alpha=\langle w \rangle$ and $\alpha'=\langle w' \rangle$ being adjacent if and only if $w\perp w'$ (see \cite[\textsection{2.6.3}]{BrouwerSRG}). The Schl\"{a}fli graph is the complement of the elliptic orthogonal graph $\mathrm{\Gamma}(\mathrm{O}^{-}_6(2))$ (see \cite[Table~11.8]{BrouwerSRG}).
\end{itemize}}
\end{definition}

We recall that a transitive permutation group $G$ on a set $\Ome$ is {\it $2$-homogeneous} if
all $2$-subsets of $\Ome$ are equivalent under $G$.
In the next theorem, a classification of finite $2$-set-homogeneous graphs is given.

\begin{theorem}\label{th:2sethomo}
Let $\Ga=(V,E)$ be a finite  $(G,2)$-set-homogeneous graph, where $G\leq\Aut(\Ga)$.
Then one of the following holds:
\begin{itemize}
\item[{\rm(1)}] $\Ga\cong\K_n$ or $n\K_1$, and $G$ is a $2$-homogeneous permutation group on $V$.

\item[{\rm(2)}] $\Ga\cong\K_{m[b]}$ or $m\K_b$, and  $G\leqslant X\wr Y$,
where $X,Y$ are $2$-homogeneous permutation groups of degree $b,m$, respectively.

\item[{\rm(3)}] $G$ is primitive of rank $3$, and $\Ga$ is $(G,2)$-homogeneous.

\item[{\rm(4)}] $G\cong\PSU_3(3)$ with vertex stabiliser $\PSL_3(2)$, $\Ga$ is of order $36$ and
valency $14$ or $21$, but $\Aut(\Ga)=\PSU_3(3).2$ is $2$-homogeneous on $\Ga$.

\item[(5)] $\Ga$ is a Paley graph of order $p^r$, where $p\equiv5 \pmod{8}$ and $r$ is odd, and
$G\leq\AGammaL_1(p^r)$ is $2$-set-homogeneous on $\Ga$.
\end{itemize}
In particular, all $2$-set-homogeneous graphs are $2$-homogeneous.
\end{theorem}

{It is worth mentioning} that all graphs in Theorem~\ref{th:2sethomo}~(1)--(5) are $2$-homogeneous and already known in the literature.
{Since} a graph $\Ga$ is $(G, 2)$-set-homogeneous if and only if both $\Ga$ and $\ov\Ga$ are $G$-edge-transitive, { our} main approach for proving the above theorem is to {analyse} permutation groups of rank at most 5. It is well-known that all orbitals of a permutation group of rank 3 and {of even order} are self-paired. The following result shows that, except for a single group, all orbitals of a primitive permutation group of rank 4 are self-paired.

\begin{theorem}\label{trans-rank-4}
Let $G$ be a primitive permutation group of rank $4$. Then either all orbitals of $G$ are self-paired, or $G=\PSU_3(3)$ acting on $36$ points.
\end{theorem}

\begin{remark}
{\rm\begin{itemize}
\item[\rm (a)] More detailed information for the single group $\PSU_3(3)$ and its orbitals is given
in Theorem~\ref{rank-4}.
\item[\rm (b)] There are infinitely many imprimitive permutation groups of rank $4$ with exactly one non-trivial self-paired orbital.
Here is an example. Let $X$ be a $2$-homogeneous but not $2$-transitive permutation group, and let
$Y$ be a $2$-transitive permutation group. Then $X\wr Y$ is an imprimitive permutation group of rank $4$ with exactly one non-trivial self-paired orbital.
\end{itemize}}
\end{remark}

Our last theorem gives a classification of finite $3$-set-homogeneous graphs.

\begin{theorem}\label{th:3sethomo}
Let $\Ga=(V,E)$ be a finite $3$-set-homogeneous graph. Then, up to complement, $\Ga$ is isomorphic either to one of the graphs: $\C_5$, $\K_n$, $\K_{m[b]}$, $\K_n\square\K_n$, or to one of the following graphs:
\begin{itemize}
\item[{\rm(1)}] the affine polar graph $\mathrm{VO}_{2m}^{\epsilon}(2)$ with $m\geq 1$ and $\epsilon \in \{+,-\}$;

\item[{\rm(2)}]  the elliptic orthogonal graph $\mathrm{\Gamma}(\mathrm{O}^-_6(q))$ with $q$ a prime power;

\item[{\rm(3)}] an orbital graph of the Higman-Sims group $\mathrm{HS}$ acting on $100$ points;

\item[{\rm(4)}] an orbital graph of the McLaughlin group $\mathrm{McL}$ acting on $275$ points.
\end{itemize}
In particular, all $3$-set-homogeneous graphs are $3$-homogeneous.
\end{theorem}

Note that all graphs in Theorem~\ref{th:3sethomo}~(1)--(4) are actually $3$-homogeneous graphs, which {are known} in the literature.

The rest of the paper is organised as follows. In Sections~\ref{subsec:2reduce}--\ref{sec6:proof-main-theorems}, we deal with finite $(G, 2)$-set-homogeneous graphs.
In Section~\ref{sec6:proof-main-theorems}, we prove Theorem~\ref{th:2sethomo} and Theorem~\ref{trans-rank-4}.
In Section~\ref{set:3sethom}, we deal with finite $(G, 3)$-set-homogeneous graphs, and prove Theorem~\ref{th:3sethomo}.
In Section~\ref{set:proof-of-main}, we prove Theorem~\ref{th-main}.

\section{A reduction}\label{subsec:2reduce}
In this section, we shall reduce the proof of Theorem~\ref{th:2sethomo} to two cases: $G$ is a primitive permutation group of rank $5$ { and odd order}, and $G$ is a primitive permutation group of rank $4$ which has exactly one self-paired orbital. The former case will be treated in Section~\ref{subset:prime-rank5}, and the latter case will be treated in Sections~\ref{subsec:affineprimitive-rank4} and \ref{subsec:almostsimple-rank4}. Based on these, Theorem~\ref{th:2sethomo} and Theorem~\ref{trans-rank-4} will be proved in Section~\ref{sec6:proof-main-theorems}.

We first introduce several terms of permutation groups.
Let $\Ome$ be a finite nonempty set. Denote by $\Sym(\Ome)$ the symmetric group on $\Ome$. Let $G\leqslant\Sym(\Ome)$ be transitive.
Let $\Ome\times\Ome=\{(u,v)\mid u,v\in \Ome\}$.
An orbit $(u,v)^G$ of $G$ on $\Ome\times\Ome$ is called an {\it orbital} of $G$ on $\Ome$, and
$(v,u)^G$ is called the {\it paired orbital} of $(u,v)^G$.
An orbital  $(u,v)^G$ is called {\it self-paired} if it equals $(v,u)^G$.
The orbital $(u,u)^G$ is called {\it trivial}, and the others are non-trivial. To an orbital $E$ we
associate the digraph with vertex set $\Ome$ and arc set $E$, called the {\em orbital digraph} for
$E$, denoted by the pair $(\Ome, E)$. For a given point $u$, each orbital $E=(u,v)^G$ corresponds to an orbit of $G_u$ on $\Ome$, namely, $\{v\ |\ (u,v)\in E\}$. The orbits of $G_u$ on $\Om$ are called {\em suborbits} of $G$, and their lengths are called the {\em subdegrees} of $G$.
For a set $\Ome$, let $\Ome^{\{2\}}$ be the set of 2-subsets of $\Ome$, namely, $\Ome^{\{2\}}=\{\{u,v\}\mid u\not=v\in \Ome\}$,  and let $\Ome^{(2)}=\{(u,v)\mid u\not=v\in \Ome\}$.

\begin{lemma}\label{rank<=5}
Let $\Ga=(V,E)$ be a $(G, 2)$-set-homogeneous graph, where $G\leq\Aut(\Ga)$.
Then $G$ is a permutation group on $V$  of rank at most $5$, and
one of the following holds:
\begin{itemize}
\item[{\rm(1)}] $G$ is of rank $2$, and $\Ga$ is a complete graph or an empty graph.

\item[{\rm(2)}] $G$ is of rank $3$, and {either $\Ga$ is a complete graph or an empty graph, or $\Ga$ is $(G, 2)$-homogeneous}.

\item[{\rm(3)}] $G$ is of rank $4$, and exactly one of the $3$ non-trivial orbitals is self-paired.

\item[{\rm(4)}] $G$ is of rank $5$, and none of the $4$ non-trivial orbitals is self-paired.

\end{itemize}
\end{lemma}
\proof
By definition, both $\Ga$ and its complement $\overline{\Ga}$ are $(G,2)$-set-homogeneous. So we may assume that $E$ is not empty. Then $G$ is transitive on $E$.

Assume first that $E=V^{\{2\}}$. Then $\Ga$ is a complete graph. If $G$ is transitive on $V^{(2)}$, then $G$ has rank $2$. If $G$ is intransitive on $V^{(2)}$, then for any $(u, v)\in V^{(2)}$, we have $V^{(2)}=(u, v)^G\cup (v, u)^G$, and so $(u, v)^G$ and $(v, u)^G$ are the only two orbits of $G$ on $V^{(2)}$. So $G$ is of rank $3$.

Assume now that $E\not=V^{\{2\}}$. Then $\Ga$ has two types of subgraphs of order $2$, namely, $\K_2$ and $2\K_1$, and then $G$ is transitive on $E$ and on $V^{\{2\}}\setminus E$. Let $\{u, v\}\in E$, and let $\{u, w\}\in V^{\{2\}}\setminus E$. Then the possible non-trivial orbitals of $G$ on $V$ are:
\[(u, v)^G, (v, u)^G, (u, w)^G, (w, u)^G.\]
If every non-trivial orbital of $G$ is self-paired, then $(u, v)^G=(v, u)^G$ and $(u, w)^G=(w, u)^G$. This implies that $G$ has rank $3$ and $\Ga$ is $(G, 2)$-homogeneous. If exactly one of the non-trivial orbitals of $G$ is self-paired, then $G$ has three pair-wise distinct non-trivial orbitals, and so $G$ has rank $4$. If no non-trivial orbital of $G$ is self-paired, then the above four orbitals of $G$ are pair-wise distinct, and then $G$ has rank $5$. This proves our lemma.
\hfill\qed

Let $G$ be a transitive permutation group on a set $\Ome$.
A nonempty subset { $\Del$ }  of $\Ome$ is called a {\em block} for $G$ if
for each $g\in G$, either $\Del^g=\Del$ or $\Del^g\cap \Del=\emptyset$. 
The set $\{\Del^g \mid g\in G\}$ of blocks is called a {\em block system} for $G$.
We say that $G$ is {\em primitive} if the only blocks for $G$ are {\it trivial} blocks, namely,
the singleton subsets or the whole of $\Ome$.
Otherwise, $G$ is called {\it imprimitive} if there exist non-trivial blocks for $G$.
Further, $G$ is called {\it $2$-transitive} or {\it $2$-homogeneous} if $G$ is transitive on $\Ome^{(2)}$ or $\Ome^{\{2\}}$, respectively. Note that a $2$-transitive group has rank $2$, and that a $2$-homogeneous group which is not $2$-transitive has rank $3$.

It is easy to see that a complete multipartite graph $\K_{m[b]}$ and its complement $m\K_b$ are $2$-homogeneous.

\begin{lemma}\label{imp}
Let $\Ga=(V,E)$ be a $(G,2)$-set-homogeneous graph.
If $G$ is imprimitive on $V$, then $\Ga=\K_{m[b]}$ or $m\K_b$ with $m,b\geqslant2$.
Moreover, $\Ga$ is homogeneous.
\end{lemma}
\proof
Let $\calB$ be a non-trivial block system, and let $B\in\calB$.
Then $\{u,v\}$ with $u,v\in B$ is not equivalent under $G$ to $\{u,w\}$ with $w\in V\setminus B$.
Thus, exactly one of $\{u,v\}$ and $\{u,w\}$ is an edge of $\Ga$.
If $\{u,v\}$ is an edge, then $\Ga=m\K_b$, and  
if $\{u,w\}$ is an edge, then $\Ga=\K_{m[b]}$, where $b=|B|$ and $m=|\calB|$. It then follows from \cite{Gardiner1976} that $\Ga$ is homogeneous.
\hfill\qed

The above Lemmas~\ref{rank<=5} and \ref{imp} lead us { to study} primitive permutation groups of rank at most $5$.
{  The {\it socle}  of a group $G$, denoted by $\soc(G)$, is the subgroup generated by all minimal normal subgroups of $G$.}


\begin{prop}{\rm ({\rm\cite[Corollary~2.2]{Cuypers}})} \label{th-rank-4}
Let $G$ be a primitive permutation group on $\Om$ of rank at most $5$.
Then one of the following holds:
\begin{enumerate}
  \item [{\rm (1)}]\ $G$ is affine, that is, $\soc(G)$ is abelian.

  \item [{\rm (2)}]\ $G$ is almost simple, that is, $\soc(G)$ is nonabelian simple.

  \item [{\rm (3)}]\ $T^t\lhd G\leqslant G_0\wr \mathrm{S}_t$, and $\Ome=\Del^t$,
  where $t=2,3$ or $4$, $G_0$ is a $2$-transitive group on $\Del$, and  $T=\soc(G_0)$ is nonabelian simple.

  \item [{\rm (4)}]\ $\soc(G)=T\times T$ with $T=\A_5$,  $\A_6$, $\PSL_2(7)$ or $\PSL_2(8)$,
 and  the point-stabiliser $\soc(G)_\a$ is the diagonal subgroup of $T\times T$.
\end{enumerate}
\end{prop}

The next lemma reduces our work to affine groups and almost simple groups.

\begin{lemma}\label{non-HA-AS}
Let $\Ga=(V,E)$ be a $(G,2)$-set-homogeneous graph, and let $G$ be primitive on $V$.
Assume that $G$ is neither affine nor almost simple.
Then
\begin{itemize}
 \item [{\rm (1)}] all orbitals of $G$ are self-paired.

 \item [{\rm (2)}] $T^2\lhd G\leqslant G_0\wr S_2$, and $\Ome=\Del^2$, $G_0$ is a $2$-transitive group on $\Del$, and
$\Ga=\K_n\square\K_n$ with $n=|\Del|$ is $3$-homogeneous but not $4$-homogeneous.
\end{itemize}
\end{lemma}
\begin{proof}
We need to deal with the candidates { given} in case~(3) and case~(4) of Proposition~\ref{th-rank-4}.

Let $T=\A_5$,  $\A_6$, $\PSL_2(7)$ or $\PSL_2(8)$.
Then any two elements of $T$ of the same order are fused in $\Aut(T)$, and so each fusion class of $T$ is inverse-closed.
An orbital graph of $G$ satisfying Proposition~\ref{th-rank-4}~(4) is a Cayley graph of $T$ with connecting set being a fusion class of $T$, and thus it is arc-transitive. So each orbital of $G$ is self-paired.
In particular, $\Ga$ is 2-homogeneous.

Assume now that Proposition~\ref{th-rank-4}~(3) occurs.
Then $\Ga$ has diameter $t$, and each orbital is clearly self-paired.
Thus $G$ has $t+1$ self-paired orbitals.
Since $\Ga$ is 2-set-homogeneous, we conclude that $t=2$.
It follows that $\Ga=\mathrm{H}(2,n)=\K_n\square\K_n$, where $|\Del|=n$. By \cite{Cameron-3-hom}, $\K_n\square\K_n$ is $3$-homogeneous but not $4$-homogeneous.
\hfill\qed
\end{proof}

Finally, we obtain a reduction for proving Theorem~\ref{th:2sethomo}.

\begin{cor}\label{reduction}
Let $\Ga$ be a $(G,2)$-set-homogeneous graph such that $G$ is primitive on $V$.
Then either $\Ga$ is $(G,2)$-homogeneous, or one of the following holds:
\begin{itemize}
\item[{\rm (1)}] { $G$ is affine or almost simple, and has rank $4$; furthermore, $G$ has exactly one non-trivial self-paired orbital;}



\item[{\rm (2)}] { $G$ has odd order and rank $5$, with no non-trivial self-paired orbital.}


\end{itemize}
\end{cor}
\proof
Assume that $\Ga$ is not $(G,2)$-homogeneous.
Then $G$ has an orbital which is not self-paired, and thus $G$ is of rank $4$ or $5$ by Lemma~\ref{rank<=5}.

If $G$ is of rank $4$, then exactly one of the three non-trivial orbitals of $G$ on $V$ is self-paired,
and by Lemma~\ref{non-HA-AS}, $G$ is either affine or almost simple.

Suppose that $G$ is of rank $5$.
Then none of the four non-trivial orbitals of $G$ is self-paired.
If $G$ has an involution $g$, then $g$ moves at least one vertex $u$ say, and thus $g$ interchanges $u$ and $u^g$.
It implies that $(u, u^g)^G$ is a self-paired orbital, which is a contradiction.
Thus $G$ contains no involution, and $G$ is of odd order.
\hfill\qed

\section{Primitive permutation groups of rank $5$}\label{subset:prime-rank5}

In this section, we shall determine the primitive groups of rank $5$ which have no non-trivial self-paired orbitals (namely, the groups in Corollary~\ref{reduction}~(2)). The arguments depend on some properties of $1$-dimensional affine groups $\AGammaL_1(p^d)$. We notice that $\GammaL_1(p^d)=\ZZ_{p^d-1}{:}\ZZ_d$ is metacyclic, and that any subgroup and any factor group of a metacyclic group are metacyclic.

\begin{lemma}\label{Hall-2'-GammaL_1(q)}
A Hall $2'$-subgroup of a metacyclic group is a normal subgroup.
\end{lemma}
\proof
Let $Y$ be a  metacyclic group, and let $\ov Y=Y/\O_{2'}(Y)$. Let $N=\O_2(\ov Y)$.
Then $N$ is the Fitting subgroup of $\ov Y$, and so $\C_{\ov Y}(N)\leqslant N$.
Suppose that $N\not=\ov Y$.
Let $g\in\ov Y\setminus N$ be of order $p$ with $p$ an odd prime.
Then $Z:=\l N, g\r=N\ {:}\ \l g\r$, and $g$ does not centralize $N$ since $\C_{\ov Y}(N)\leqslant N$.
Let $\Phi(N)$ be the Frattini subgroup of $N$. If $Z/\Phi(N)$ is abelian, then $Z'\leq\Phi(N)\leq\Phi(Z)$, and then by a result of Wielandt (see \cite[5.2.16]{Robinson}), we have $Z$ is nilpotent, which is impossible.
Thus, $g$ does not centralize $N/\Phi(N)$. Since $Y$ is a metacyclic group, $N$ is metacyclic, and $N/\Phi(N)\cong\ZZ_2^2$.
It follows that $Z/\Phi(N)\cong\ZZ_2^2\ {:}\ \ZZ_3\cong \A_4$, which is a contradiction since $\A_4$ is not metacyclic.
Therefore, $N=\ov Y$, and $\O_{2'}(Y)$ is a Hall $2'$-subgroup of $Y$,
{namely, a Hall $2'$-subgroup of a metacyclic group of normal.}
\hfill\qed

In the rest of this section,
let $G$ be a primitive permutation group on $\Ome$ of rank $5$ such that
none of its $4$ non-trivial orbitals is self-paired.
Then $G$ is of odd order, and $G$ is affine with socle $V\cong \ZZ_p^d$, where $p$ is an odd prime and $d$ is a positive integer.
Identify $V$ with a $d$-dimensional vector space over the finite field $\GF(p)$ and identify $\Ome$ with  $V$.
Then
\[G=V{:}G_0,\]
 where $G_0\leqslant\GL_d(p)$ acts naturally on $V$, and $V$ acts on itself by addition.
Note that, for every non-self-paired  orbital $(0,v)^G$ with $v\in V$, since  $(-v,0)=(0,v)^{-v} \in  (0,v)^G $, it follows that $(0,-v) \notin (0,v)^G$. Therefore,  $G_0$ partitions the set $V\setminus\{0\}$ of
non-zero vectors into 4 orbits: $\Del_1$, $-\Del_1$, $\Del_2$, and $-\Del_2$, where
$-\Del_i=\{-v \mid v\in\Del_i\}$ with $i=1$ or 2.

\begin{lemma}\label{1-dim}
The stabiliser $G_0$ is a subgroup of $\GammaL_1(p^d)$.
\end{lemma}
\proof
Let $z$ be the unique involution of the centre $\Z(\GL_d(p))$.
Then $z$ inverts all vectors of $V$, and so $\Del_i^z=-\Del_i$ for $i=1$ or 2.
Let $X=\l G,z\r$.
Then $X$ is a soluble primitive permutation group of rank 3.
By~\cite[Theorem 1.1]{Foulser-1969}, we conclude that $G_0\lhd X_0\leqslant\GammaL_1(p^d)$.
\hfill\qed

Let $H$ be a Hall $2'$-subgroup of $\GammaL_1(p^d)$. Then $H\unlhd\GammaL_1(p^d)$ by Lemma~\ref{Hall-2'-GammaL_1(q)}.

\begin{lemma}\label{4-orbits-=}
The Hall $2'$-subgroup $H$ has exactly $4$ orbits on $V\setminus\{0\}$,
which are the $4$ orbits of $G_0$ on $V\setminus\{0\}$ and have equal sizes.
Furthermore, $d$ is odd, and $p\equiv 5\pmod{8}$.
\end{lemma}
\proof
Since $G_0\leqslant H$ and $G_0$ has exactly $4$ orbits on $V\setminus\{0\}$, it follows that $H$ has at most $4$ orbits on $V\setminus\{0\}$. Clearly, $|V\setminus\{0\}|=p^d-1$ is even, so $H$ is intransitive on $V\setminus\{0\}$ because $H$ has odd order. Since $\GammaL_1(p^d)$ is transitive on $V\setminus\{0\}$ and $H\unlhd\GammaL_1(p^d)$, all orbits of $H$ on $V\setminus\{0\}$ have the same size. Thus $H$ has  $2$ or $4$ orbits on $V\setminus\{0\}$, with the same size.

Suppose that $H$ has exactly 2 orbits on $V\setminus\{0\}$, say $\Sig_1$ and $\Sig_2$.
Since $H$ is of odd order, a vector $v$ is not equivalent under the action of $H$ to $-v$.
It follows that $\Sig_2=-\Sig_1$. Since $\Del_1$, $-\Del_1$, $\Del_2$, and $-\Del_2$ are all orbits of $G_0$ on $V\setminus\{0\}$, we would have $|\Sig_1|=|\Del_1|+|\Del_2|$, contradicting that $|\Sig_1|$ is odd.

Therefore, $H$ has exactly 4 orbits on $V\setminus\{0\}$, of size $m$ say.
Then $m$ is odd, and $|V\setminus\{0\}|=p^d-1=4m$.
It follows that $d$ is odd, and $p\equiv 5 \pmod{8}$.
\hfill\qed

\begin{prop}\label{AGL_1(p^r)}
Let $\Ga$ be a $(G, 2)$-set-homogeneous graph which is not $(G,2)$-homogeneous.
Suppose further that $G$ is of rank $5$.
Then  $|G|$ is odd, and
\begin{enumerate}
 \item [{\rm (1)}]  $G_0\leqslant\GammaL_1( q)$ where $q=p^d$ such that $p\equiv 5\pmod{8}$ and $d$ is odd, and
 \item [{\rm (2)}]  $\Ga$ is a Paley graph ${\rm Paley}{(q)}$.
\end{enumerate}
Moreover, $\Aut(\Ga)$ is of rank $3$, and $\Ga$ is indeed $2$-homogeneous, but not $3$-homogeneous.
\end{prop}
\proof By Corollary~\ref{reduction}~(2),  $|G|$ is odd.
From Lemmas~\ref{1-dim} and \ref{4-orbits-=}, we obtain part (1).

Let $L=H\cap \GL_1(q)$, where $H$ is the normal Hall $2'$-subgroup of $\GammaL_1(q)$. Then $L\cong\ZZ_{(q-1)/4}$, and $L$ has exactly 4 orbits on $V\setminus\{0\}$, which are the $4$ orbits of $H$ and so they are the $4$ orbits of $G_0$. Let $z$ be the unique involution of the centre $\Z(\GL_1(q))$. Then $\l L,z\r=\ZZ_{(q-1)/2}$ partitions $\GF(q)\setminus\{0\}$ into two orbits,
one of which is $S=\{\o^2, \o^4, \ldots, \o^{q-3}, \o^{q-1}\}$, {where $ \o$ is a primitive element of $\GF(q)$}. It follows that $\Ga$ is a Paley graph, proving part (2).

The group $\l L,z\r$ is a subgroup of $\Aut(\Ga)$, and is of rank 3. So $\Aut(\Ga)$ is of rank $3$, and then $\Ga$ is $2$-homogeneous. By \cite{Cameron-3-hom}, $\Ga$ is not $3$-homogeneous. \hfill\qed

\section{Affine primitive permutation groups of rank 4}\label{subsec:affineprimitive-rank4}

In this section, we deal with affine groups (in Corollary~\ref{reduction}~(1)), and prove the following theorem.

\begin{theorem}\label{affine-rank-4}
Let $G$ be an affine primitive rank $4$ permutation group.
Then each orbital of $G$ is self-paired, and no non-trivial orbital graphs of $G$ are $2$-set-homogeneous.
\end{theorem}

For a group $X$, let $X^{(\infty)}$ be the smallest normal subgroup of $X$ such that $X/X^{(\infty)}$ is solvable.
To prove Theorem~\ref{affine-rank-4}, suppose for a contradiction that $G$ is an affine primitive permutation group of rank $4$ which has exactly one non-trivial self-paired orbital, and seek a contradiction.

\begin{lemma}\label{HA-rank4-pty}
Let $G=V{:}G_0$ be an affine primitive rank $4$ permutation group  which has exactly one non-trivial self-paired orbital, where $V\cong \ZZ_p^d$ is the socle of $G$.
Then $p$ is an odd prime.
Let $z$ be the involution of $\Z(\GL_d(p))$ and let $X_0=\l G_0, z\r$.
Then the following hold:
\begin{itemize}
\item[{\rm(1)}]  $X_0=G_0\times \l z\r$ and $V{:} X_0$ is of rank $3$;

\item[{\rm(2)}]  $\Z(G_0)$ is of odd order;

\item[{\rm(3)}]  $X_0^{(\infty)}\leqslant G_0$, and $V{:}X_0^{(\infty)}$ is of rank at least $4$.
\end{itemize}
\end{lemma}
\begin{proof}
View the socle $V$ as a $d$-dimensional vector space over the field $\GF(p)$, and identify $G$ with a permutation group on $V$.
Then   $G_0\leqslant\GL_d(p)$ acts naturally on $V$, and $V$ acts on itself by addition.
Note that for every non-self-paired orbital $(0,v)^G$ with $v\in V$, we have $(0,-v) \notin (0,v)^G$.
Thus $p$ is odd.
Since $z$ inverts every vector of $V$, we see that $z$ is not in  $ G_0 $, and so
\[\l G_0, z\r=G_0\times\l z\r.\]
Since those two non-self-paired orbitals are fused by $z$, the group $X_0=G_0\times\l z\r$ has exactly two orbits on $V\setminus\{0\}$,  and so $V{:} X_0$ is of rank $3$, as in part~(1).

Since $G=V{:} G_0$ is an affine primitive group, $G_0$ is an irreducible subgroup of $\mathrm{GL}_d(p)$. In view of \cite[Lemmas 2.10.1 and 2.10.2]{K-Lie}, we know that $\mathbf{C}_{\mathrm{GL}_d(p)}(G_0)$ is cyclic. Clearly, $z\in \mathbf{C}_{\mathrm{GL}_d(p)}(G_0)$, so $z$ is the unique involution of $\mathbf{C}_{\mathrm{GL}_d(p)}(G_0)$. Then $Z(G_0)$ is of odd order because $X_0=G_0\times\l z\r$, proving part (2).

Since $X_0^{(\infty)}\leq G_0$ and $G_0$ has four orbits on $V$, $X_0^{(\infty)}$ has at least four orbits on $V$, and then $V{:}X_0^{(\infty)}$ has rank at least $4$. Part (3) holds.
\hfill\qed
\end{proof}

Primitive affine groups of rank 3 are classified into three classes
labelled (A), (B), and (C) by Liebeck in \cite{Liebeck-affine}, as stated in the following theorem.

\begin{theorem}[\cite{Liebeck-affine}]\label{subdegree}
Let $X$ be a finite primitive affine permutation group of rank $3$ and of degree $p^d$, with socle $V$, where $V\cong\mz_p^d$ for some prime $p$, and let $X_0$ be the stabiliser of the zero vector in $V$. Then $X_0$ belongs to one of the following classes $($and, conversely, each of the possibilities listed below does give rise to a rank $3$ affine group$)$.
\begin{itemize}
\item[\rm (A)] Infinite classes. These are
\begin{itemize}
\item[\rm (1)] $X_0\leq \mathrm{\Gamma L}_1(p^d)$: all possibilities for $X_0$ are determined in~\cite[\textsection{3}]{Foulser-Kallaher-solvable-rank-3};
\item[\rm (2)] $X_0$ imprimitive: $X_0$ stabilises a pair $\{V_1,V_2 \}$ of subspaces of $V$ where $V=V_1 \oplus V_2$ and $\dim(V_1)=\dim(V_2)$; moreover, $(X_0)_{V_i}$ is transitive on $V_i\setminus 0$ for $i=1$, $2$ (and hence $ X_0$ is determined by Hering's Theorem~\cite{Hering});
\item[\rm (3)]  tensor product case: for some $a$, $q$ with $q^a= p^d$, consider $V$ as a vector space $\GF(q)^a$ of dimension $a$ over $\GF(q)$; then $X_0$ stabilises a decomposition of $\GF(q)^a$ as a tensor product $V_1 \otimes V_2$ where $\dim_{\GF(q)}(V_1)=2$; moreover, $(X_0)^{V_2}\vartriangleright \SL(V_2)$,  or $(X_0)^{V_2}=\A_7<\SL_4(2)$ $($and $p=q=2$, $d=a=8)$, or $\dim_{\GF(q)}(V_2)\leq 3$;
\item[\rm (4)] $X_0\vartriangleright \SL_a(q) $  and $p^d=q^{2a}$;
\item[\rm (5)] $X_0\vartriangleright \SL_2(q) $  and $p^d=q^{6}$;
\item[\rm (6)] $X_0\vartriangleright \mathrm{SU}_a(q) $  and $p^d=q^{2a}$;
\item[\rm (7)] $X_0\vartriangleright \Omega^{\pm}_{2a}(q)$  $($and if $q$ is odd, $X_0$ contains an automorphism interchanging the two orbits of $\Omega^{\pm}_{2a}(q)$ on nonsingular $1$-spaces$)$;
\item[\rm (8)] $X_0\vartriangleright \SL_5(q) $ and $p^d=q^{10}$ $($from the action of $\SL_5(q)$ on $\bigwedge^2( \GF(q) ^5)$, the   exterior square of $\GF(q)^5)$;

\item[\rm (9)] $X_0/\Z(X_0) \vartriangleright \Omega_7(q).2$ and $p^d=q^{8}$ $($from the action of $B_3(q)$ on a spin module$)$;

\item[\rm (10)] $X_0/\Z(X_0) \vartriangleright \mathrm{P\Omega}_{10}^{+}(q)$ and  $p^d=q^{16}$ $($from the action of $D_5(q)$ on a spin module$)$;

\item[\rm (11)] $X_0  \vartriangleright \mathrm{Sz}(q)$, $q=2^{2m+1}$ and $p^d=q^4$ $($from the embedding $\mathrm{Sz}(q)<\mathrm{Sp}_4(q))$.
\end{itemize}
\item[\rm (B)] `Extraspecial' classes. Here $X_0\leq \N_{\GL_d(p)}(R)$ where $R$ is an $r$-group, irreducible on $V$. Either $r=3$ and $R\cong 3^{1+2}$ $($extraspecial of order $27)$ or $r = 2$ and $|R/\Z(R)| = 2^{2m}$ with $m = 1$ or $2$.
If $r = 2$, then either $|\Z(R)| = 2$ and $R$ is one of the two extraspecial groups $R_1^m,R_2^m$ of order $2^{1+2m}$, or $|\Z(R)|= 4$, when we write $R_3^m$. The possibilities for $(r,p^d,R)$ are
\begin{itemize}
\item[\rm (i)] $r=3$, $p^d=2^6$ and $R=3^{1+2}$;
\item[\rm (ii)] $r=2$, $p^d=3^4$, $7^2$, $13^2$, $17^2$, $19^2$, $23^2$, $3^6$, $29^2$, $31^2$ or $47^2$, and $R=\mathrm{D}_8$ or $\mathrm{Q}_8$;
\item[\rm (iii)] $r=2$, and $(p^d,R)=(3^4,R_{1}^2)$, $(3^4,R_{2}^2)$, $(5^4,R_{2}^2)$, $(5^4,R_{3}^2)$, $(7^4,R_{2}^2)$ or $(3^8,R_{2}^3)$.
\end{itemize}
\item[\rm (C)] `Exceptional' classes.   Here the socle $L$ of $X_0/\Z(X_0)$ is simple, and the possibilities are listed in Table~$\ref{tb:af3C}$.
\end{itemize}
\end{theorem}

\begin{table}[htbp]
\caption{The  possibilities for $L$ in class (C) of Theorem~\ref{subdegree}.} \label{tb:af3C}
\[
\begin{array}{lll|lll}
\hline
L & p^d & \text{embedding of }L &L & p^d & \text{embedding of }L\\
\hline
\A_5 &3^4&\A_5\leq \PSL_2(p^{d/2}) &\A_9&2^8&\A_9<\Omega_8^{+}(2)\\
&31^2& &\A_{10}&2^8&\A_9<\mathrm{Sp}_8 (2)\\
&41^2& &\PSL_2(17)&2^8&\PSL_2(17)<\mathrm{Sp}_8 (2)\\
&7^4& &\PSL_3(4)&3^6&\PSL_3(4)<\mathrm{P\Omega}_6^{-}(3)\\
&71^2& &\mathrm{PSU}_4(2)&7^4&\mathrm{PSU}_4(2)<\PSL_4(7)\\
&79^2& &\mathrm{M}_{11}&3^5&\mathrm{M}_{11}<\PSL_5(3)\\
&89^2& &\mathrm{M}_{24}&2^{11}&\mathrm{M}_{24}<\PSL_{11}(2)\\
\A_6&2^6& \A_6\leq \PSL_3(4)&\mathrm{Suz}&3^{12}&\mathrm{Suz}<\mathrm{PSp}_{12}(3)\\
\A_6&5^4&\A_6\leq \PSp_4(5) &G_2(4)&3^{12}&G_2(4)<\mathrm{Suz}<\mathrm{PSp}_{12}(3) \\
\A_7&2^8& \A_7\leq \PSL_7(4)&\mathrm{J}_2&2^{12}&\mathrm{J}_2<G_2(4)<\mathrm{Sp}_6(4)\\
\A_7&7^4&\A_7\leq \PSp_4(7) &\mathrm{J}_2&5^{6}&\mathrm{J}_2<\mathrm{PSp}_6(5)\\
\hline
\end{array}
\]
\end{table}

For convenience, we label (1)--(11) in class (A) of  Theorem~\ref{subdegree} by (A1)--(A11) in the following arguments.
For a positive integer $n$ and a prime $\ell$, use $n_\ell$ to denote the $\ell$-part of $n$, that is the largest power of $\ell$ that divides $n$.

\subsection{Type (A1)}

In this case, $X_0=G_0\times\l z\r$ should be a subgroup of $\GammaL_1( p^d)$.
The following lemma shows that this is not the case.

\begin{lemma}\label{GL(1,q)}
The group $X_0$ is not a subgroup of $\GammaL_1(p^d)$.
\end{lemma}
\proof
Assume that $X_0\leqslant\GammaL_1(p^d)$. Since $\GL_1(p^d)\cong \ZZ_{p^d-1}$ is cyclic, it implies that $z$ is its unique involution. Therefore, from $z \notin G_0$ we conclude $|G_0 \cap \GL_1(p^d)|_2=1$. Since
 \[
  G_0/G_0 \cap \mathrm{GL}_1(p^d) \cong G_0\mathrm{GL}_1(p^d)/\mathrm{GL}_1(p^d)\leq        \mathrm{\Gamma L}_1(p^d)/\mathrm{GL}_1(p^d)\cong \mathbb{Z}_d,
  \]
It follows that  $|G_0|_2 \leq  |d|_2$. If $d$ is odd, then  $G_0$ is of odd order and so is $G$,  which contradicts that $G$ has a self-paired orbital. Therefore, $d$ is even.

Let $d=2^ed'$ where $e\geqslant1$ and $d'$ is odd. Then $p^d-1$ is divisible by $2^{e+2}$. Let $H$ be a Hall $2'$-subgroup of $\GammaL_1(p^d)$. By Lemma~\ref{Hall-2'-GammaL_1(q)}, we have $H\unlhd\GammaL_1(p^d)$, and since $\GammaL_1(p^d)$ is transitive on $V\setminus\{0\}$, all orbits of $H$ on $V\setminus\{0\}$ have the same size, say $\ell$. Let $m$ be the number of orbits of $H$ on $V\setminus\{0\}$. Then $\ell$ is an odd divisor of ${p^d-1}$, and thus $\ell$ divides $(p^d-1)/ 2^{e+2}$. Since $m\ell=p^d-1$, we have $m\geqslant 2^{e+2}$. The number of orbits of $G_0\cap H$ on $V\setminus\{0\}$ is at least $2^{e+2}$. Recall that $|G_0|_2 \leq  |d|_2=2^e$. Hence,
\[ |G_0/(G_0\cap H)| =|(G_0H)/H| =|G_0|_2 \leq 2^e.\]
It follows that $G_0$ has at least $4$ orbits on $V\setminus\{0\}$, which is a contradiction.\hfill\qed

\subsection{Type (A2)}

In this case, $X_0$ preserves a direct sum of the vector space $V$. We shall prove that this is impossible.

\begin{lemma}\label{direct-sum}
The group $X_0=G_0\times\lg z\rg$ does not preserve a direct sum of the vector space $V$.
\end{lemma}

\begin{proof} Suppose on the contrary that there exists a direct decomposition $V=V_1\oplus V_2$ with $\dim(V_1)=\dim(V_2)=n$ such that $X_0$ acts transitively on $\{V_1, V_2\}$. Let $K$ be the subgroup of $X_0$ fixing $V_1$ and $V_2$ setwise, and let $\s\in X_0$ interchange $V_1$ and $V_2$. Then $X_0=\l K, \s\r$. Clearly, $|X_0|/|K|=2$ and $\s^2\in K$. In particular, the central involution $z$ fixes $V_1$ and $V_2$, and lies in $K$.

Pick $u\in V_1\setminus\{0\}$ and $w\in V_2\setminus\{0\}$.
Then $u$ and $u+w$ lie in different orbits of $X_0$, and hence the two orbits of $X_0$ on $V\setminus\{0\}$ are
\[\Del_1=(V_1\setminus\{0\})\cup(V_2\setminus\{0\}),\ \Del_2=V\setminus(V_1\cup V_2).\]
We remark that the corresponding orbital graphs are $\mathrm{H}(2, p^n)=\K_{p^n}\square\K_{p^n}$ and its complement, respectively.

Since $G_0$ has three orbits on $V\setminus\{0\}$, one of $\Del_1$ and $\Del_2$ is split into two orbits, and the other is an orbit of $G_0$.

Suppose that $G_0$ has exactly two orbits on $\Del_1=(V_1\setminus\{0\})\cup(V_2\setminus\{0\})$, say $\Sig_1$ and $\Sig_2$.
Then $\Sig_2=\{-v \mid v\in\Sig_1\}$, and $z$ fuses $\Sig_1$ and $\Sig_2$.
It follows that for $i=1$ or $2$, a vector $v_i\in V_i$ is not equivalent to $-v_i$ under $G_0$.
Thus $v_1+v_2$, $v_1-v_2$, $-v_1+v_2$ and $-v_1-v_2$ are pairwise inequivalent under the action of $G_0\cap K$.
So $G_0$ has four orbits on $V\setminus\{0\}$, which is a contradiction.

Thus, $G_0$ has exactly two orbits, say $\Sig_1$ and $\Sig_2=-\Sig_1$, on $\Del_2=V\setminus(V_1\cup V_2)$.
Now $G_0$ is transitive on $\Del_1=(V_1\setminus\{0\})\cup(V_2\setminus\{0\})$.
Recall $X_0=G_0\times \langle z\rangle=K\langle \sigma\rangle$ and $z\in K$.
Thus, $G_0\cap K$ is of index $2$ in $ G_0$. We first prove the following claim.

\medskip
\noindent{\bf Claim~1.}\  For $i\in\{1,2\}$, let $z_i$ be the unique involution in $\Z(\GL(V_i))$. Then $(G_0\cap K )\cap\lg z_1,z_2\rg=1$.

Suppose that $z_1\in G_0\cap K$. Take $v_1+v_2\in \Sig_1$ with $v_1\in V_1\setminus\{0\}$ and $v_2\in V_2\setminus\{0\}$.
Then $-v_1+v_2=(v_1+v_2)^{z_1}\in\Sig_1$.
Since $G_0$ is transitive on $\Del_1=(V_1\setminus\{0\})\cup(V_2\setminus\{0\})$, there exists $g\in G_0$ such that $v_2=v_1^g$.
Then $g$ swaps $V_1$ and $V_2$, and so $v_2^g\in V_1$.
Let $w_1=v_2^g$.
Then $v_2+w_1=(v_1+v_2)^g\in\Sig_1$, and so $-v_2-w_1\in \Sig_2$.
It follows that $-v_2+w_1=(-v_2-w_1)^{z_1}\in\Sig_2$, and hence $(-v_1+v_2)^g=-v_1^g+v_2^g=-v_2+w_1\in\Sig_2$
As each $\Sig_i$ is an orbit of $G_0$, we obtain that $-v_1+v_2\in \Sig_2$, but this is impossible because we have already shown that $-v_1+v_2\in \Sig_1$. Thus, $z_1\notin K$. Similarly, we have $z_2\notin K$, and Claim~1 follows.
 \smallskip

Recall that $|V|=p^d$, and $V=V_1\oplus V_2$ and  $\dim(V_1)=\dim(V_2)=n$.
By~\cite[Theorem~3.3]{BDFP2021}, $X_0\leq \mathrm{\Gamma L}_m(q) \wr \mathrm{S}_2$ where $q^m=p^n$ and one of the following holds:
\begin{itemize}
\item[\rm (0)] $X_0\leq \mathrm{\Gamma L}_1(q)\wr \mathrm{S}_2$ and $m=1$;
\item[\rm (1)]  $\mathrm{SL}_m(q) \times \mathrm{SL}_m(q) \unlhd  X_0$ where $m\geq 2$ and $(m,q)\neq (2,2),(2,3)$;
\item[\rm (2)]  $\mathrm{Sp}_m(q) \times \mathrm{Sp}_m(q) \unlhd  X_0$ where $m$ is even and  $m\geq 4$;
\item[\rm (3)]  $G_2(q)' \times G_2(q)'\unlhd  X_0$ where $m=6$  and  $q$ is even;
\item[\rm (4)]  $S \times S \unlhd  X_0 \leq \mathbf{N}_{\mathrm{\Gamma L}_m(q)}(S)$ where $S=\SL_2(5)$, $m=2$  and  $q \in \{9, 11, 19, 29, 59\}$;
\item[\rm (5)]  $X_0=\mathrm{A}_6\wr \S_2$ or $\mathrm{A}_7\wr \S_2$  where $m = 4$ and $q = 2$;
\item[\rm (6)] $X_0=\mathrm{SL}_2(13)\wr \S_2$ where $m = 6$ and $q = 3$;
\item[\rm (7)] $X_0\leq \mathbf{N}_{\mathrm{GL}_m(q)}(\mathrm{SL}_2(3)) \wr \mathrm{S}_2$ where $m = 2$ and $q\in \{3, 5, 7, 11, 23\}$;
\item[\rm (8)] $E\times E \unlhd X_0 \leq \mathbf{N}_{\mathrm{GL}_m(q)}(E) \wr \mathrm{S}_2$ where $E=\mathrm{Q}_8 \circ \mathrm{D}_8 $, $m = 4$ and $q = 3$.
\end{itemize}

Since $p$ is odd, the cases (3) and (5) are impossible.
If  one of the cases (2), (4) or (6) happens, then $z \in  X_0^{(\infty)}$. However, this is impossible because $X_0^{(\infty)}\leq  G_0$.
For cases (7) and (8),  by computation in Magma~\cite{BCP}, we obtain all candidates for $X_0$, and further computation shows that $X_0$ has no  subgroup $Y$ with index $2$ such that $z \notin Y$, and $Y$ has three orbits on $V \setminus \{0\}$ of lengths $2(q^m-1)$, $(q^m-1)^2/2$ and $(q^m-1)^2/2$.

Now suppose that the case (0) happens. Then $G_0\cap K \leq \mathrm{\Gamma L}_1(q)^2  $ and $(G_0\cap K)^{V_i} \leq \mathrm{\Gamma L}_1(q)=\mathrm{GL}_1(q){:}\ZZ_{n}$ for every $i\in\{1,2\}$.
By Claim~1, $z_i \notin (G_0\cap K)^{V_i}$ and so the order of $ (G_0\cap K)^{V_i} \cap \mathrm{GL}_1(q) $ is odd.
If $n$ is odd, then $ (G_0\cap K)^{V_i}$ has odd order and so has at least two orbits on $V_i\setminus \{0\}$, which implies that $G_0\cap K$ has four orbits on $\Del_1=(V_1\setminus\{0\})\cup(V_2\setminus\{0\})$, contradicting the transitivity of $G_0$  on $\Del_1$.
Therefore, $n$ is even.  Write $n=2^e n'$ for some positive integer $e$ and odd integer $n'$.
Since $n$ is even, we conclude that $p^n-1$ is divisible by $2^{e+2}$.
By the transitivity of $G_0$ on $\Del_1$ and $|\Del_1|=(q-1)^2= (p^n-1)^2$, we have $|G_0|_2\geq |\Del_1|_2 \geq 2^{2(e+2)}$.
However, by Claim~1 we see $(G_0\cap K) \cap \GL_1(q)^2=1$ and so $|G_0|_2=2|G_0\cap K|_2=2(n_2)^2=2^{2e+1}$,  a contradiction.

Finally, suppose that case (1) occurs. Then $\mathrm{SL}_m(q) \times \mathrm{SL}_m(q) \leq X_0^{(\infty)} \leq  G_0$ and $m\geq 2$.
Notice that  $\mathrm{SL}_m(q)$ is transitive on $V_i\setminus \{0\}$ for every $i \in \{1,2\}$ (see \cite[Table~7.3]{Cameron-Book}).
It follows that $G_0$ is transitive on $\Del_2=V\setminus(V_1\cup V_2)$, contradicting our assumption.
\hfill\qed
\end{proof}

\subsection{Type (A3)}

In this case, the group $X_0 $ preserves a tensor decomposition of the vector space $V$. We will prove this is not the case.

\begin{lemma}
The group $X_0=G_0\times\lg z\rg$ does not preserve a tensor decomposition of the vector space $V$.
\end{lemma}

\begin{proof}
Suppose on the contrary that $X_0$ stabilises a tensor decomposition of $V$.
From Theorem~\ref{subdegree}~(A)(3) we see  that  $V$ is viewed as an $a$-dimensional vector space over $\GF(q)$ such that $q^a=p^d$, and  $X_0$ stabilises a tensor decomposition
\[V=V_1\otimes V_2,\]
such that $\dim(V_1)=2$, and $\dim(V_2) \geq 2$, and either $X_0^{V_2}\rhd\SL(V_2)$, or $\dim(V_2)=2$ or 3.
Let $\dim(V_2)=b$. Then $a=2b$.
Moreover, by \cite[Lemma~1.1]{Liebeck-affine}, the two orbits of $X_0$ on $V\setminus\{0\}$ are as follows,
for a basis $\{v_1,v_2\}$ of $V_1$,
$$\begin{array}{lll}
\Ome_1=\{v_1\otimes w+v_2\otimes (\lambda w)\mid w\in V_2, \lambda\in\GF(q)\},\\
\Ome_2=\{v_1\otimes w_1+v_2\otimes w_2 \mid w_1,w_2\in V_2,\, \dim\lg w_1,w_2\rg=2\rg\}.
 \end{array}
$$
Clearly, the orbit $\Ome_1$ can be rewritten as $\Ome_1=\{v\otimes w\mid v\in V_1,w\in V_2\}$.
By \cite[Table~12]{Liebeck-affine}, we have
\[|\Ome_1|=(q+1)(q^b-1),\, |\Ome_2|=q(q^b-1)(q^{b-1}-1).\]

 By~\cite[Theorem~3.2]{BDFP2021}, one of the following holds:
\begin{itemize}
\item[\rm (R0)] $X_0 \leq \mathrm{A\Gamma L}_1(p^d)$;
\item[\rm (I0)] $X_0$ stabilises a direct decomposition of $V$;
\item[\rm (T1)] $\mathrm{SL}_2(q) \otimes \mathrm{SL}_b(q) \unlhd   X_0$ where $q \geq 4$ or $(b, q) = (2, 3)$;
\item[\rm (T2)] $\mathrm{SL}_2(5) \otimes \mathrm{SL}_b(q) \unlhd   X_0$ where $q \in\{ 9, 11, 19, 29, 59\}$;
\item[\rm (T3)] $1 \otimes \mathrm{SL}_b(q)  \unlhd   X_0 \leq  \mathbf{N}_{  \mathrm{GL}_2(q)}(\mathrm{SL}_2(3)) \otimes \mathrm{GL}_b(q)$ where $q \in\{ 3, 5, 7, 11, 23\}$  but $X_0 \not\leq \mathrm{\Gamma L}_1(q^2) \times \mathrm{GL}_b(q)$, and further, $-1 \otimes g \in X_0$ for some  $g \in \mathrm{GL}_b(q)$;
\item[\rm (T4)] $1\otimes \mathrm{A}_7 \unlhd X_0 $ and $(b,q)=(4,2)$;
\item[\rm (T5)] $1\otimes \mathrm{SL}_b(q) \unlhd X_0  \leq (\mathrm{GL}_1(q^2)\otimes \mathrm{GL}_b(q)){:}\langle (1\otimes t )\sigma_q\rangle $  for some $t\in \mathrm{GL}_2(q)$ where $\mathrm{\Gamma L}_1(q^2)=\mathrm{GL}_1(q^2){:}\langle (1\otimes t )\sigma_q\rangle$;
\item[\rm (T6)]$\mathrm{SL}_2(q)\otimes 1 \unlhd X_0  \leq (\mathrm{GL}_2(q)\otimes \mathrm{GL}_1(q^3)){:}\langle (1\otimes t )\sigma_q\rangle $  for some $t\in \mathrm{GL}_b(q)$ where $b=3$ and $\mathrm{\Gamma L}_1(q^3)=\mathrm{GL}_1(q^3){:}\langle (1\otimes t )\sigma_q\rangle$.
\end{itemize}

By Lemmas~\ref{GL(1,q)} and \ref{direct-sum} we know that the cases (R0) and (I0) are impossible.
Since $p$ is odd, we do not need to consider the case (T4).
Notice that $z=(-1) \otimes 1=1 \otimes (-1)  \notin G_0$. By Lemma~\ref{HA-rank4-pty}, we have $X_0^{(\infty)} \leq G_0$ and so $z\notin X_0^{(\infty)}$. In cases (T1) or (T6), if $X_0$ is non-solvable, then $z=(-1) \otimes 1 \in \SL_2(q)\otimes 1 \leq X_0^{(\infty)} \leq G_0$, a contradiction. Similarly, if the case (T2) happens, then $z=(-1) \otimes 1 \in \SL_2(5)\otimes 1 \leq X_0^{(\infty)} \leq G_0$, a contradiction.
Now it remains to consider one of the following cases: (T1) with $(b,q)=(2,3)$,  (T3), (T5), and (T6) with $q=3$.

Suppose first that the case (T1) with $(b,q)=(2,3)$ happens. Computation in Magma~\cite{BCP} shows that $\mathbf{N}_{\GL_4(3)}(\SL_2(3) \otimes \SL_2(3))$ has rank $3$ on $V$, and there are four candidates for $X_0$  between $\SL_2(3)\otimes \SL_2(3)$ and  $\mathbf{N}_{\GL_4(3)}(\SL_2(3) \otimes \SL_2(3))$, and any possible $X_0$ has no subgroup $Y$ of index $2$ such that  $Y$ has rank $4$ on $V$ and $z \notin Y$.
Therefore, this case is impossible.

Suppose next that the case (T3) happens. Since $z=1\otimes (-1) \notin G_0$, we conclude that $b$  is odd and $b\geq 3$.
Then $\SL(V_2)=1\otimes \SL_b(q) \leq X_0^{(\infty)} \leq G_0$.
Since $\dim(V_2)=b\geq 3$, we may pick two linearly independent vectors $w_1, w_2$ of $V_2$ and $g\in\SL(V_2)$ such that $w_1^g=-w_1$ and $w_2^g=-w_2$. Let $u_1=v_1\otimes w_1$ and $u_2=v_1\otimes w_1+v_2\otimes w_2$. Then $\Omega_1=u_1^{X_0}$ and $\Omega_2=u_2^{X_0}$, and
\[
\begin{array}{ll}
u_1^g &=v_1^g\otimes w_1^g=v_1\otimes (-w_1)=-u_1, \\
u_2^g &=(v_1\otimes w_1+v_2\otimes w_2)^g=-(v_1\otimes w_1+v_2\otimes w_2)=-u_2.
\end{array}
\]
This is impossible because either $\Ome_1$ or $\Ome_2$ is split into two non-self-paired  orbits  $\Sig_1$ and $\Sig_2 (=-\Sig_1)$ of $G_0$.

Now suppose that the case (T5) happens.
If $(b,q)=(2,3)$, then
\[X_0\leq \mathbf{N}_{\mathrm{GL}_4(3)}(\mathrm{GL}_1(3^2)\otimes \mathrm{GL}_2(3)),\] and  computation in Magma shows that  any possible $X_0$ has no subgroup $Y$ of index $2$ such that  $Y$ has rank $4$ on $V$ and $z \notin Y$.
Assume that $(b,q)\neq (2,3)$. If $b$ is even, then $z=1\otimes -1 \in  X_0^{(\infty)} \leq G_0 $, a contradiction.
Therefore, $b$ is odd and $b\geq 3$.
With a similar argument as for the case (T3), one can obtain a contradiction.

Finally suppose that the case (T6) with $q=3$ happens. Then \[X_0\leq \mathbf{N}_{\mathrm{GL}_6(3)}(\mathrm{GL}_2(3)\otimes \mathrm{GL}_1(3^3)).\] Again by computation in Magma, we find that any possible $X_0$ has no subgroup $Y$ of index $2$ such that $Y$ has rank $4$ on $V$ and $z \notin Y$. \hfill\qed
\end{proof}

\subsection{Types (A4)--(A11)}
In class (A11), $\mathrm{Sz}(q) \trianglelefteq X_0$ with $|V|=p^r$ is even, a contradiction.
For other cases, from~\cite[p.481]{Liebeck-affine} we see that $X_0\leq \mathrm{\Gamma L}_{n}(q_0)$ for some integer $n\geq 2$ and  $q_0$ such that $q_0^n=p^d$, and some information for $X_0$ is given in Table~\ref{tb:A4-A11}.
\begin{table}[htbp]
\caption{Classes (A4)--(A11)}\label{tb:A4-A11}
\[
\begin{array}{llll }
\hline
\text{Class} & X_0 & X_0^{(\infty)}   & \text{Condition} \\
\hline
\text{(A4)} & \mathrm{SL}_a(q) \trianglelefteq X_0
  & \mathrm{SL}_a(q)\text{ or }1 &n=a,q_0=q^2  \\
\text{(A5)} & \mathrm{SL}_2(q) \trianglelefteq X_0
 & \mathrm{SL}_2(q)\text{ or }1  & n=2,q_0=q^3 \\
\text{(A6)} & \mathrm{SU}_a(q)\trianglelefteq X_0
 & \mathrm{SU}_a(q) \text{ or }1  & n=a,q_0=q^2 \\
\text{(A7)} & \Omega^{\pm}_{2a}(q).2 \trianglelefteq X_0
& \Omega^{\pm}_{2a}(q)  \text{ or }1 &  n=2a    \\
 \text{(A8)}  & \mathrm{SL}_5(q) \trianglelefteq X_0 & \mathrm{SL}_5(q)
 &  n=10  \\
  \text{(A9)} &\mathrm{P\Omega}_7(q).2 \trianglelefteq X_0/\mathbf{Z}( X_0)
 &\mathrm{Spin}_7(q) &  n=8 \\
   \text{(A10)} & \mathrm{P\Omega}_{10}^{+}(q)  \trianglelefteq X_0/\mathbf{Z}( X_0)
 & \mathrm{Spin}_{10}(q) &  n=16 \\
 \hline
\end{array}
\]
\end{table}

Suppose first that $X_0^{(\infty)} =1$, that is, $X_0$ is solvable. Then one of the following occurs: (A4) with $(a,q)\in \{(2,2),(2,3)\}$, or (A5) with $q\in\{2, 3\}$, or (A6) with $(a,q)\in \{(2,2),(2,3)\}$, or (A7) with $a=1$ or $(a,q)\in \{(2,2),(2,3) \}$. In case (A7) with $a=2$, noticing that $\Omega^{\pm}_2(q) <\SL_2(q)$, from~\cite[Table~8.1]{Bray-Holt-Dougal} we see that $X_0\leq \mathrm{Q}_{2(q\pm 1)}$. Since $G$ has a self-paired suborbit, we see that $|G|$ is even, which implies that  $|G_0|=|G|/|V|$ is even. Then from the fact that $z$ is the unique involution of the group $\mathrm{Q}_{2(q\pm 1)}$,  we conclude that $z\in G_0$, a contradiction. For the other cases, we can obtain all such groups $X_0$ by Magma~\cite{BCP}, and moreover, we find that every $X_0$ has no subgroup  $G_0$  such that $G_0$ is irreducible on $V$ and has $4$ orbits, and $X_0=G_0\times\lg z\rg$.

We therefore conclude that $X_0^{(\infty)} \neq 1$, that is, $X_0$ is non-solvable.
By Lemma~\ref{HA-rank4-pty}, $p$ is odd, $X_0^{(\infty)} \leq G_0$ and $z \notin X_0^{(\infty)}$. For (A5), (A7), (A9) or (A10), since $X_0^{(\infty)}$ is irreducible (see \cite[p.481]{Liebeck-affine}) and the order of $ \mathbf{Z}(X_0^{(\infty)})$ is divisible by $2$, we obtain that $z\in \mathbf{Z}(X_0^{(\infty)})$, a contradiction.
For (A6), from~\cite[p.483]{Liebeck-affine} we see that $X_0^{(\infty)}=\mathrm{SU}_a(q)$ has two orbits on $V\setminus \{ 0\}$, which implies that $G_0$ has two orbits on $V\setminus \{ 0\}$, a contradiction.

We now consider the class (A4). Following~\cite[p.482, Case(IIc)]{Liebeck-affine}, we pick $\sigma \in \GF(q^2) \setminus \GF(q)  $, and let  $ \{ e_1,e_2,\dots,e_a\}$ be a $\GF(q^2)$-basis for $V$  with respect to which the subgroup  $\mathrm{SL}_a(q)$ of $X_0$  { acts} naturally over $\GF(q) $, and let $v_1 = e_1$ and $v_2=e_1+\sigma e_2$.
By~\cite[p.483]{Liebeck-affine},  $v_1$ and $v_2$ are in different orbits of $X_0$ on $V\setminus\{0\}$. Since $z \notin \mathbf{Z}(X_0^{(\infty)})$, it follows that $a\geq 3$ is odd.
Let $g \in \mathrm{SL}_a(q)$ be such that
\[e_1^g=-e_1,\, e_2^g=-e_2,\, e_j^g=e_j \text{ for all }j\in \{3,\dots,a\}.\]
 Then $v_1^g=-v_1$ and $v_2^g=-v_2$. This implies that neither $v_1^{X_0}$ nor $v_2^{X_0}$ splits into the two  $G_0$-orbits $\Sig_1$ and $\Sig_2$ with $\Sig_2=-\Sig_1$, a contradiction.

Finally, we deal with the class (A8). In this case, we have $V=\bigwedge^2( \GF(q) ^5)$.
Following~\cite[Lemma~2.5]{Liebeck-affine}, we let $\{z_1,z_2,z_3,z_4,z_5 \}$ be a basis for $\GF(q)^5$.
By~\cite[p.489]{Liebeck-affine},  $v_1:=z_1 \wedge z_2$ and $v_2:=z_1 \wedge z_2+z_3 \wedge z_4$ are in different orbits of $X_0$ on $V\setminus\{0\}$. Let $g\in \mathrm{SL}_5(q)$ be such that \[ z_1^g=z_2,\, z_2^g=z_1,\, z_3^g=z_4,\, z_4^g=z_3,\, z_5^g=z_5. \]
Then  $v_1^g=z_1^g \wedge z_2^g=z_2 \wedge z_1=-(z_1 \wedge z_2)=-v_1$, and similarly, $v_2^g=-v_2$.
Since $\mathrm{SL}_5(q)=X_0^{(\infty)}\leq G_0$, neither $v_1^{X_0}$ nor $v_2^{X_0}$ splits into the two  $G_0$-orbits $\Sig_1$ and $\Sig_2=-\Sig_1$.

\subsection{Class (B)}

This class consists of finitely many candidates appeared as normalisers of an extraspecial
group of prime-power order, see Theorem~\ref{subdegree} (B), namely,
$X_0\leqslant \N_{{\rm \Gamma L}_{a}(q)}(R)$, where $R$ is an extraspecial $2$- or $3$-group. If $R$ is a $3$-group, then $p^d=2^6$, which contradicts that $|V|=p^d$ is odd.
Assume that $R$ is a $2$-group.
If $p^d\neq 3^8$,  then by Magma~\cite{BCP}, we can obtain all such $2$-groups $R$ and all subgroups $X_0\leq\N_{{\rm \Gamma L}_{a}(q)}(R)$. Furthermore, we find that $X_0$ has no subgroup $G_0$ such that $G_0$ is irreducible on $V$ and has $4$ orbits, and $X_0=G_0\times\lg z\rg$.
If $p^d=3^8$, then  we have $\N_{{\GL}_8( 3)}(R)<{\rm GSp}_8(3)<\GL_8(3)$ (see \cite[p.485, Case $R=R_2^3$]{Liebeck-affine}).
By Magma~\cite{BCP}, $\N_{\GL_8( 3)}(R)$ is maximal in ${\rm GSp}_8(3)$, and furthermore, we obtain the following information about the maximal subgroups of $\N_{\GL_8( 3)}(R)$:
\begin{itemize}
\item $\N_{\GL_8( 3)}(R)$ has a maximal subgroup, say $M$, which has $3$ orbits on $V$;
\item $\N_{\GL_8( 3)}(R)$ has a maximal subgroup, say $N$, which has $4$ orbits on $V$;
\item For any maximal subgroup $L$ of $\N_{\GL_8( 3)}(R)$, if $L$ is not conjugate to $M$ or $N$, then $L$ has more than $4$ orbits on $V$;
\item $M$ has a maximal subgroup, say $M_1$, which has $4$ orbits on $V$, and any maximal subgroup of $M$ which is not conjugate to $M_1$ has more than $4$ orbits on $V$;
\item $z\in N, M_1$.
\end{itemize}
So $X_0=\N_{\GL_8( 3)}(R)$ or $M$. If $X_0=\N_{\GL_8( 3)}(R)$, then we must have $G_0=N$, which is impossible because $z\in N$ but we require that $z\notin G_0$. If $X_0=M$, then we must have $G_0=M_1$, which is again impossible because $z\in M_1$ but we require that $z\notin G_0$.

\subsection{Class (C)}

Finally, consider the candidates in Theorem~\ref{subdegree} (C). In this case, the socle $L$ of $X_0/\Z(X_0)$ is simple and is listed in Table~\ref{tb:af3C}. If $p^d\neq 3^{12}$, then by Magma~\cite{BCP}, we can obtain all such groups $X_0$, and moreover, we find that every $X_0$ has no subgroup $G_0$ such that $G_0$ is irreducible on $V$ and has $4$ orbits, and $X_0=G_0\times\lg z\rg$.

If $p^d=3^{12}$, then $L= \mathrm{Suz}$ or $G_2(4)$. Note that the (projective) representation of $L$ on $V=\mz_3^{12}$ is absolutely irreducible (see \cite[p.481]{Liebeck-affine}). By \cite[Theorem~4.3.3]{Bray-Holt-Dougal}, $X_0^{(\infty)}$ is given in \cite[Table~4.4]{Bray-Holt-Dougal}. By inspecting \cite[Table~4.4]{Bray-Holt-Dougal}, in case $L=\mathrm{Suz}$, we have $X_0^{(\infty)}=2^\cdot \mathrm{Suz}$, and in case $L=G_2(4)$, we have $X_0^{(\infty)}=2^\cdot G_2(4)$.
For both cases, we have $z\in \mathbf{Z}(X_0^{(\infty)})\leq G_0$, a contradiction.

\section{Almost simple primitive permutation groups of rank 4}\label{subsec:almostsimple-rank4}

In this section, we deal with almost simple groups (in Corollary~\ref{reduction}~(1)). In particular, we will prove the following theorem, which together with Theorem~\ref{affine-rank-4} completes the proof of Theorems~\ref{th:2sethomo} and \ref{trans-rank-4}.

\begin{theorem}\label{rank-4}
Let $G$ be an almost simple primitive rank $4$ permutation group on a set $\Om$ with a unique non-self-paired orbital.
Then $G=\PSU_3(3)$ with stabiliser $\PSL_3(2)$, and the subdegrees are $1$, $7$, $7$ and $21$. Moreover, $\mathbf{N}_{\mathrm{Sym}(\mathit{\Omega})}(G)=\mathrm{P\Gamma U}_3(3)$ is a primitive rank $3$ group with subdegrees $1$, $14$ and $21$.
\end{theorem}

Here we consider almost simple primitive rank $4$ permutation groups $G$ with exactly one non-trivial self-paired orbital. For the almost simple sporadic case, by \cite[Chapter~4]{Praeger-Soicher}, we have the following result (one can also obtain this by using \cite[Theorem~1.1]{Muzychuk-Spiga} and Magma~\cite{BCP}).

\begin{proposition}\label{prop-sporadic}
Every non-trivial orbital of a primitive rank $4$ permutation group with socle a sporadic simple group is self-paired.
\end{proposition}

For the case where $\soc(G)$ is alternating, by \cite{Bannai}, we have the following result.

\begin{proposition}\label{prop-alt}
Let $G$ be an almost simple primitive rank $4$ permutation group on $\Om$. If $\soc(G)=\A_n(n\geq 2)$, then every non-trivial orbital of $G$ is self-paired.
\end{proposition}
\proof Let $M=\soc(G)$ and let $M_\a$ be the stabiliser of $\a\in\Om$ in $M$. If $n=6$, then computation by Magma~\cite{BCP} shows that $G$ has rank at most $3$, a contradiction. So we may assume that $n\neq6$. Set $\Del =\{1,2,\ldots,n\}$. Then $G=\A_n$ or $\S_n$. By inspecting \cite[Tables~1--2]{Bannai}, we can see that one of the following holds:
\begin{enumerate}
  \item [{\rm (1)}] $n\geq 7$, $\Ome=\{\o \mid \o\subseteq \Del, |\o|=3\}$, $|M_\a|=3\cdot(n-3)!$ and $|\Ome|=\frac{n(n-1)(n-2)}{6}$;
  \item [{\rm (2)}] $n=12$ or $14$, $M_\a$ fixes the pair of subsets $\{\o_1,\o_2\}$ such that $\o_1=\{1,\ldots,n/2\}$ and $\o_2=\{n/2,\ldots,n\}$, and $|\Om|=\frac{n!}{2\cdot (n/2)!}$;
  \item [{\rm (3)}] $n=12$, $M=G=\A_{12}$, $M_\a\cong \M_{12}$, and $|\Om|=2520$.
\end{enumerate}
If (1) happens, then the suborbits of $M_\a$ are:
\[
\begin{array}{l}
\Del_i=\{\o \mid \o\in\Ome, |\o\cap\a|=i\} \text{ for all } i \in \{1,2,3\}.\\
\end{array}
\]
It is easy to see that all these suborbits are self-paired. By a direct calculation, we also have $|\Del_3|=1$, $|\Del_2|=3(n-3)$,  $|\Del_1|=3\binom{n-3}{2} $ and $|\Del_0|=\binom{n-3}{3} $.

Assume that (2) happens. By Magma~\cite{BCP}, we know that the subdegrees of $M$ are $1,36,200,225$ for $n=12$ and are $1,49,441, 1225$ for $n=14$.
If (3) happens, then again by Magma~\cite{BCP}, the subdegrees of $M$ are $1,440,495,1584$. Now one may easily see that all the non-trivial orbitals of the groups in (2)--(3) are self-paired.\hfill\qed

To complete the proof of Theorem~\ref{rank-4}, we only need to consider the case of groups of Lie type.
We first cite a lemma concerning the orbitals of simple groups of Lie type acting on a set with stabiliser a maximal parabolic subgroup.

\begin{lemma}[{\rm \cite[Lemma~2.13]{YFX}}]\label{lm:parabolic}
Let $T$ be a simple group of Lie type acting on a set $\Delta$ with stabiliser a maximal parabolic subgroup, and let $v\in\Delta$. If $T$ has a non-self-paired orbital on $\Delta$, then one of the following holds:
\begin{enumerate}
\item[\rm (1)] $T=\mathrm{P\Omega}_{2n}^{+}(q)$ with $n\geq 5$, and $T_v$ is of type $A_iD_{n-1-i}$ with $i \in \{\lfloor n/2\rfloor,\ldots,n-3 \}$, so that $T_v$ is the stabiliser of a totally singular $k$-subspace with $i \in \{\lfloor n/2\rfloor+1,\ldots,n-2 \}$;
\item[\rm (2)] $T=F_4(q)$, and $T_v$ is of type $A_1A_2$;
\item[\rm (3)] $T={}^2E_6(q)$, and $T_v=[q^{29}].(\PSL_3(q^2)\times \SL_2(q)){:}(q-1))$ or $[q^{31}].(\SL_3(q)\times \SL_2(q^2)){:}(q^2-1)/(3,q+1)$;
\item[\rm (4)]  $T=E_6(q)$, and $T_v$ is of type $A_1A_4$ or $A_1A_2A_2$;
\item[\rm (5)]  $T=E_7(q)$, and $T_v$ is of type $A_1A_5$, $A_1A_2A_3$, $A_2A_4$ or $A_1D_5$;
\item[\rm (6)]  $T=E_8(q)$, and $T_v$ is of type $A_7$, $A_1A_6$, $A_1A_2A_4$, $A_3A_4$, $A_2D_5$ or $A_1E_6$.
\end{enumerate}
\end{lemma}

The next lemma is about the orbitals of the derived subgroup of the special orthogonal group acting on nonsingular $1$-subspaces.

\begin{lemma}\label{lm:Omeganonsingular}
Let $V$ be an $n$-dimensional vector space over $\GF(q)$ with $n\geq 5$.
Suppose that $q$ is odd if $n$ is odd.
Let $T=\Omega(V)$, the derived subgroup of the special orthogonal group $\mathrm{SO}(V)$, and let $\Del$ be an orbit of $T$ acting on the set of nonsingular $1$-subspaces of $V$. Then all $T$-orbitals on $\Del$ are self-paired.
\end{lemma}
\f\proof
Let $Q$ be the quadratic form on $V$ fixed by $T$, and let $\boldsymbol{\beta}(,)$ be the associated polar form. Take two different nonsingular $1$-subspaces $\langle v\rangle, \langle w\rangle \in \Del$. Since $\Del$ is an orbit of $T$, we must have $Q(v)=Q(w)\neq0$. To show that all $T$-orbitals on $\Del$ are self-paired, we shall find some $g\in T$ swapping $\langle v\rangle$ and $\langle w\rangle$.

Since the dimension of $V$ is at least $5$, by Witt's lemma, $V$ has two vectors $v'$ and $w'$ and an $(n-4)$-dimensional subspace $U$ such that
\[V=\langle v,w\rangle \perp \langle v',w'\rangle \perp U,\, Q(v')=Q(w') =Q(v) \text{ and } \boldsymbol{\beta}(v',w')=\boldsymbol{\beta}(v,w).\]
Then $Q(v'+w')=Q(v+w)$ and $Q(v'-w')=Q(v-w)$.

For a nonsingular element $z$ of $V$, the reflection $r_z{:} V\mapsto V$ is defined by $x^{r_z}=x-\boldsymbol{\beta}(x,z)z/Q(z)$ for all $x\in V$ (see \cite[Definition~1.6.8]{Bray-Holt-Dougal}). Note that all reflections are in the general orthogonal group $\mathrm{GO}(V)$. For arbitrary $x, y\in V$ with $Q(x)=Q(y)$, direct computation shows that if $Q(x-y) \neq 0$ then $x^{r_{(x-y)}}=y$ and  $y^{r_{(x-y)}}=x$, and if $Q(x+y) \neq 0$ then $x^{r_{(x+y)}}=-y$ and  $y^{r_{(x+y)}}=-x$.

Assume  that $q$ is odd. Since $Q(v)=Q(w)$, we deduce  $\boldsymbol{\beta}(v+w,v-w) =0$ and so
\[
4Q(v)=Q(2v) =Q(v+w)+Q(v-w)+\boldsymbol{\beta}(v+w,v-w)=Q(v+w)+Q(v-w).
\]
Since $q$ is odd, it follows that $4Q(v)\neq 0$, and hence  if $Q(v-w)= 0$ then  $Q(v+w) \neq 0$.
Let $g$ be $r_{(v-w)}r_{(v'-w')}$ if $Q(v-w) \neq 0$, otherwise $r_{(v+w)}r_{(v'+w')}$. Clearly, $g$ swaps $\langle v\rangle$ and $\langle w\rangle$. Further, since $g$ is the product of two reflections, its spinor norm is $1$ and hence $g \in T$ (see~\cite[Definitions~1.6.10~and~1.6.13]{Bray-Holt-Dougal}), as required.

Assume that $q$ is even.
By assumption we see that $n$ is even.
If $Q(v+w)\neq 0$, then $g:=r_{(v+w)}r_{(v'+w')}$ is in $T$  (see~\cite[Definitions~1.6.10~and~1.6.13]{Bray-Holt-Dougal}) and swaps $\langle v\rangle$ and $\langle w\rangle$,  as required. Thus assume $Q(v+w)=0$. Then $\boldsymbol{\beta}(v,w)=Q(v+w)-Q(v)-Q(w)=0$ and $\boldsymbol{\beta}(v',w')=0$. Now we let $g \in \GL(V)$  such that
 \[ v^g=w,\, w^g=v,\, (v')^g=w',\, (w')^g=v',\,u^g=u \text{ for all } u \in U. \]
By matrix computation and applying~\cite[Lemma~1.5.21]{Bray-Holt-Dougal}, we derive that $g\in \mathrm{GO}(V)$. Let $I$ be the $n\times n$ identity matrix. It is easy to check that $I-g$ has rank $2$. By~\cite[Proposition~1.6.11]{Bray-Holt-Dougal}  the quasideterminant of $g$ is $1$, and so $g\in T$, as required. \hfill\qed

Now, we  deal with groups of Lie type.

\begin{proposition}\label{prop-sim-lie}
Let $G$ be an almost simple primitive rank $4$ permutation group on $\Om$ and $\a\in\Om$.
If $\soc(G)$ is of Lie type, then either every non-trivial orbital of $G$ is self-paired,
or $G=\PSU_3(3)$ as in Theorem~$\ref{rank-4}$.
\end{proposition}

\f\proof
Suppose that $G$ has at least one non-self-paired orbital.
Then  since  $G$ has rank $4$, we conclude that $G$ has two non-self-paired orbitals of the same subdegree and one non-trivial self-paired orbital.
Let $S=\soc(G)$ and $ S_\alpha=S \cap G_\alpha$.

\medskip
\f{\bf Case~1.}\ The Linear Case.

In this case, if $S_\alpha$ is a maximal parabolic subgroup of $S$, then by Lemma~\ref{lm:parabolic}, all $S$-orbitals are self-paired, and so are $G$-orbitals. Thus, by \cite[Theorem \& Table~1]{Vauhkonen}, we only need to consider the following two subcases:\medskip

\medskip
\f{\bf Subcase 1.1.}\  $\soc(G)=\PSL_3(q)$ and $G$ contains a graph automorphism of $\PSL_3( q)$, and $\Ome$ is the set of pairs $\{U,W\}$ of subspaces of $\GF(q)^3$
satisfying $\dim(U)=1$, $\dim(W)=2$, and $U\leq W$.\medskip

Take $U=\lg e_1\rg$ and $W=\lg e_1,e_2\rg$, define $\a=\{U,W\}$ and
$$
\begin{array}{lll}
\Del_0&=&\{\{U, W\}\};\\
\Del_1&=&\{\{U',W'\} \mid U'=U,\ {\rm or}\ W'=W\}\setminus\Del_0;\\
\Del_2&=&\{\{U',W'\} \mid U'\neq U, W'\neq W,\ {\rm and}\\
&&  {\rm either}\
U\nleq W', U'\leq W,\ {\rm or }\ U\leq W', U'\nleq W \};\\
\Del_3&=&\{\{U',W'\} \mid U\nleq W', U'\nleq W\}.\\
\end{array}
$$
Then the $\Del_i$ are four $G_\a$-orbits on $\Om$ (see \cite[p.185--186]{Vauhkonen}). Clearly, $|\Del_0|=1$.

Take $\b=\{U, W'\}\in\Del_1$ such that $W'\neq W$. Then we may assume that $W'=\{e_1, e_3\}$ such that $e_3\notin W$. This implies that $e_1, e_2, e_3$ form a basis of $\GF(q)^3$. Then there exists an involution $g\in\PSL_3( q)$ such that $g$ sends $\lg e_1\rg, \lg e_2\rg, \lg e_3\rg$ to $\lg e_1\rg, \lg e_3\rg, \lg e_2\rg$, respectively. This implies that $g$ interchanges $\a$ and $\b$, and so $\Del_1$ is self-paired.

Let $U'=\{e_3'\}$ with $e_3'\notin W$, and let $W'=\lg e_2, e_3'\rg$. Then $e_1\notin W'$, and then $\g=\{U', W'\}\in \Del_3$. Again, we have $e_1, e_2, e_3'$ form a basis of $\GF(q)^3$. Then there exists an involution $g\in\PSL_3( q)$ such that $g$ sends $\lg e_1\rg, \lg e_2\rg, \lg e_3'\rg$ to $\lg e_3'\rg, \lg e_2\rg, \lg e_1\rg$, respectively. This implies that $g$ interchanges $\a$ and $\g$, and so $\Del_3$ is self-paired.

Consequently, $\Del_2$ must be also self-paired. \medskip

\f{\bf Subcase 1.2.}\  $G$ is one of the groups in \cite[Table~1]{Vauhkonen} (see Appendix~1).\medskip

Since $G$ has exactly one non-trivial self-paired orbital, by \cite[Table~2]{Vauhkonen}, $G$ is one of the three groups as follows:
\[
\begin{array}{llllc}
\hline
G & G_\alpha & \text{degree} &\text{subdegrees}& \text{rank of }\N_{{\rm Sym}(\Om)}(G) \\
\hline
\PGL_2(7) &\mathrm{D}_{16}&21&1,4,8,8&4 \\
\PSL_2(8) &\mathrm{D}_{18}&28&1,9,9,9&2 \\
\PSL_2(32).5 &\mz_{33}{:}\mz_{10}&496&1,165,165,165&4 \\
\hline
\end{array}
\]

\medskip
For symplectic groups, unitary groups, orthogonal groups and the exceptional groups of Lie type,
we will analyse the classification of primitive groups of rank at most $5$ given by Cuypers in \cite[Table~I]{Cuypers} (also see  Appendix~2),
which consists of 68 cases. When we write ``\underline{Line~$n$}", we mean Line~$n$ in Table~I of \cite{Cuypers}.
We shall proceed down these candidates.

\medskip
\f{\bf Case~2.}\ The Symplectic Case.\medskip

 In this case, $(S,  S_\alpha)$ is one of the pairs in Lines 14--22 of \cite[Table~I]{Cuypers}.

\vspace{0.1cm}
\underline{Lines 14-16}:\   For these lines, $S_\alpha$ is a maximal parabolic subgroup of $S$, and hence all $G$-orbitals are self-paired by Lemma~\ref{lm:parabolic}.

\vspace{0.1cm}
\underline{Line 17}:\
 Now $S=\mathrm{PSp}_n(2)$ with $n\geq 6$ and $ S_\alpha  $ is the stabiliser of a non-degenerate $2$-space.
 By \cite[Lemma~8.5~(i)]{Cuypers}, $G$ has rank $5$.

\vspace{0.1cm}
\underline{Line 18}:\   Now $S=\PSp_4(q)'$ and $S_\alpha$
is the stabiliser of a decomposition of $V$ into a direct sum of two isomorphic non-degenerate $2$-spaces, where $q\in\{2,3,4,5,7,$ $8,$ $9,$ $16,$ $32\}$. If $q$ is even, then $S_\alpha$ is conjugate by a graph automorphism of $S$ to $\mathrm{SO}^{+}_4(q)$, and by~\cite[Theorem~1]{Inglis1990}, all $G$-orbitals are self-paired.
Let $q$ be odd. View $S$  as $\Omega_{5}(q)$. Then $S_\alpha \cong \Omega^{+}_4(q).2$ is conjugate to the stabiliser in $\Omega_{5}(q)$ of some nonsingular $1$-subspace. According to Lemma~\ref{lm:Omeganonsingular},  all $G$-orbitals  are self-paired.

\vspace{0.1cm}
\underline{Line 19}:\
 Now $S=\PSp_4(q)$ and $S_\alpha$
is the stabiliser of an extension field of $\GF(q)$ of order $q^2$, where $q\in\{3,5,7,9\}$. Also, viewing $S$  as $\Omega_{5}(q)$, $S_\alpha  \cong \Omega^{-}_4(q).2 $ is conjugate to the stabiliser in $\Omega_{5}(q)$ of some nonsingular $1$-subspace. According to Lemma~\ref{lm:Omeganonsingular},  all $G$-orbitals are self-paired.

\vspace{0.1cm}
\underline{Line 20}:\  Now $S=\PSp_4(q)$ and $S_\alpha$
is the normaliser of a non-degenerate quadratic form on $V$, where $q\in\{2,4,8,16,32\}$. By~\cite[Theorem~1]{Inglis1990}, all orbitals of $G$ are self-paired.

\vspace{0.1cm}
\underline{Line 21}:\  Now $(S, S_\alpha)=(\PSp_6(q),G_2(q))$ and $q\in \{2,4,8\}$. By \cite{Liebeck-2-closures}, all orbitals of $G$ are self-paired.

\vspace{0.1cm}
\underline{Line 22}:\  In this case, $S=\PSp_4(q)$ with $q$  even, and $G$ contains a graph automorphism, and $S_\alpha$ is the stabiliser of a pair $(U, W)$, where $U\leq W$, $\dim(U)=1$, $\dim(W)=2$ and $W$ is totally singular. Now $S_\alpha$ is a Borel subgroup of $S$. By \cite[Proposition~8.2.1]{Carter}, $S$ has a subgroup $N$ such that $S=S_\alpha N S_\alpha$, $N\cap S_\a\unlhd N$, and $N/(S_\a\cap N)$ is isomorphic to the Weyl group of $S$ which is $D_8$. By \cite[Proposition~8.2.3]{Carter}, the number of double cosets of $S_\a$ is equal to the order of the Weyl group of $S$. This implies that $S$ has rank $8$. Furthermore, since $N/(S_\a\cap N)\cong D_8$, it follows that $S$ has six self-paired suborbits and two paired orbits, say $\Delta, \Delta'$.

By \cite[Table~8.14]{Bray-Holt-Dougal}, we see that $G=S\lg g\rg$, where $g$ is a graph automorphism of $S$ which fixes $S_\a$. So $g\in G_\a$. If $g$ is a field automorphism of $S$, then $U^g=U$ and $W^g=W$, and then $G_\a=\lg S_\a, g\rg=G_U\cap G_W$ which is not maximal in $G$, a contradiction. Thus, $g$ is not a field automorphism of $S$, and then $g^2$ is field automorphism of $S$. Suppose that $g$ does not fuse $\Delta$ and $\Delta'$. Then $\Delta$ and $\Delta'$ are two non-trivial suborbits of $G$.

Let $O_1:=\{(U, W')\mid U\leq W', W'\neq W, \dim(W')=2\}$ and $O_2:=\{(U', W)\mid U'\leq W, U'\neq U, \dim(U')=1\}$. Then $O_1$ and $O_2$ are two orbits of $S_\a$, and $g$ fuses $O_1$ and $O_2$ into one orbit of $G_\a$. This implies that $G$ has at least five suborbits, a contradiction.

\medskip
\f{\bf Case~3.} The Unitary Case.\medskip

In this case, $(S, S_\alpha)$ is one of the pairs in Lines 23--36 of \cite[Table~I]{Cuypers}.

\vspace{0.1cm}
\underline{Lines 23--25}:\  For these lines, $S_\alpha$ is a maximal parabolic subgroup of $S$, and hence all $G$-orbitals are self-paired by Lemma~\ref{lm:parabolic}.

\vspace{0.1cm}
\underline{Line 26}:\  Now $S=\PSU_n(r)$ with $r \in \{2,3,4,8 \}$, and  $S_\alpha$ is the stabiliser of a nonsingular $1$-space.  From \cite[Section~2]{Wei-1983} we may see that all the orbitals of $G$ are self-paired.

\vspace{0.1cm}
\underline{Lines 27--28, 30--36}:\
By \cite[Lemma~8.3]{Cuypers} and Atlas~\cite{Conway-Atlas}, among the pairs in Lines 27--28, 30--36, the groups of rank $4$ are as follows (the subdegrees are obtained by Magma~\cite{BCP}):
\[
\begin{array}{cllllc}
\hline
\text{Line}&S& S_\alpha & \text{degree} &\text{subdegrees}& \text{rank of }\N_{{\rm Sym}(\Om)}(S) \\
\hline
27&\PSU_3(3) &4^2{:}\S_3 &63&1,6,32,24&4 \\
30&\PSU_6(2) &\mathrm{PSp}_6(2)&6336&1,315,2240,3780&4 \\
32&\PSU_3(5) &\M_{10}&175 & 1,12,72,90 &4 \\
36&\PSU_3(3) &\PSL_3(2)  & 36 &  1,7,7,21 &3 \\
\hline
\end{array}
\]
It is easy to see all the orbitals of the first four groups are self-paired. For the last group, by Magma~\cite{BCP}, exactly one non-trivial orbital is self-paired, and this is the group in Theorem~\ref{rank-4}.

\vspace{0.1cm}
\underline{Line 29}:\  Here $S=\PSU_4(r)$ with $r\in\{2,3,4,5,7,8,9\}$, and  $S_\alpha $ is the normaliser of  $\PSp_4(r)$.
Note that $\PSU_4(r)\cong \mathrm{P\Omega}^{-}_6(r)$ and $\PSp_4(r)\cong  \mathrm{P\Omega}_5(r)$.
Thus, we identify $S$ with $\mathrm{P\Omega}^{-}_6(r)$, and then $S_\alpha $ is the stabiliser of a  nonsingular $1$-space.
By Lemma~\ref{lm:Omeganonsingular}, we see that all $S$-orbitals  are self-paired.

\medskip
\f{\bf Case~4.}\ The Orthogonal Case.\medskip

In this case, $(S, S_\alpha)$ is one of the pairs in Lines 37--50 of \cite[Table~I]{Cuypers}.

\vspace{0.1cm}
\underline{Lines 37--43}:\  For these lines, $S_\alpha$ is a maximal parabolic subgroup of $S$, and hence all $G$-orbitals are self-paired by Lemma~\ref{lm:parabolic}.

 \vspace{0.1cm}
\underline{Lines 44--45}:\  Now $S=\mathrm{P\Omega}_{2m+1}(q)$  with $q \in \{3,5,7,9 \} $ or $S=\mathrm{P\Omega}^{+}_{2m}(q)$ with $q \in \{2,3,4,5,7,8,9 \} $, and $S_\alpha $ is the stabiliser of a nonsingular $1$-space. By Lemma~\ref{lm:Omeganonsingular},  all $S$-orbitals  are self-paired.

 \vspace{0.1cm}
\underline{Line 46}:\
Now $S=\mathrm{P\Omega}^{\pm}_{2m}(2)$ with $m\geq3$, and $S_\alpha $ is the stabiliser of a non-degenerate $2$-space of elliptic type.
By \cite[Lemma~8.5(ii)]{Cuypers}, we derive that either $S=\PO^-_6(2)(\cong\PSU_4(2))$ and $ S_\alpha=3^3{:}\S_4$, or $S=\PO^{+}_6(2)(\cong\PSL_4(2)\cong \A_8)$ and $ S_\alpha=(\A_5\times3){:}2$. For the former, $S$ has rank $3$, and for the latter, $S$ has rank $4$ with
subdegrees $1, 15, 30, 10$.

 \vspace{0.1cm}
\underline{Line 47}:\
 Now $S=\mathrm{P\Omega}_7(q)$ and $S_\alpha=G_2(q)$.
By \cite{Liebeck-2-closures}, all orbitals of $G$ are self-paired.

 \vspace{0.1cm}
\underline{Lines 48--50}:\
For the pairs in Lines 48-50, by \cite[Table~9]{John-3/2-tran}, we have the following:
\[
\begin{array}{cllllc}
\hline
\text{Line}&S& S_\alpha & \text{degree} &\text{subdegrees}& \text{rank of }\N_{{\rm Sym}(\Om)}(S) \\
\hline
48 &  \mathrm{P\Omega}_7(3)  &  \PSp_6(2)  & 3159 &  1,288,630,2240  &  4 \\
49 &  \mathrm{P\Omega}^{+}_8(2)  &  \A_9  & 960 &  1,84,315,560  &  4 \\
50 & \mathrm{P\Omega}^{+}_8(3)  &  \mathrm{P\Omega}^{+}_8(2)  &  28431  &  1,960,960,960,3150,22400  &  4 \\
\hline
\end{array}
\]
It implies that all orbitals of $G$ are self-paired.
\medskip

\f{\bf Case~5.}\ The Exceptional Group of Lie Type.\medskip

In this case, $(S, S_\alpha)$ is one of the pairs in Lines 51--68 of \cite[Table~I]{Cuypers}.

 \vspace{0.1cm}
\underline{Lines  51--54, 58--59, 61, 63, 65, 67--68}:\  These lines are impossible by Lemma~\ref{lm:parabolic}.

 \vspace{0.1cm}
\underline{Lines 55--56}:\
Now $S=G_2(q)$ and $S_\alpha=\mathrm{SU}_3(q).2$ or $\SL_3(q).2$.
By \cite{Liebeck-2-closures}, all orbitals of $S$ are self-paired.

 \vspace{0.1cm}
\underline{Line  57}:\
Now $S=G_2(4)$ and $S_\alpha=\mathrm{J}_2$, which  gives  rise to a group of rank $3$ by~\cite{Liebeck-Saxl}.

 \vspace{0.1cm}
\underline{Line  60}:\
Now $S={}^2F_4(2)'$ and $S_\alpha=\PSL_3(2).2$. By computation in Magma~\cite{BCP},  $S$ has rank $4$ with subdegrees $1$, $312$, $351$ and $936$.

 \vspace{0.1cm}
\underline{Line  62}:\
Now $S=F_4(2)$ and $S_\alpha=\PSp_8(2)$. From Atlas~\cite{Conway-Atlas} we see that $G=S$.
 By~\cite[Lemma~8.6~(i)]{Cuypers}, $G$ would have  rank $5$.

 \vspace{0.1cm}
\underline{Line 64}:\
Now $S={}^2E_6(2)$ and $S_\alpha=F_4(2)$. By \cite[Lemma~8.6(ii)]{Cuypers}, $G$ has rank $4$. In this case, $S_\alpha$ is the centraliser of the out-involution $\s$ of $S$. Let $G={}^2E_6(2){:}\lg\s\rg$ and $C=F_4(2)\times\lg\s\rg$. Set $\Om=\{\s^g \mid g\in G\}$. Now let $H\leq G$ s.t. $H=\PSU_6(2){:}\lg\s\rg$. Then $C_H(\s)=\PSp_6(2)\times\lg\s\rg$. Take $x,y\in\Om$ such that $x,y$ are in the same orbit of $C$. Then $\s x$ and $\s y$ are conjugate, and so they have the same order. Let $X=\{\s^h\ |\ h\in H\}$. Then $H$ is primitive on $X$ with rank $4$ and all orbitals are self-paired. By Magma~\cite{BCP}, we can obtain that $|\{o (\s x)\ |\ x\in X\}|=4$, and so each orbital of $H$ on $X$ is contained in exactly one orbital of $G$ on $\Om$. This implies that all orbitals of $G$ are also self-paired.
(We are grateful to Prof. Frank L${\rm \ddot{u}}$beck who kindly helps us to obtain the subdegrees (namely, $1,69615, 6336512, 16707600$) of this group by using GAP. )

 \vspace{0.1cm}
\underline{Line  66}:\
Now $S=E_6(2)$ and $S_\alpha=F_4(2)$, and $G$ contains a graph automorphism. By \cite[Lemma~8.7]{Cuypers}, $G$ would have rank $5$.
\hfill\qed

Now combining Propositions~\ref{prop-sporadic}, \ref{prop-alt} and \ref{prop-sim-lie}, the proof of Theorem~\ref{rank-4} is complete.

\section{Proof of Theorems~\ref{th:2sethomo} and \ref{trans-rank-4}}\label{sec6:proof-main-theorems}

\f{\bf Proof of Theorem~\ref{th:2sethomo}.}\ Let $\Ga=(V,E)$ be a $(G,2)$-set-homogeneous graph.
Suppose that $\Ga$ is neither complete nor empty.
If $G$ is imprimitive on the vertex set $V$, then by Lemma~\ref{imp},
$\Ga\cong m\K_b$ or $\K_{m[b]}$ for some positive integers $m$ and $ b$. Then Theorem~\ref{th:2sethomo}~(2) holds.

Assume that $G$ is primitive on $V$. By Lemma~\ref{rank<=5}, $G$ is of rank at most 5.

If $G$ is of rank $3$, then $\Ga$ is $(G,2)$-homogeneous and all candidates for $G$ are listed in \cite{Bannai,Kantor,Liebeck-affine,Liebeck-Saxl}. Then Theorem~\ref{th:2sethomo}~(3)  holds.

If $G$ is of rank $4$, then exactly one of the orbitals is self-paired. By Corollary~\ref{reduction}, $G$ is affine or almost simple. By Theorem~\ref{affine-rank-4}, $G$ is not affine. If $G$ is almost simple, then by Theorem~\ref{rank-4}, $G\cong\PSU_3(3)$ with stabiliser $\PSL_3(2)$, and the subdegrees are $1,7,7$ and $21$. Then $\Ga$ is either the orbital graph of $G$ of order $36$ and of valency $21$, or the complement of this graph. Furthermore, by Magma~\cite{BCP}, $\Aut(\Ga)={\rm P\Gamma U}_3(3)$ is $2$-homogeneous on $\Ga$. This proves part (4) of our theorem.

If $G$ is of rank 5, then none of the non-trivial orbitals of $G$ is self-paired. From Proposition~\ref{AGL_1(p^r)} we obtain part (5) of our theorem. \hfill\qed

\medskip
\f{\bf Proof of Theorem~\ref{trans-rank-4}.}\ Combining Proposition~\ref{th-rank-4}, Lemma~\ref{non-HA-AS}, Theorem~\ref{affine-rank-4} and Theorem~\ref{rank-4}, we obtain the proof of this theorem.\hfill\qed

\section{Finite $3$-set-homogeneous graphs}\label{set:3sethom}

In this section, we shall classify finite $3$-set-homogeneous graphs, and then prove Theorem~\ref{th:3sethomo}.

\subsection{Preliminaries}\label{subsec:3reduce}
\subsubsection{A reduction}
A finite graph $\Ga$ is \emph{strongly regular} with parameters $(v, k, \lambda, \mu)$ if it has $v$ vertices and is regular with valency $k$ where $0<k<v-1$, and if any two adjacent vertices have $\ld=\ld(\Ga)$ common neighbours and any two non-adjacent vertices have $\mu=\mu(\Ga)$ common neighbours. Note that a strongly regular graph is neither complete nor edgeless.

\begin{lem}\label{lem:prim-3ch-is-srg}
Let $\Ga $ be a finite $(G, 3)$-set-homogeneous graph. Then one of the following holds.
\begin{enumerate}
\item[{\rm(1)}] $\Ga \cong\K_n$ or $n\K_1$;

\item[{\rm(2)}]  $\Ga \cong \K_{m[b]}$ or $m\K_b$ where $m,b\geq 2$;

\item [{\rm (3)}]\ $G$ is primitive on $V(\Ga)$ of rank $3$, and $\Ga$ is strongly regular of diameter $2$. Furthermore, if $\Ga$ is not the pentagon, then either $\Ga$ or its complement $\overline{\Ga}$ has girth $3$.
\end{enumerate}
\end{lem}

\proof By \cite[Theorem~1.1 \& Corollary~1.2]{Zhou-EJC2021}, $\Ga$ is $(G,2)$-homogeneous, and then by Theorem~\ref{th:2sethomo}, we conclude that either parts (1) or (2) hold, or $G$ is primitive on $V(\Ga)$ of rank $3$.

Assume that neither part (1) nor part (2) happens. Then $\Ga$ has diameter $2$ and $G$ is arc-transitive on both $\Ga$ and the complement $\overline{\Ga}$ of $\Ga$. This implies that $\Ga$ is strongly regular. The first assertion of the part (3) is proved. If $\Ga$ has valency $2$, then since $\Ga$ has diameter $2$, it follows that $\Ga$ is the pentagon. If $\Ga$ has valency at least $3$, then either $\Ga$ or its complement $\overline{\Ga}$ has girth $3$. \hfill\qed

By Lemma~\ref{lem:prim-3ch-is-srg}, to prove Theorem~\ref{th:3sethomo}, it suffices to classify those $(G,3)$-set-homogeneous graphs $\Ga$ such that $G$ is primitive on $V(\Ga)$ of rank $3$. The following result can be obtained by the fundamental O'Nan-Scott theorem (see \cite[\textsection{1}]{Liebeck-ONan-Scott}).

\begin{theorem}\label{th:rank3groups}
Let $G$ be a primitive group of rank  $3$.
Then one of the following holds:
\begin{itemize}
\item[\rm (1)]  $T\times T\lhd G\leqslant H\wr\S_2$, where $H$ is a $2$-transitive permutation group of degree $n$ and $\soc(H)=T$ is a nonabelian simple group, and $G$ has production action;
\item[\rm (2)]  $G$ is an affine group;
\item[\rm (3)]  $G$ is an almost simple group.
\end{itemize}
\end{theorem}

\begin{lemma}\label{lm:CaseI}
Suppose that $\Ga$ is a finite $(G,3)$-set-homogeneous graph such that $G$ satisfies Theorem~{\rm\ref{th:rank3groups}~(1)}.
Then $\Ga$ is isomorphic to $\K_n\square\K_n$ or its complement $\K_n\times\K_n$, and $\Ga$ is $3$-homogeneous.
\end{lemma}

\proof Without loss of generality, assume that $H$ is a $2$-transitive permutation group on the set $\Del =\{1,2,\dots, n\}$. Then $V(\Ga)=\Del \times\Del $. Take $\a=(1, 1)\in V(\Ga)$. Then the three orbits of $G$ are
\[\Ome_0=\{\a\}, \Ome_1=\{(1, i), (i,1)\mid i\in\Delta\setminus\{1\}\}, \Ome_2=\{(i,j)\mid i, j\in\Del \setminus\{1\}\}.\]
So $\Ga(\a)=\Ome_1$ or $\Ome_2$. It is easy to see that $\Ga\cong\K_n\square\K_n$ or $\K_n\times\K_n$. By \cite[Corollary~1.2]{Cameron-3-hom}, $\Ga$ is $3$-homogeneous.\hfill\qed

By Lemmas~\ref{lem:prim-3ch-is-srg} and \ref{lm:CaseI}, from now on we will adopt the following hypothesis.

\begin{hypothesis}\label{hyp}
Let $\Ga=(V, E)$ be a finite $(G,3)$-set-homogeneous graph satisfying the following:
\begin{itemize}
  \item $G\leq\Aut(\Ga)$ is primitive on $V$ with rank $3$,
  \item $\soc(G)$ is either abelian or nonabelian simple,
  \item $\Ga$ has girth $3$ and diameter $2$.
\end{itemize}
\end{hypothesis}

We will deal with the case where $\soc(G)$ is abelian in Subsection~\ref{subsec:affine}, and the case where $\soc(G)$ is nonabelian simple in Subsection~\ref{subsec:simple}.

\subsubsection{Some useful lemmas}

The first two lemmas concern the arc-stabilisers of $(G,3)$-set-homogeneous graphs. Let $s$ be a positive integer and let $\Sig$ be a graph. For a subgroup $H$ of $\Aut(\Sig)$, we say that $\Sig$ is {\em $(H, s)$-connected-set-homogeneous} if any two isomorphic connected induced subgraphs of $\Sigma$ on at most $s$ vertices are equivalent under $H$. Clearly, every $(H,s)$-set-homogeneous graph is $(H,s)$-connected-set-homogeneous.
For any vertex $v$ of a graph $\Sig$ of diameter at least two, let $\Sig(v)$ be the vertices of $\Sig$ adjacent to $v$, and let $\Sig_2(v)$ be the vertices of $\Sig$ at distance $2$ from $v$.

\begin{lemma}{\rm \cite[Theorem~1.3~(3)-(4)]{Zhou-EJC2021}}\label{rank}
Let $\Sig$ be a finite $(H, 3)$-connected-set-homogeneous non-complete graph of girth $3$ with $H\leq\Aut(\Sig)$. Then for every edge $\{\a, \b\}$ of $\Sig$, we have the following:
\begin{enumerate}
  \item [{\rm (1)}]\ $H_{\a\b}$ has $s$ orbits on $\Sig(\a)\cap \Sig(\b)$ with equal size, where $s=1, 2, 3$ or $6$.
  \item [{\rm (2)}]\ $H_{\a\b}$ has $t$ orbits on $\Sig(\a)-(\Sig(\a)\cap \Sig(\b))\cup\{\b\}$ with equal size, where $t=1$ or $ 2$.
      \end{enumerate}
\end{lemma}

\begin{lemma}\label{lm:Gab10orbits}
Under Hypothesis~{\rm \ref{hyp}}, let $\{\alpha,\beta\}\in E$. Then $G_{\alpha\beta}$ has at most $12$ orbits on $V(\mathit{\Ga})\setminus\{\alpha,\beta\}$.
\end{lemma}

\begin{proof}By Hypothesis~{\rm \ref{hyp}}, $\Ga$ is a $(G,3)$-set-homogeneous graph. Furthermore, $\Ga$ has diameter $2$ and girth $3$. Note that $V(\mathit{\Ga})\setminus\{\alpha,\beta\}$ is partitioned into the following four subsets:
\[\Ga(\a)\cap\Ga(\b), \Ga(\a)\cap\Ga_2(\b), \Ga_2(\a)\cap\Ga(\b), \Ga_2(\a)\cap\Ga_2(\b).\]
Moreover, $G_{\a\b}$ setwise fixes each of the four subsets.

Since every $(G,3)$-set-homogeneous graph is $(G,3)$-connected-set-homogeneous, by Lemma~\ref{rank} $G_{\a\b}$ has at most $6$ orbits on $\Ga(\a)\cap\Ga(\b)$, and has at most $2$ orbits on $\Ga(\a)\cap\Ga_2(\b)$ and on $\Ga_2(\a)\cap\Ga(\b)$.

To complete the proof, it suffices to prove that $G_{\a\b}$ has at most two orbits on $\Ga_2(\a)\cap\Ga_2(\b)$. To see this, we may assume that $G_{\a\b}$ is intransitive on $\Ga(\a)\cap\Ga_2(\b)$. Then there exist $\d_1, \d_2\in\Ga_2(\a)\cap\Ga_2(\b)$ such that the orbits $\d_1^{G_{\a\b}}\neq \d_2^{G_{\a\b}}$. Then $\d_1,\d_2$ are not adjacent to $\a$ and $\b$. Since $\Ga$ is $(G,3)$-set-homogeneous, there exists a $g\in G$ such that $\d_1^g=\d_2$, and $\a^g=\b$ and $\b^g=\a$. For any $\d\in \Ga_2(\a)\cap\Ga_2(\b)$, also since $\Ga$ is $(G,3)$-set-homogeneous, there exists $h\in G_{\{\a,\b\}}$ such that $\d^h=\d_1$. If $h\in G_{\a\b}$, then $\d\in \d_1^{G_{\a\b}}$. If $h\notin G_{\a\b}$, then $h$ swaps $\a$ and $\b$ and so $hg\in G_{\a\b}$. Furthermore, $\d^{hg}=\d_1^g=\d_2$, and hence $\d\in \d_2^{G_{\a\b}}$. Thus, $\Ga_2(\a)\cap\Ga_2(\b)=\d_1^{G_{\a\b}}\cup \d_2^{G_{\a\b}}$.
The proof is completed.
\hfill\qed
\end{proof}

The next lemma gives some basic properties about strongly regular graphs, see \cite[\textsection{1.1.1-1.1.2}]{BrouwerSRG}.

\begin{lem}\label{lem:drg-basic}
Let $\Ga$ be a strongly regular graph with parameters $(v, k, \lambda, \mu)$. Then the following hold:
\begin{enumerate}
  \item [{\rm (1)}]\ The complement $\overline{\Ga}$ of $\Ga$ is strongly regular with parameters $(v, \overline{k}, \overline{\ld}, \overline{\mu})$, where $\overline{k}=v-k-1$, $\overline{\ld}=v-2k+\mu-2$ and $\overline{\mu}=v-2k+\lambda$;
  \item [{\rm (2)}]\ $0\leq \ld\leq k-1$ and $0\leq\mu\leq k$;
  \item [{\rm (3)}]\ $k(k-1-\lambda)=\mu(v-1-k)$.
\end{enumerate}
\end{lem}

Let $\mathbf{P}_{3}$ be a path of length $2$. For a graph $\Ga$, let
\[
\begin{array}{l}
\Ome_1(\Ga):=\{U\subseteq V(\Ga) \mid [U]\cong \K_1\cup\K_2 \},\\
\Ome_2(\Ga):=\{U\subseteq V(\Ga) \mid [U]\cong \mathbf{P}_{3}\},\\
\Ome_3(\Ga):=\{U\subseteq V(\Ga) \mid [U]\cong \K_3\},\\
\Ome_4(\Ga):=\{U\subseteq V(\Ga) \mid [U]\cong 3\K_1\}.\\
\end{array}
\]
\begin{lemma}\label{lm:Gaorder}
Let $\Ga$ be a strongly regular graph with parameters $(v, k, \lambda, \mu)$. Suppose that $G\leq\Aut(\Ga)$ is primitive on $V(\Ga)$ of rank $3$. Then $\Ga$ is $(G, 3)$-set-homogeneous if and only if $G$ is transitive on each $\Om_i(\Ga)$ for $i=1,2,3,4$.
Furthermore,
\begin{align*}
&|\Ome_1(\Ga)|=\frac{vk(v+\lambda-2k)}{2},\ |\Ome_2(\Ga)|=\frac{vk(k-\lambda-1)}{2},\  |\Ome_3(\Ga)|=\frac{vk\lambda}{6}, \\
& |\Ome_4(\Ga)|=\frac{v(v-k-1)(v-2k+\mu-2)}{6} =\frac{v(v^2 - 3vk - 3v + 3k^2 - k\lambda + 3k + 2)}{6}.
\end{align*}
\end{lemma}

\proof Note that every induced subgraph of $\Ga$ of order $3$ is isomorphic to $3\K_1$, $\K_1\cup\K_2$, $\mathbf{P}_3$ or $\K_3$. Since $G$ is primitive on $V(\Ga)$ of rank $3$, this implies that $\Ga$ is $(G,2)$-homogeneous. It follows that $\Ga$ is $(G,3)$-set-homogeneous if and only if $G$ acts transitively on each $\Om_i$ for $i=1,2,3,4$.

For each $U\in\Om_1(\Ga)$, $[U]$ is a union of an edge and a singleton. For every edge $e=\{\a,\b\}$ of $\Ga$, $e$ is a non-edge of $\overline{\Ga}$, and by Lemma~\ref{lem:drg-basic}, $\a$ and $\b$ have $v-2k+\ld$ common neighbours in $\overline{\Ga}$. So there are exactly $v-2k+\ld$ vertices which are not incident with $e$ in $\Ga$. It follows that $|\Om_1(\Ga)|=|E\Ga|(v-2k+\ld)=vk(v-2k+\ld)/2$.

For each $U\in\Om_2(\Ga)$, $[U]$ is a $2$-path, and so $\overline{[U]}\cong \K_1\cup\K_2$. This implies that $|\Om_2(\Ga)|=|\Om_1(\overline{\Ga})|$. By Lemma~\ref{lem:drg-basic}, $\overline{\Ga}$ is also a strongly regular graph with parameters $(v, \overline{k}, \overline{\ld}, \overline{\mu})$, where $\overline{k}=v-k-1$, $\overline{\ld}=v-2k+\mu-2$ and $\overline{\mu}=v-2k+\lambda$.
Then \[|\Om_1(\overline{\Ga})|=\frac{v\overline{k}}{2}(v-2\overline{k}+\overline{\ld})=\frac{v(v-k-1)}{2}\mu.\]
By Lemma~\ref{lem:drg-basic}, we have $\mu(v-k-1)=k(k-1-\lambda)$. It follows that $|\Om_2(\Ga)|=|\Om_1(\overline{\Ga})|= {vk(k-\ld-1)}/{2}$.


For an arbitrary edge $e \in E\Ga$, there are $\lambda$ triangles passing through $e$.
It follows that $|\Om_3(\Ga)|=|E(\Ga)|\lambda/3= {vk\lambda}/{6}$.

Clearly, for each $U\in\Om_4(\Ga)$, we have $\overline{[U]}\cong \K_3$. Consequently, we have $|\Om_4(\Ga)|=|\Om_3(\overline{\Ga})|$.  By Lemma~\ref{lem:drg-basic}~(2), $|\Om_3(\overline{\Ga})|= {v(v-k-1)(v-2k+\mu-2)}/{6}$. \hfill\qed

Let $q\geq 2$ be a prime power and $n\geq 3$. A prime $r$ is called a {\em Zsigmondy prime} of $q^n-1$ if $r\mid q^n-1$ but $r\nmid q^i-1$ for $i<n$. By \cite{Z} and \cite[Lemma~1.13.3]{Bray-Holt-Dougal}, we have the following lemma.

\begin{lemma}\label{primitive-prime}
Let $q\geq 2$ be a prime power and $n\geq 3$ with $(q,n)\neq (2, 6)$. Then $q^n-1$ has at least one Zsigmondy prime $r$. Furthermore, $r\equiv 1\ (\mod n)$.
\end{lemma}

Finally, we prove an easy lemma.

\begin{lemma}\label{lem:claim}
Let $p$ be a prime and $d$ an integer. Then
\begin{enumerate}
  \item [{\rm (i)}]\ If $d\geq 3$ and $p$ is an odd prime, then $p^d-1>8d$.
  \item [{\rm (ii)}]\ If $d\geq 6$, then $2^d-1>8d.$
\end{enumerate}
\end{lemma}

\begin{proof} For (i), if $d=3$, then $p^3-1\geq 26>8d=24$, as required. Assume that $d>3$. By induction on $d$, we have $p^{d-1}>8(d-1)$. Then $p^d-1=(p-1)\cdot p^{d-1}+p^{d-1}-1>3(p^{d-1}-1)>24(d-1)>8d$, completing the proof of part~(i).

For (ii), if $d=6$, then $2^6-1=63>8d=48$, as required. Assume that $d>6$. Then $8d>48$. By induction on $d$, we have $2^{d-1}>8(d-1)$. Then $2^d-1= 2^{d-1}+2^{d-1}-1>16(d-1)>8d$, completing the proof of part~(ii).\hfill\qed
\end{proof}

\subsection{$G$ is affine}\label{subsec:affine}

In this subsection, we shall assume that $\Ga$ is a $(G,3)$-set-homogeneous graph satisfying Hypothesis~\ref{hyp} and $\soc(G)$ is abelian. The following is the main result of this subsection.

\begin{theorem}\label{th-affine}
Under Hypothesis~{\rm\ref{hyp}}, assume that $\soc(G)$ is abelian. Then $\Ga$ is $3$-homogeneous, and, up to a complement, $\Ga$ is isomorphic to one of the following graphs: $\K_n\square\K_n$ $(n$ is a prime power$)$, the affine polar graph $\mathrm{VO}_{2m}^{\pm}(2)$ $(m\geq 1)$.
\end{theorem}

Under Hypothesis~\ref{hyp}, $\Ga$ is a $(G,3)$-set-homogeneous graph of girth $3$ and diameter $2$. Furthermore, $G$ is primitive on $V(\Ga)$ of rank $3$. To prove this theorem, assume that $G$ is of affine type. Then $\soc(G)$ is isomorphic to $\mz_p^d$ for some prime $p$ and integer $d$ and acts regularly on $V(\Ga)$. For convenience, we may let $V(\Ga)=\mz_p^d$ and let $G_0$ be the stabiliser of the zero vector $0$ in $V$. Then $G$ satisfies Theorem~\ref{subdegree}, and then we obtain the proof of this theorem from the following Lemmas~\ref{lem:class:b-c}--\ref{A8-10-11}.

\begin{lemma}\label{lem:class:b-c}
If $G$ satisfies Theorem~{\rm\ref{subdegree}} {\rm(B)--(C)}, then one of the following holds:
\begin{enumerate} [\rm (1)]
\item  $\Ga\cong \mathrm{VO}^{-}_6(2)$ or $\overline{\mathrm{VO}^{-}_6(2)}$, arising from the case where $p^d=2^6$ and $G_0\leq \N_{\GL_6(2)}(3^{1+2})$ in {\rm(B)};
\item  $\Ga\cong \K_9\square\K_9$ or $\K_9\times\K_9$, arising from the case where $p^d=3^4$ and $G_0\leq \N_{\GL_4(3)}(2^{1+4}) $ in  {\rm(B)};
\item  $\Ga\cong \mathrm{VO}^{+}_8(2)$ or $\overline{\mathrm{VO}^{+}_8(2)}$, arising from the case where $p^d=2^8$ and $\soc(G_0/\Z(G_0))=\A_9<\mathrm{P\Omega}^{+}_8(2)$ in {\rm(C)}.
\end{enumerate}
\end{lemma}

\proof If $G$ satisfies Theorem~\ref{subdegree} (B)--(C), then all the corresponding rank $3$ graphs can be found in \cite[Tables~11.5-11.6]{BrouwerSRG}, and we can also obtain their parameters $(v, k, \lambda, \mu)$ from \cite[Table~11.7]{BrouwerSRG}.
For each rank $3$ group $G$ given in Theorem~\ref{subdegree} (B)--(C), we can construct it in Magma~\cite{BCP}, and further, for each rank $3$ graph $\Ga$ arising from $G$, we can determine whether $|G|$ is divisible by $|\Om_i(\Ga)|$ with $i=1,2,3,4$, where $\Om_i(\Ga)$ is given in Lemma~\ref{lm:Gaorder}. We obtain three graphs as given in parts (1)--(3).
\hfill\qed

Next we consider the groups satisfying Theorem~\ref{subdegree} (A). The following two lemmas deal with the cases when $G$ is of type {\rm (A1)} or {\rm (A2)}, respectively.

\begin{lemma}\label{lem:A1}
If $G$ is of type {\rm(A1)}, then $\Ga\cong\mathrm{Paley}(9)\cong \K_3\square\K_3$.
\end{lemma}

\begin{proof} Assume that $G$ is of type {\rm(A1)}. Then $G\leq \mathrm{A\Gamma L}_1(p^d)$. Let $\{\alpha, \beta\}$ be an edge of $\Ga$.
Then $|V(\Ga)|=p^d$ and $G_{\alpha\beta}\leq \ZZ_d$. Furthermore, $G_{\a\b}$ is semiregular on $V(\Ga) \setminus \{\alpha, \beta\}$.
From Lemma~\ref{lm:Gab10orbits} it follows that $G_{\alpha\beta}$ has at most $12$ orbits on $V(\Ga) \setminus \{\alpha, \beta\}$.
Therefore, $G_{\alpha\beta}$ has one orbit of length at least $(|V(\Ga)|-2)/12$.
It follows that $|G_{\alpha\beta}|\geq (|V(\Ga)|-2)/12$, which means
\[ p^d-2 \leq 12d.\]
This implies that
\[
p^d\in\{2^i (1\leq i\leq6), 3^j (1\leq j\leq 3), 5^k, 7^k, 11^k (1\leq k\leq 2), 13\}.
\]
Since $\Ga$ has girth $3$ and diameter $2$ by Hypothesis~\ref{hyp}, it follows that $|V(\Ga)|=p^d\geq7$. For all of the remaining possible values of $p^d$, we make use of Magma~\cite{BCP} to construct all primitive rank $3$ subgroups $G$ of $\mathrm{A\Gamma L}_1(p^d)$, and then for each rank $3$ graph arising from $G$, we can further check if $|\Om_i(\Ga)|$ divides $|G|$ for $i=1,2,3,4$, where $\Om_i(\Ga)$ is given in Lemma~\ref{lm:Gaorder}. We find that the only possibility is $p^d=3^2$, and $\Ga\cong \mathrm{Paley}(9)\cong\K_3\square\K_3$. \hfill\qed
\end{proof}

\begin{lemma}\label{lem:A2}
If $G$ is of type {\rm(A2)}, then $\Ga$ is isomorphic to $\K_{q^{d/2}}\square\K_{q^{d/2}}$ or its complement, and $\Ga$ is $3$-homogeneous.
\end{lemma}

\proof By \cite[Table~11.4]{BrouwerSRG} we see that $\Ga\cong \K_{q^{d/2}}\square\K_{q^{d/2}}$ or $\K_{q^{d/2}}\times\K_{q^{d/2}}$. Then by \cite[Corollary~1.2]{Cameron-3-hom}, $\Ga$ is $3$-homogeneous. \hfill\qed

Next we deal with the case when $G$ is of type  (A3), (A4)  or  (A5).

Following~\cite[\textsection{3.4.1}]{BrouwerSRG}, for integers $m, n\geq 2$ and a prime power $q$,  the \emph{bilinear forms graph} $\mathrm{H}_{q}(n,m)$ is a graph of which vertices are the
$n\times m$ matrices over $\GF(q)$, and two distinct vertices are adjacent whenever their difference has rank $1$.

\begin{lemma}\label{lem:A3-A5}
Suppose that $G$ is of type {\rm(A3)}, {\rm(A4)} or {\rm(A5)}.
Then $\Ga\cong\K_4\square\K_4$ or $\K_4\times\K_4$, and $\Ga$ is $3$-homogeneous.
\end{lemma}

\proof If $G$ is of type (A3), then by \cite[\textsection{3.4.1}~and~Table~11.4]{BrouwerSRG}, $\Ga\cong\mathrm{H}_{q}(2,m)$ or $\overline{\mathrm{H}_{q}(2, m)}$, where $m \geq 2$ is an integer such that $q^{2m}=p^d$. If $G$ is of type  (A4), then by \cite[\textsection{3.4.5}~and~Table~11.4]{BrouwerSRG}, $\Ga\cong\mathrm{H}_{q}(2, m)$ or $\overline{\mathrm{H}_{q}(2,m)}$, where $m\geq 2$. If $G$ is of type  (A5), then by \cite[\textsection{3.4.5}]{BrouwerSRG}, $\Ga\cong\mathrm{H}_{q}(3, 2)\cong \mathrm{H}_{q}(2, 3)$ or  $\overline{\mathrm{H}_{q}(2,3)}$.

Noticing that $\mathrm{H}_2(2, 2)\cong\mathrm{H}(2,4)$, and by \cite[Corollary~1.2]{Cameron-3-hom}, $\mathrm{H}(2,4)$ is $3$-homogeneous. To prove the lemma, it is enough to prove the following claim.\smallskip

\f{\bf Claim~1.}\ $\mathrm{H}_{q}(2,m)$ is $3$-set-homogeneous if and only if $m=q=2$.
\smallskip

By the argument in the paragraph above Claim~1, we only need to prove the necessity. Let $\Sig =\mathrm{H}_{q}(2,m)$, and let $\a$ be the $2\times m$ zero matrix. Then $\Sig (\a)$ consists of all $2\times m$ matrices of rank $1$. Take
\[\b=\left(\begin{array}{cccc}
                 1 & 0 &  \cdots & 0 \\
                 0 & 0 & \cdots & 0 \\
               \end{array}
             \right).\]
We first prove that $\Sig (\a)\cap\Sig (\b)=N_1\cup N_2$, where
\[N_1=\left\{\left(\begin{array}{cccc}
                 s & 0 & \cdots & 0 \\
                 t & 0 & \cdots & 0 \\
               \end{array}
             \right)\,\middle\vert\, s, t \in \GF(q), t\neq 0\right\},\]
             and \[N_2=\left\{ \left(\begin{array}{cccc}
                 r_1 & r_2 &  \cdots & r_m \\
                 0 & 0 & \cdots & 0 \\
               \end{array}
             \right)\,\middle\vert\, r_i\in \GF(q), (r_1,r_2,\dots,r_m)\neq (0,0,\dots,0), (1,0,\dots,0)\right\}.\]
It is easy to check that $N_1\cup N_2\subseteq \Sig (\a)\cap\Sig (\b)$. For every $\g\in\Sig (\a)\cap\Sig (\b)$, since $\g$ has rank $1$, we may assume that
\[\g=\left(\begin{array}{cccc}
                 sr_1 & sr_2 &  \cdots & sr_m \\
                 tr_1 & tr_2 & \cdots & tr_m \\
               \end{array}
             \right)\in \Sig (\a).\]
If $t=0$, then $\g\in N_2$. Assume $t\neq0$. Note that
\[\g-\b=\left(\begin{array}{cccc}
                 sr_1-1 & sr_2 & \cdots & sr_m \\
                 tr_1 & tr_2 & \cdots & tr_m \\
               \end{array}
             \right)\]
             has rank $1$.
Since $t\neq 0$, $\g-\b$ is row equivalent to $\left(\begin{array}{cccc}
                 1 & 0 & \cdots & 0 \\
                 0 & tr_2 & \cdots & tr_m \\
               \end{array}
             \right)$.
Since $\g-\b$ has rank $1$, one has $tr_2=tr_3=\cdots=tr_m=0$, and so $r_2=r_3=\cdots=r_m=0$ as $t\neq 0$. This implies that $\g\in N_1$.
Thus, we have proved that $\Sig (\a)\cap\Sig (\b)=N_1\cup N_2$.

Clearly, $[N_1]\cong\K_{q(q-1)}$ and $[N_2]\cong\K_{q^m-2}$. Furthermore, for
\[\g_1=\left(\begin{array}{cccc}
                 s & 0 &  \cdots & 0 \\
                 t & 0 & \cdots & 0 \\
               \end{array}
             \right)\in N_1, \g_2=\left(\begin{array}{cccc}
                 r_1 & r_2 &  \cdots & r_m \\
                 0 & 0 & \cdots & 0 \\
               \end{array}
             \right)\in N_2,\]
$\g_1$ is adjacent to $\g_2$ if and only if $r_2=\cdots=r_m=0$ and $r_1\neq0,1$. Set
\[ N_{21}=\left\{\left(\begin{array}{cccc}
                 r & 0 &  \cdots & 0 \\
                 0 & 0 & \cdots & 0 \\
               \end{array}
             \right)\,\middle\vert\  r\neq 0,1\right\} \text{ and } N_{22}=N_2\setminus N_{21}.\]
The above argument implies that in $[\Sig (\a)\cap\Sig (\b)]=[N_1\cup N_2]$, every vertex of $N_1$ has valency $q(q-1)-1+(q-2)=q^2-3$, every vertex of $N_{21}$ has valency $(q^m-3)+q(q-1)=q^m+q^2-4$ and every vertex of $N_{22}$ has valency $q^m-3$.

Note that $|N_{21}|=q-2$ and $|N_{22}|=(q^m-2)-(q-2)=q^m-q$. By Proposition~\ref{rank}~(3), $G_{\a\b}$ has $s$ orbits on $\Sig(\a)\cap\Sig(\b)$ with equal size $\ell$, where $s=1, 2, 3$ or $6$. Note that all the vertices lying in the same orbit of $G_{\a\b}$ on $\Sig(\a)\cap\Sig(\b)$ have the same valency in $[\Sig(\a)\cap\Sig(\b)]$.

If $m>2$, then $q^2-3, q^m+q^2-4, q^m-3$ are pair-wise distinct. Then either $N_1, N_{21}$ and $N_{22}$ are three orbits of $G_{\a\b}$ on $\Sig(\a)\cap\Sig(\b)$, or each of $N_1, N_{21}$ and $N_{22}$ splits into two orbits of $G_{\a\b}$ on $\Sig(\a)\cap\Sig(\b)$. This implies that
$q^2-3=q^m+q^2-4=q^m-3$, a contradiction.

If $m=2$ and $q>2$, then in $[\Sig (\a)\cap\Sig (\b)]=[N_1\cup N_2]$, every vertex of $N_1\cup N_{22}$ has valency $q(q-1)-1+(q-2)=q^2-3$, and every vertex of $N_{21}$ has valency $(q^m-3)+q(q-1)=2q^2-4$. Clearly, $2q^2-4>q^2-3$. Recall that $|N_{21}|=q-2$, $|N_1|=q^2-q$ and $|N_{22}|=q^2-q$. This implies that $N_{21}$ is a union of some orbits of $G_{\a\b}$ on $\Sig(\a)\cap\Sig(\b)$. Then $\ell$ divides $ |N_{21}|$, and $|\Sig(\a)\cap\Sig(\b)|=6\ell$. If $|N_{21}|=\ell$, then $|N_1|+|N_{22}|=5\ell$ and so $\ell$ divides $ |N_1|$ because $|N_1|=|N_{22}|=q^2-q$. It follows that $q-2\mid (q^2-q)$. Since $(q-2, q-1)=1$, one has $q-2\mid q$, implying $q=4$. Then $24=2(q^2-q)=|N_1|+|N_{22}|=5\ell=5|N_{21}|=5(q-2)=10$, a contradiction. If $|N_{21}|\geq2\ell$, then since $q>2$, one has $q^2-q>q-2$, and hence $|N_1|+|N_{22}|>4|N_{21}|=8\ell>6\ell=|\Sig(\a)\cap\Sig(\b)|$, a contradiction. Thus, $m=q=2$, proving Claim~1.\hfill\qed

Now we consider the case where $G$ is of type  {\rm(A6)}, {\rm(A7)} or {\rm(A9)}.

\begin{lemma}\label{lem:A6-7-9}
Suppose that $G$ is of type  {\rm(A6)}, {\rm(A7)} or {\rm(A9)}.
Then $\Ga\cong\K_q\square\K_q$, $\mathrm{VO}_{2m}^{+}(2)$ or $\mathrm{VO}_{2m}^{-}(2)$ with $m\geq 1$, or their complements. In particular, $\Ga$ is $3$-homogeneous.
\end{lemma}

\proof From~\cite[\textsection{3.3.1}~and~Table~11.4]{BrouwerSRG}, we see that if $G$ is of type  (A6), (A7)  or  (A9), then $\Ga$ is isomorphic to $\mathrm{VO}_{2m}^{\epsilon}(q)$ for some positive integer $m$, where $\epsilon\in \{-,+\}$. By definition, we see that $\mathrm{VO}_{2}^{-}(q)\cong q^2\K_1$. 
Furthermore, as remarked in \cite[\textsection{3.3.1}]{BrouwerSRG}, we know that $\mathrm{VO}_{2}^{+}(q)\cong \K_q\square\K_q$. By \cite[Corollary~1.2]{Cameron-3-hom}, the graphs $\K_q\square\K_q$, $\mathrm{VO}_{2m}^{+}(2)$ and $\mathrm{VO}_{2m}^{-}(2)$ are $3$-homogeneous, where $m\geq 1$.

Let $\Sig =\mathrm{VO}_{2m}^{\pm}(q)$, where $m\geq 2$ and $q\geq 3$.
It remains to show that $\Sig$ is not $3$-set-homogeneous.

Suppose that $m=2$ and $\epsilon=-$. By~\cite[\textsection{3.3.1}]{BrouwerSRG}, $v=v(\Sig )=q^{4}=p^d$, $k=k(\Sig )=(q^2+1)(q-1)$ and $\lambda=\ld(\Sig )=q-2$. Let $A=\Aut(\Sig )$. Then $A=\ZZ_p^d{:}A_0$, where $A_0={\rm\Gamma O}_{4}^-(q)$ (see \cite[Theorem~1.1]{Skresanov-arXiv}). Then $|A_0|=2q^{2}(q^2+1)(q^2-1)(q-1)d$. By Lemma~\ref{lm:Gaorder}, we have
$ {vk\ld}/{6} $ divides $ |A|.$ It follows that $k\ld=(q^2+1)(q-1)(q-2)$ is a divisor of $6|A_0|=12q^{2}(q^2+1)(q^2-1)(q-1)d$. Hence,
$(q-2)\mid 12dq^2(q^2-1)$, and then $(q-2)\mid 36dq^2$ as $(q-2, q-1)=1$ and $(q-2, q+1)\in \{1,3\}$.
Assume first that $q$ is odd. Then $(q-2, q)=1$, and then $(q-2,q^2)=1$. So $(q-2)\mid 36d$, namely, $(p^d-2)\mid 36d$. By induction on $d$, one may prove that if $d\geq 5$, then $p^d-2>36d$. Thus, $d\leq 4$. Since $p^d=q^4$, one has $d=4$. However, since $p^d-2$ is an odd divisor of $36d=4\cdot9d$, one has $p^d-2=1$, $3,9$ or $27$, and so $p^d=3,5,11$ or $29$, implying $d\leq 2$, a contradiction. Assume now that $q$ is even. Then $(q-2, q)=2$ and then $(q-2, q^2)=2$ since $4\nmid (q-2)$. Thus, $(q-2)\mid 72d$, namely, $(2^d-2)\mid 72d$. By induction on $d$, one may prove that if $d\geq 10$, then $2^d-2>72d$. Thus, $d\leq 9$, and then $4\leq d\leq 9$ as $2^d=q^4$. Clearly, $2^{d-1}-1$ is an odd divisor of $36d$, and $2^{d-1}-1\geq 7$. We have $2^{d-1}-1\in\{9, 5, 45, 27, 7, 21, 63, 81\}$. So $2^{d-1}-1=7$ or $63$, and then $q=2^d=2^4$ or $2^7$.
By Lemma~\ref{lm:Gaorder}, we have $k(v+\ld-2k)/2$ divides $ |A_0|$, and then
\[\frac{(q^2+1)(q-1)}{2}(q^4+(q-2)-2(q^2+1)(q-1))\mid 2dq^{2}(q^4-1)(q-1).\]
It follows that \[ (q^4-2q^3+2q^2-q) \mid 2^2dq^{2}(q^2-1).\]
If $q=2^d=2^4$, then $(q^3-2q^2+2q-1)\mid (q^2-1)$, which is impossible. If $q=2^d=2^7$, then $(q^3-2q^2+2q-1)\mid 7(q^2-1)$, which is also impossible.

Suppose that $\epsilon=+$ and $m\geq 2$ or $\epsilon=-$ and $m\geq 3$. Recall that the vertex set of $\Sig =\mathrm{VO}_{2m}^{\pm}(q)$ is a vector space $V$ of dimension $2m$ over $\GF(q)$, provided with a non-degenerate quadratic form $Q$ of type $\varepsilon=\pm$. Furthermore, two different vectors $u, v$ of $\Sigma$ are adjacent whenever $Q(v-u)=0$. Since either $\epsilon=+$ and $m\geq 2$ or $\epsilon=-$ and $m\geq 3$,  there are two mutually orthogonal hyperbolic pairs $(e_1, f_1)$ and $(e_2, f_2)$ in $V$, and so $\lg e_1, e_2\rg$ is totally singular subspace of $V$. This implies that $[\lg e_1, e_2\rg]$ is a clique. Noticing that $q\geq 3$, we can take $a\in \GF(q)\setminus \{0,1\}$. Let $\alpha_1=0$, $\alpha_2=e_1$, $\alpha_3=ae_1$ and $\alpha_4=e_2$. Then $[\{\alpha_1, \alpha_2, \alpha_3\}]$ and $[\{\alpha_1, \alpha_2, \alpha_4\}]$ are different triangles. Clearly, $\langle \alpha_1,\alpha_2,\alpha_3\rangle$ has dimension $1$ and $\langle \alpha_1,\alpha_2,\alpha_4\rangle$ has dimension $2$, so these two triangles are not equivalent under  $\Aut(\Sig )\leq q^{2m}{:}\GammaL_{2m}(q)$. Thus, $\Sig$ is not $3$-set-homogeneous.
\hfill\qed

Finally, we prove the following lemma.

\begin{lemma}\label{A8-10-11}
$G$ is not of type {\rm(A8)}, {\rm(A10)} or {\rm(A11)}.
\end{lemma}

\proof Suppose first that $G$ is of type  (A8). From~\cite[\textsection{3.4.2}~and~Table~11.4]{BrouwerSRG}, up to a complement, $\Ga$ is the alternating forms graph on $\GF(q)^5$, which has vertices the $5\times 5$ skew-symmetric matrices with zero diagonal over $\GF(q)$, and two vertices are adjacent whenever their difference has rank $2$. If $q=2$, then $G_0\leq \GL_5(2)$ and from~\cite[\textsection{3.4.2}~and~Table~11.4]{BrouwerSRG}, we see that $\Ga$ has parameters $(v,k,\ld,\mu)=(2^{10}, 155, 42, 20)$.
By Lemma~\ref{lm:Gaorder}, $vk(v+\lambda-2k)/2$ is a divisor of $|G|$, and so $k(v+\lambda-2k)/2=2\cdot3^3\cdot5\cdot7\cdot11$ is a divisor of $|\GL_5(2)|=2^{10}\cdot3^2\cdot5\cdot7\cdot11$, a contradiction. Assume next $q>2$. Let $\a$ be the $5\times 5$ zero matrix. Let
\[\b=\left(
    \begin{array}{ccccc}
      0 & 1 & 0 & 0 & 0 \\
      -1 & 0 & 0 & 0 & 0 \\
      0 & 0 & 0 & 0 & 0 \\
      0 & 0 & 0 & 0 & 0 \\
      0 & 0 & 0 & 0 & 0 \\
    \end{array}
  \right)\ {\rm and}\ \g=\left(
    \begin{array}{ccccc}
      0 & 1 & 1 & 0 & 0 \\
      -1 & 0 & 0 & 0 & 0 \\
      -1 & 0 & 0 & 0 & 0 \\
      0 & 0 & 0 & 0 & 0 \\
      0 & 0 & 0 & 0 & 0 \\
    \end{array}
  \right).
\]
Since $q>2$, we can take $t\in \GF(q)\setminus \{0,1\}$. It is easy to see that $[\{\a,\b, t\b\}]$ and $[\{\a, \b, \g\}]$ are two triangles.
Note that we may view $V(\Ga)$ as a $10$-dimensional vector space over $\GF(q)$ and so  $G_0\leq\GammaL_{10}(q)$. Clearly, the subspace generated by $\a, \b, t\b$ has dimension $1$ but the subspace generated by $\a,\b, \g$ has dimension $2$. So $ \{\a,\b, t\b\} $ and $ \{\a, \b, \g\} $ are not equivalent under $G$, a contradiction.

Suppose now that $G$ is of type  (A10). By \cite[\textsection{3.3.3}~and~Table~11.4]{BrouwerSRG}, we may assume that $\Ga= VD_{5,5}(q)$, the affine half spin graph. By \cite[p.95]{BrouwerSRG}, the vertex set of $\Ga= VD_{5,5}(q)$ is a $16$-dimensional vector space $V$ over $\GF(q)$, and there exists a subset $S$ of vectors in $V$ such that two different vertices $u,v$ of $\Ga$ are adjacent if and only if $u-v\in S$. Furthermore, $S$ contains a subset $\Phi$ which is defined as follows: Let $V=V_1\oplus V_2$ such that $V_1, V_2$ are two $8$-dimensional subspaces over $\GF(q)$. Identify $V_1$ with $\GF(q)^8$, labelling the coordinates $x_i$ with $i\in \{-4, -3, -2, -1, 1, 2, 3, 4\}$ and consider the quadratic form
\[Q: V_1\rightarrow \GF(q): (x_{-4}, x_{-3}, x_{-2}, x_{-1}, x_{1}, x_{2}, x_{3}, x_{4})\mapsto x_{-1}x_1+x_{-2}x_2+x_{-3}x_3+x_{-4}x_4.\]
Let $\Phi=\{u\in V_1\mid Q(u)=0\}$. If $q>2$, then we can take $t\in \GF(q)\setminus\{0,1\}$. Let $\a$ be the zero vector of $V$. Let $\b=(0,0,1,1,1,-1,0,0)\in V_1$ and let $\g=(1,1,1,1,1,-1,1,-1)\in V_1$. Then $\b, t\b, \g, (t-1)\b, \g-\b\in\Phi$. So $[\{\a, \b, t\b\}]$ and $[\{\a, \b, \g\}]$ are two triangles. Clearly, the subspace generated by $\a, \b, t\b$ has dimension $1$ but the subspace generated by $\a,\b, \g$ has dimension $2$. So $ \{\a,\b, t\b\} $ and $ \{\a, \b, \g\} $ are not equivalent under $G$ since  $G\leq q^{16}{:} \GammaL(16,q)$, a contradiction. Thus, $q=2$. Then by \cite[Proposition~3.3.1]{BrouwerSRG}, $\Ga$ has parameters $(v,k,\ld,\mu)=(q^{16}, (q^8-1)(q^3+1), q^8+q^6-q^3-2, q^3(q^3+1))$. By~\cite[p.100]{BrouwerSRG} (or~\cite[Theorem~1.1]{Skresanov-arXiv}), $G$ is contained in the subgroup of index $2$ in $q^{16} {:} \Aut(\mathrm{GO}^{+}_{10}(q))$  preserving the systems of maximal singular subspaces. Since $\Ga$ is $(G,3)$-set-homogeneous, by Lemma~\ref{lm:Gaorder}, $ vk(v+\ld-2k)/2$ divides $|G|$, and hence $k(v+\ld-2k)/2=2^{2}\cdot3^3\cdot5\cdot13\cdot17\cdot19\cdot31$ divides $|\Aut(\mathrm{GO}^{+}_{10}(q))|/2=2^{20}\cdot3^5\cdot5^2\cdot7\cdot17\cdot31$, a contradiction.

Finally, suppose that $G$ is of type  (A11). From~\cite[Table~11.4]{BrouwerSRG}, we see that $\Ga$ is the Suzuki-Tits ovoid graph $\mathrm{VSz}(q)$, where $q=2^{2e+1}$ with $e\geq1$, and by \cite[\textsection{3.3.1}]{BrouwerSRG}, the parameters of this graph are $(v,k,\ld,\mu)=(q^{4}, (q^2+1)(q-1), q-2, q(q-1))$. By \cite[Theorem~1.1]{Skresanov-arXiv}, we have $G_0\leq (\ZZ_{q-1} \times \mathrm{Sz}(q)).(2e+1)$. Since $\Ga$ is $(G,3)$-set-homogeneous, by Lemma~\ref{lm:Gaorder}, we have $ {vk\ld}/{6}$ divides $|G|$. It follows that $ {k\ld}/{6}$ divides $ |G_0|$ and hence $ {k\ld}/{6}$ divides $ |(\ZZ_{q-1} \times \mathrm{Sz}(q)).(2e+1)|$. It follows that
\[\frac{(q^2+1)(q-1)(q-2)}{6} \mid q^2(q^2+1)(q-1)^2(2e+1),\]
and then
\[(q-2)\mid 6q^2(q-1)(2e+1).\]
Since $(q-2, q-1)=1$ and $(q-2, q^2)=(2(2^{2e}-1),2^{4e+2})=2$, we derive that $(q-2)\mid 12(2e+1)$, namely, $2(2^{2e}-1)\mid 12(2e+1)$.
Thus, $(2^{2e}-1)\mid 3(2e+1)$. By Lemma~\ref{lem:claim}~(ii), if $e\geq 3$, then $2^{2e}-1>16e=6e+10e\geq 3(2e+1)$. It follows that $e\leq 2$. By Lemma~\ref{lm:Gaorder}, $vk(v+\lambda-2k)/2$ is a divisor of $|G|$, and so $ k(v+\lambda-2k)/2$ divides $|G_0|$. This would imply that either $e=1$ and $726180= (8^2+1)(8-1)(8^4+8-2-2(8^2+1)(8-1))/2$ divides $3\cdot 8^2(8^2+1)(8-1)^2=611520$, or $e=2$ and $3025996800= (32^2+1)(32-1)(32^4+32-2-2(32^2+1)(32-1))/2$ divides $5\cdot 32^2(32^2+1)(32-1)^2=5043328000$, a contradiction.
\hfill\qed

\subsection{Almost simple case}\label{subsec:simple}

In this section, we shall assume that $\Ga$ is a $(G,3)$-set-homogeneous graph satisfying Hypothesis~\ref{hyp} and $\soc(G)$ is non-abelian simple. The following is the main result of this subsection.

\begin{theorem}\label{th-almost-simple}
Under Hypothesis~\ref{hyp}, let $\soc(G)$ be non-abelian simple. Then $\Ga$ is $3$-homogeneous, and, up to a complement, $\Ga$ is isomorphic to one of the following graphs:
\begin{enumerate}[{\rm (1)}]
  \item an orbital graph of the Higman-Sims group $\mathrm{HS}$ on $100$ points;
  \item an orbital graph of the McLaughlin group $\mathrm{McL}$ on $275$ points;
  \item the elliptic orthogonal graph $\mathrm{\Gamma}(\mathrm{P\Omega}^{-}_6(q))$.
\end{enumerate}
\end{theorem}

\proof By Hypothesis~\ref{hyp}, $\Ga$ is a $(G, 3)$-set-homogeneous graph such that $G\leq\Aut(\Ga)$ is primitive on $V(\Ga)$ of rank $3$. Assume further that $G$ is almost simple with socle $T$. We shall finish the proof by the following five steps:

\smallskip
\f{\bf Step~1.}\  If $T$ is a sporadic simple group, then one of the following holds:
\begin{enumerate}[\rm (1)]
\item $\Ga$ is an orbital graph of the Higman-Sims group $\mathrm{HS}$ on $100$ points, and $G=\mathrm{HS}$ or $\mathrm{HS}.2$.
\item $\Ga$ is an orbital graph of the McLaughlin group $\mathrm{McL}$ on $275$ points, and $G=\mathrm{McL}$ or $\mathrm{McL}.2$.
\end{enumerate}

Suppose that $T$ is  a sporadic simple group. One may find all candidates for $T$ from \cite[Table~11.3]{BrouwerSRG}, and the parameters of the rank $3$ graphs arising from each $T$ from \cite[Table~11.7]{BrouwerSRG}. For each of these groups $T$, we can use Magma~\cite{BCP} to determine the rank $3$ groups $G$ and the rank 3 graphs $\Ga$ arising from $G$ satisfying $T=\soc(G)$ and $|\Om_i(\Ga)|$ divides $|G|$ for $i=1,2,3,4$, where $\Om_i(\Ga)$ is given in Lemma~\ref{lm:Gaorder}. Then we obtain the graphs $\Ga$ and groups $G$ as given (1) and (2), respectively.

\smallskip
\f{\bf Step~2.}\ $T\ncong\A_n$ with $n\geq 5$.
\smallskip

Suppose on the contrary that $T\cong\A_n$ with $n\geq 5$. 
According to~\cite[Theorem~11.3.1]{BrouwerSRG}, one of the following holds:
\begin{enumerate}[\rm (1)]
\item Up to a complement, $\Ga$ is the Johnson graph $J(n,2)$ whose vertex set is the set of $2$-subsets of $\{ 1,2,\ldots,n\}$, and two $2$-subsets $\alpha$ and $\b$ are adjacent if and only if $|\alpha \cap \b|=1$, and $\Aut(\Ga)=\S_n$.
\item $n=6$, $G_\a\cong\S_2\wr\S_3$, and the graph $\Ga$ has parameters $(v,k,\lambda,\mu)=(15,6,1,3)$.
\item $n=8$, $G_\a\cong\S_4\wr\S_2$, and the graph $\Ga$ has parameters $(v,k,\lambda,\mu)=(35,16,6,8)$.
\item $n=10$, $G_\a\cong\S_5\wr\S_2$, and the graph $\Ga$ has parameters $(v,k,\lambda,\mu)=(126, 25, 8, 4)$.
\item $n=9$, $G=\A_9$, $G_\a \cong \mathrm{P\Gamma L}_2(8)$, and the graph $\Ga$ has parameters $(v,k,\lambda,\mu)=(120, 56, 28, 24)$.
\end{enumerate}

For (1), $\Ga$ is the Johnson graph $J(n,2)$. Let  $\alpha_1=\{1,2\}$,  $\alpha_2=\{1,3\}$, $\alpha_3=\{2,3\}$ and $\alpha_3'=\{1,4\}$.
Then $[\{\alpha_1,\alpha_2,\alpha_3 \}]$ and $[\{\alpha_1,\alpha_2,\alpha_3' \}]$ are two triangles.  However, the setwise stabiliser of $\{\alpha_1,\alpha_2,\alpha_3 \}$ in $\S_n$ is $\Sym(\{2, 3\})\times \Sym(\{4, 5,\ldots, n\})\cong \S_2 \times \S_{n-3}$, while the setwise stabiliser of $\{\alpha_1,\alpha_2,\alpha_3' \}$ in $\S_n$ is $\Sym(\{2, 3, 4\})\times \Sym(\{5, 6, \ldots, n\})\cong \S_3 \times \S_{n-4}$.
This implies that $[\{\alpha_1, \alpha_2, \alpha_3\}]$ and $[\{\alpha_1, \alpha_2, \alpha_3'\}]$ are not equivalent under $\Aut(\Ga)$, a contradiction.

For every graph in (2)--(5), by Magma~\cite{BCP}, we see that the induced subgraphs isomorphic to $3\K_1$ are not equivalent under $\Aut(\Ga)$, a contradiction.\smallskip

From now on we deal with simple groups of Lie type. In the remainder of the proof, we let $q=p^s$ where $p$ is a prime and $s$ is a positive integer.

\smallskip
\f{\bf Step~3.}\ $T$ is not isomorphic to a simple group of exceptional Lie type.
\smallskip

Suppose on the contrary that $T$ is a simple group of exceptional Lie type. From~\cite[Theorem~11.3.4]{BrouwerSRG}, we see that
one of the following holds:
\begin{enumerate}[\rm (1)]
\item $T=E_6(q)$ and $\Ga$ is the collinearity graph of $E_{6,1}(q)$ (see \cite[Proposition~4.9.1]{BrouwerSRG}).
\item $T=G_2(q)$ with $q=3,4,8$, and the graphs $\Ga$ arising from $T$ are listed in \cite[Table~11.2]{BrouwerSRG}.
\end{enumerate}

For (1), by \cite[Proposition~4.9.1]{BrouwerSRG}, the parameters $(v,k,\ld,\mu)$ of $\Ga$ are
\[\left(\frac{(q^{12}-1)(q^9-1)}{(q^{4}-1)(q -1)}, \frac{q(q^3+1)(q^8-1)}{q-1}, \frac{q^2(q^2+1)(q^5-1)}{q-1}+q-1, \frac{(q^3+1)(q^4-1)}{q-1}\right).\]
Since $\Ga$ is $(G,3)$-set-homogeneous, it follows from Lemma~\ref{lm:Gaorder} that $vk(v+\lambda-2k)/2$ is a divisor of $|G|$.  Since
$G\leq\Aut(E_6(q))$, $|G|$ is a divisor of $2sq^{36}\prod_{i\in \{2, 5, 6, 8, 9, 12\}}(q^i-1)$.
Computation shows that
\[
\begin{split}
v+\lambda-2k& = q^7(q^2 - q + 1)(q^3 + q^2 - 1)(q^4 + q^3 + q^2 + q + 1),\\
\frac{|G|}{vk} & \mid  6sq^{35}(q-1)^6(q+1)^2(q^2+1)(q^2+q+1)(q^4 + q^3 + q^2 + q + 1).
\end{split}
\]
Then since $vk(v+\lambda-2k)/2$ divides $|G|$, it implies that
\[
(q^2-q+1)(q^3+q^2 - 1) \mid 4s(q-1)^6(q+1)^2(q^2+1)(q^2+q+1).
\]
By a direct computation, we have $q\neq 2$. By Lemma~\ref{primitive-prime}, there exists a Zsigmondy  prime divisor $r$ of $q^6-1=p^{6s}-1$ such that $r$ does not divide $p^i-1$ for $i<6s$. As $q^6-1=(q^3-1)(q^3+1)$, one has $r\mid (q^3+1)$. Since $q^3+1=(q+1)(q^2-q+1)$ and $r\nmid (q^2-1)$, one has $r\mid  (q^2-q+1) $, and so $r\mid (q^2-q+1)(q^3+q^2-1)$. Since $r\nmid (q^i-1)$ for $i<6$, one has $r\nmid (q-1)^6(q+1)^2(q^2+1)(q^2+q+1)$. It follows that $r\mid 4s$. However, by Lemma~\ref{primitive-prime}, we have $r\equiv 1\pmod{6s}$, a contradiction.

 For (2), if $T=G_2(3)$, then by \cite[Table~11.7]{BrouwerSRG}, the corresponding rank $3$ graph has parameters $(v,k,\ld,\mu)=(351, 126, 45, 45)$, and by Lemma~\ref{lm:Gaorder}, we would have $ {vk\ld}/{6}=351\cdot63\cdot15 $ divides $|\Aut (G_2(3))|=2^7\cdot3^6\cdot7\cdot13$, a contradiction. If $T=G_2(4)$, then by \cite[Table~11.7]{BrouwerSRG}, there are two corresponding rank $3$ graphs, of which one has parameters $(v, k, \ld, \mu)=(416, 100, 36, 20)$ and the other has parameters $(v, k, \ld, \mu)=(2016,  975, 462, 480)$. By Lemma~\ref{lm:Gaorder}, for the former, we would have $ v(v-k-1)(v-2k+\mu-2)/{6}=2^5\cdot3^3\cdot5\cdot7\cdot13^2$ divides $|\Aut(G_2(4))|=2^{13}\cdot 3^3\cdot5^2\cdot7\cdot13$, and for the latter, we would have $ {vk\ld}/{6}=2^5\cdot3^3\cdot5^2\cdot7^2\cdot11\cdot13$ divides $|\Aut (G_2(4))|=2^{13}\cdot 3^3\cdot5^2\cdot7\cdot13$. This is clearly impossible. If $T=G_2(8)$, then by \cite[Table~11.7]{BrouwerSRG}, the corresponding rank $3$ graph has parameters $(v, k, \ld, \mu)=(130816, 32319, 7742, 8064)$, and by Lemma~\ref{lm:Gaorder}, we would have ${vk\ld}/{6}=2^8\cdot3^4\cdot7^4\cdot19\cdot73\cdot79  $ divides $|\Aut (G_2(8))|=2^{18}\cdot3^6\cdot7^2\cdot19\cdot73$, a contradiction.

\smallskip
\f{\bf Step~4.}\ $T$ is not isomorphic to $\PSL_m(q)$ with $(m,q)\neq (2,2), (2,3)$.
\smallskip

Suppose on the contrary that $T\cong\PSL_m( q)$ with $(m,q)\neq (2,2), (2,3)$.
By \cite[Theorem~11.3.3]{BrouwerSRG}, we see that $T\in\{\PSL_m(q) (m\geq 4), \PSL_2(5)\cong\PSL_2(4)\cong\A_5, \PSL_2(9)\cong\A_6, \PSL_2(8), \PSL_4(2)\cong\A_8, \PSL_3(4), \PSL_4(3)\}$.

By Step~2, $T$ is not isomorphic to $\A_n$ for $n\geq5$, and so $T$ is not isomorphic to $\PSL_2(5)\cong\PSL_2(4), \PSL_2(9)$, $\PSL_4(2)$. If $T=\PSL_2(8)$, then by \cite[Theorem~11.3.3]{BrouwerSRG}, the graph arising from $T$ has $36$ vertices and automorphism group $\S_9$, which is also impossible by Step~2. If $T\cong\PSL_3(4)$, then by \cite[Theorem~11.3.3]{BrouwerSRG}, $\Ga$ is the Gewirtz graph with parameters $(v, k, \ld, \mu)=(56, 10, 0, 2)$, and by Lemma~\ref{lm:Gaorder}, we have $ {vk(k-\ld-1)}/{2}=2^3\cdot 5\cdot 7\cdot19$ divides $|\Aut(\PSL_3(4))|=2^8\cdot 3^3\cdot 5\cdot7$, a contradiction. If $T\cong\PSL_4(3)$, then by \cite[Theorem~11.3.3]{BrouwerSRG}, $\Ga$ is the $NO_6^+(3)$ graph (see \cite[p.293]{BrouwerSRG}) with parameters $(v, k, \ld, \mu)=(117, 36, 15, 9)$, and by Lemma~\ref{lm:Gaorder}, we have $ v(v-k-1)(v-2k+\mu-2)/{2}=2^5\cdot 3^2\cdot 5\cdot13^2$ divides $ |\Aut(\PSL_4(3))|=2^9\cdot 3^6\cdot 5\cdot13$, a contradiction.

Suppose now that $T\cong\PSL_m(q)$ with $m\geq 4$. By \cite[\textsection{3.5.1}~and~Theorem~11.3.3]{BrouwerSRG}, $\Ga$ is the Grassmann graph with vertex set the set of $2$-subspaces of $\GF(q)^m$, and two different $2$-subspaces $\alpha$ and $\b$ are adjacent whenever $\dim(\alpha \cap \b)=1$. Moreover, by \cite[Theorem~9.3.1]{Brouwer-Cohen-Neumarer}, $\Aut(\Ga)=\mathrm{P\Gamma L}_m( q)$ if $m\geq 5$, and $\Aut(\Ga)=\mathrm{P\Gamma L}_m( q)\lg\s\rg$ if $m=4$, where $\s$ is a graph automorphism of $\PSL_4(q)$. Fix a basis $\{e_1, e_2, \ldots, e_m\}$ of $\GF(q)^m$.
Let
\[\alpha_1=\langle e_1,e_2 \rangle, \alpha_2=\langle e_1, e_3\rangle, \alpha_3=\langle e_2,e_1+e_3\rangle,  \alpha_4=\langle e_1,e_4\rangle.
\]
Then $[\{\alpha_1, \alpha_2, \alpha_3\}]$ and $[\{ \alpha_1, \alpha_2, \alpha_4\}]$ are two triangles.
Note that $\dim(\langle \alpha_1,\alpha_2,\alpha_3\rangle)=3$ and $\dim(\langle \alpha_2,\alpha_3,\alpha_4\rangle)=4$. Clearly, $[\{\alpha_1, \alpha_2, \alpha_3\}]$ and $[\{ \alpha_2, \alpha_3, \alpha_4\}]$ are not equivalent under $\mathrm{P\Gamma L}_m( q)$.

Therefore, we must have $m=4$. Notice that now $q \neq 2$ by Step 2.
From~\cite[\textsection{3.5.1}]{BrouwerSRG} we see that $\Ga$ has parameters
\[(v,k,\ld,\mu)=\left(\frac{(q^4-1)(q^3-1)}{(q^2-1)(q-1)}, q(q+1)^2, 2q^2+q-1, (q+1)^2\right).\]
By Lemma~\ref{lm:Gaorder}, we conclude that $\lambda$ divides $6|G|/(vk)$, which implies that \[2q^2+q-1 \mid 12sq^5(q-1)^3.\] Since $2q^2+q-1=(2q-1)(q-1)$, $(2q^2+q-1,q)=1$ and $(2q-1,q-1)=1$, it follows that $(2q-1)\mid 3s$. If $q=p^s$ with $p$ odd, then by Lemma~\ref{lem:claim}~(i), we have $s\leq2$ and so $2q-1=1$ or $3$, which is impossible. If $q=2^s$, then by Lemma~\ref{lem:claim}~(ii), we have $s\leq5$ and so $2q-1=3, 9, 5$ or $15$. It follows that $q=2^3$. Then we would have $2q-1=15$ divides $ 3s=3\cdot 3$, a contradiction.

\smallskip
\f{\bf Step~5.}\ Let $T$ be one of the groups $\PSp_{2m-2}( q)$, $\mathrm{P\Omega}^{\pm}_{2m}(q)$, $\mathrm{P\Omega}_{2m-1}(q)$ or $\PSU_m(q)$ with $m\geq 3$. Then $T\cong \PSU_4(q)\cong\mathrm{P\Omega}^{-}_6( q)$, and up to a complement, $\Ga\cong\mathrm{\Gamma}(\mathrm{O}^-_6(q))$.

\smallskip
From \cite[Theorem~11.3.2]{BrouwerSRG}, one of the following cases holds:

\begin{enumerate}[\rm (1)]
\item $V(\Ga)$ is an orbit of $T$ acting on the set of singular (or isotropic) points of a projective space.

\item $V(\Ga)$ is an orbit of $T$ acting on the set of maximal totally singular (or isotropic) subspaces of a vector space, and $T$ is one of the groups $\PSp_4( q)$, $\PSU_4( q)$,  $\PSU_5( q)$, $\mathrm{P\Omega}^{-}_6(q)$, $\mathrm{P\Omega}^{+}_8(q)$ or $\mathrm{P\Omega}^{+}_{10}(q)$.

\item  $V(\Ga)$ is an orbit of $T$ acting on the set of nonsingular points of a projective space, and $T$ is one of the groups $\PSU_m( 2)$, $\mathrm{P\Omega}^{\pm}_{2m}(2)$, $\mathrm{P\Omega}^{\pm}_{2m}(3)$ or $\mathrm{P\Omega}_{2m-1}( 3)$.

\item $V(\Ga)$ is an orbit of $T$ acting on the set of nonsingular hyperplanes of a projective space, and $T=\mathrm{P\Omega}_{2m-1}( 4)$ or $\mathrm{P\Omega}_{2m-1}( 8)$, where in the latter case $G= \mathrm{P\Omega}_{2m-1}(8).3$.

\item $(T, T_\a)=(\PSU_3(3), \PSL_3(2))$, $(\PSU_3(5), \A_7)$, $(\PSU_4(3), \PSL_3(4))$, $(\PSp_6(2), G_2(2))$,  $(\mathrm{P\Omega}_{7}(3), G_2(3))$ or $(\PSU_6(2), \PSU_4(3).2)$, where $\a\in V(\Ga)$.
\end{enumerate}

We will complete the proof of this final step case by case.

\smallskip
\f{\bf Step~5.1.}\ Case (5) cannot happen.
\smallskip

If $(T, T_\a)=(\PSU_3(3), \PSL_3(2))$, then by~\cite[\textsection{10.14}]{BrouwerSRG}, $\Ga$ has parameters $(v, k, \lambda, \mu)=(36, 14, 4, 6)$, and by Lemma~\ref{lm:Gaorder}, we have $  {vk(k-\lambda-1)}/{2} = 2^2\cdot 3^4 \cdot 7 $ divides $|\Aut(\PSU_3(3))|=2^6\cdot 3^3 \cdot 7$, a contradiction.

If $(T, T_\a)=(\PSU_3(5),\A_7)$, then by~\cite[\textsection{10.19}]{BrouwerSRG}, $\Ga$ has parameters $(v,k,\lambda,\mu)=(50,7,0,1)$, and by Lemma~\ref{lm:Gaorder}, we have $    {v(v-k-1)(v-2k+\mu-2)}/{6} =2 \cdot 5^3 \cdot 7^2$ divides $ |\Aut(\PSU_3(5))|=2^5\cdot 3^3 \cdot 5^3\cdot 7$, a contradiction.

If $(T, T_\a)=(\PSU_4(3), \PSL_3(4))$, then by~\cite[\textsection{10.19}]{BrouwerSRG}, $\Ga$ has parameters $(v,k,\lambda,\mu)=(162,56,10,24)$, and computation in Magma~\cite{BCP} shows that $\Aut(\Ga)=\PSU_4(3).2^2$ is not transitive on $\Ome_4(\Ga)$, which is impossible by Lemma~\ref{lm:Gaorder}.

If $(T, T_\a)=(\PSp_6(2), G_2(2))$, then by~\cite[\textsection{10.19}]{BrouwerSRG}, $\Ga$ has parameters $(v,k,\lambda,\mu)=(120,56,28,24)$, and by Lemma~\ref{lm:Gaorder}, we have $  {vk\lambda}/{6} = 2^7 \cdot 5 \cdot 7^2 $ divides $ |\Aut(\PSp_6(2))|=2^9\cdot 3^4 \cdot 5 \cdot 7$, a contradiction.

If $(T, T_\a)=(\mathrm{P\Omega}_{7}(3),G_2(3))$, then by~\cite[\textsection{10.19}]{BrouwerSRG}, $\Ga$ has parameters $(v,k,\lambda,\mu)=(1080,351,126,108)$, and by Lemma~\ref{lm:Gaorder}, we have $   {v(v-k-1)(v-2k+\mu-2)}/{6} =2^7 \cdot 3^3\cdot 5  \cdot 7 \cdot 11^2 \cdot 13 $ divides $ |\Aut(\mathrm{P\Omega}_{7}(3))|=2^{10}\cdot 3^9 \cdot 5 \cdot 7\cdot 13$, a contradiction.

If $(T,T_v)=(\PSU_6(2), \PSU_4(3).2)$, then by~\cite[\textsection{10.19}]{BrouwerSRG}, $\Ga$ has parameters $(v,k,\lambda,\mu)=(1408,567,246,216)$, and by Lemma~\ref{lm:Gaorder}, we have $ {vk\lambda }/{6}= 2^7\cdot 3^4 \cdot 7 \cdot 11 \cdot 41 $ divides $ |\Aut(\PSU_6(2))|=2^{16}\cdot 3^7 \cdot 5 \cdot 7 \cdot 11$, a contradiction.

\smallskip

In the remainder, we let $T$ be a classical group satisfying Step~5, and let $V$ be the vector space on which $T$ acts naturally,
and let $Q$ be the quadratic  form associated with $T$ if $T$ is an  orthogonal  group.

\smallskip
\f{\bf Step~5.2.}\ If Case (1) happens, then $T=\mathrm{P\Omega}^{-}_6(q)$, and $\Ga\cong \mathrm{\Gamma}(\mathrm{O}^-_6(q))$ or its complementary graph. In particular, $\Ga$ is $3$-homogeneous.
\smallskip

In Case (1), $V(\Ga)$ is an orbit of $T$ acting on the set of singular (or isotropic) points of a projective space.

Assume first that $T=\PSp_{2m-2}(q)$ with $m\geq 3$.
Let $\boldsymbol{\beta}(,)$ be the symplectic form on $V$ associated with $T$.
By \cite[\textsection{2.5.1}]{BrouwerSRG}, we may assume that $\Ga$ is isomorphic to the symplectic graph $\mathrm{\Gamma}(\mathrm{PSp}_{2m-2}(q))$, namely, $\Ga$ is a graph with vertex set the $1$-subspaces of $V$, and two different vertices $\lg e \rg$ and $\lg f\rg$ are adjacent whenever $\boldsymbol{\beta}(e, f)=0$.
Since $2m-2\geq 4$, there exist two mutually orthogonal hyperbolic pairs $(e_1, f_1)$ and $(e_2, f_2)$. Let $e_3=e_1+f_1$. By direct computation, we see that $\boldsymbol{\beta}(e_3, e_1)\neq 0$ and $\boldsymbol{\beta}(e_3, e_2)\neq 0$.
Let
\[
 \alpha_1=\langle e_1\rangle, \alpha_2=\langle f_1\rangle, \alpha_3=\langle e_3\rangle, \alpha_4=\langle e_2+e_3\rangle.
\]
Then  $[\{\alpha_1,\alpha_2,\alpha_3\}]\cong [\{ \alpha_1,\alpha_2,\alpha_4\}] \cong 3\K_1$.
Since $\dim(\langle \alpha_1,\alpha_2,\alpha_3\rangle)=2$ while $\dim(\langle \alpha_1,\alpha_2,\alpha_4\rangle)=3$,  $[\{ \alpha_1, \alpha_2, \alpha_3\}]$ and $[\{\alpha_1, \alpha_2, \alpha_4\}]$ are not equivalent under $G$, contradicting that $\Ga$ is $(G,3)$-set-homogeneous.

Assume now that $T=\PSU_m(q)$ with $m\geq 3$.
Let $\boldsymbol{\beta}(,)$ be the Hermitian form on $V$ associated with $T$.
By \cite[\textsection{2.7.3}]{BrouwerSRG}, $\Ga$ is isomorphic to the unitary graph $\mathrm{\Gamma}(\mathrm{PSU}_{m}(q))$, namely, $\Ga$ is a graph with vertex set the set of totally isotropic $1$-subspaces of $V$, and two different vertices $\lg e\rg$ and $\lg f\rg$ are adjacent whenever $\boldsymbol{\beta}(e, f)=0$. Since $\PSU_3( q)$ acts $2$-transitively on the set of isotropic $1$-subspaces, we have $m\geq4$. Hence there exist two mutually orthogonal hyperbolic pairs $(e_1, f_1)$ and $(e_2, f_2)$. Let $e_3=e_1+(t-t^q)f_1$ with $t\in \GF(q^2)\setminus \GF(q)$.
By direct computation, we see that $e_3$ is isotropic, and $\boldsymbol{\beta}(e_3,e_1)\neq 0$ and $\boldsymbol{\beta}(e_3,e_2)\neq 0$.
Let
\[
 \alpha_1=\langle e_1\rangle, \alpha_2=\langle f_1\rangle, \alpha_3=\langle e_3\rangle, \alpha_4=\langle e_2+e_3\rangle.
\]
Then  $[\{\alpha_1,\alpha_2,\alpha_3\}]\cong [\{ \alpha_1,\alpha_2,\alpha_4\}] \cong 3\K_1$. Since $\dim(\langle \alpha_1, \alpha_2, \alpha_3\rangle)=2$ while $\dim(\langle \alpha_1, \alpha_2, \alpha_4\rangle)=3$,  $[\{ \alpha_1, \alpha_2, \alpha_3\}]$ and $[\{\alpha_1, \alpha_2, \alpha_4\}]$ are not equivalent under $G$, contradicting that $\Ga$ is $(G,3)$-set-homogeneous.

We now consider the orthogonal groups. In this case, we may assume that $\Ga$ is isomorphic to the so-called orthogonal graph (see, for its definition, \cite[\textsection{2.6.2}]{BrouwerSRG}).
In particular, following~\cite[\textsection{2.6.2}]{BrouwerSRG} for the notation for the orthogonal graph, if $T=\mathrm{P\Omega}^{\epsilon}_n(q)$, where $n$ is even and $\epsilon \in \{+,- \}$,  then $\Ga \cong \Gamma(\mathrm{O}^{\epsilon}_n(q))$, and if $T=\mathrm{P\Omega}_n(q)$, where $n$ is odd, then $\Ga \cong \Gamma(\mathrm{O}_n(q))$.

If $T=\mathrm{P\Omega}_{5}( q)$ with $q$ even, then by \cite[\textsection{2.6.4}]{BrouwerSRG}, $\Ga$ is also isomorphic to the symplectic graph $\mathrm{\Gamma}(\mathrm{PSp}_{4}(q))$, and we have already proved that this graph is not $(G,3)$-set-homogeneous (see the 2nd paragraph in the proof of Step~5.2).

If $T=\mathrm{P\Omega}_{5}(q)$ with $q$ odd, then $|G|$ divides $2sq^4(q^4-1)(q^2-1)$, and from~\cite[\textsection{2.6.3}]{BrouwerSRG}, we see that $\Ga$ has parameters \[(v,k,\ld,\mu)=(q^3+q^2+q+1, q(q+1), q-1, q+1).\]
Since $\Ga$ is  $(G,3)$-set-homogeneous, by Lemma~\ref{lm:Gaorder}, $|\Ome_4(\Ga)|$ divides $|G|$, where $|\Ome_4(\Ga)|=v(v^2-3vk-3v+3k^2-k\lambda+3k+2)/6=(q^9-q^5)/6=q^5(q^4-1)/6$. It implies $q \mid 3s$.
The only candidate for $q$ is $3$.
 However, computation in Magma~\cite{BCP} shows that in the case $q=3$, $\Aut(\mathrm{P\Omega}_{5}(3))$ is not transitive on the set of independent subsets of $\Ga$ of size $3$, contradicting that $\Ga$ is $3$-set-homogeneous.

If either $(n, \epsilon)=(6,+)$ or $n\geq 7$, then there exist three mutually orthogonal hyperbolic pairs $(e_i, f_i)$ for all $i\in \{1, 2, 3\}$.
Let
\[
\alpha_1=\langle e_1\rangle, \alpha_2=\langle e_2\rangle, \alpha_3=\langle e_1+e_2\rangle,\alpha_4=\langle e_3\rangle.
\]
Then $[\{ \alpha_1,\alpha_2,\alpha_3\}]\cong [\{ \alpha_1,\alpha_2,\alpha_4\}] \cong  \K_3$.
Clearly,  $\dim(\langle \alpha_1,\alpha_2,\alpha_3\rangle)=2$ while $\dim(\langle \alpha_1,\alpha_2,\alpha_4\rangle)=3$.
It follows that the two triangles $[\{ \alpha_1, \alpha_2, \alpha_3\}]$ and $[\{\alpha_1, \alpha_2, \alpha_4\}]$ are not equivalent under $G$, contradicting that $\Ga$ is $(G,3)$-set-homogeneous.

Thus, $T=\mathrm{P\Omega}^{-}_{6}(q)$, and by \cite[\textsection{2.6.1}]{BrouwerSRG}, up to a complement, we have $\Ga\cong \mathrm{\Gamma}(\mathrm{O}^-_6(q))$.  By~\cite[Corollary~1.2]{Cameron-3-hom}(vi), $\Ga$ is $3$-homogeneous, as required.

\smallskip
\f{\bf Step~5.3.}\ If Case (2) happen, then $T=\PSU_4(q)$ and $\Ga\cong \mathrm{\Gamma}(\mathrm{O}^-_6(q))$.
\smallskip

In Case (2), $V(\Ga)$ is an orbit of $T$ acting on the set of maximal totally singular (or isotropic) subspaces of $V$, and $T$ is one of the groups $\PSp_4( q)$, $\PSU_4( q)$,  $\PSU_5( q)$, $\mathrm{P\Omega}^{-}_6(q)$, $\mathrm{P\Omega}^{+}_8(q)$ or $\mathrm{P\Omega}^{+}_{10}(q)$.
Let $\Delta(T)$ be the graph with vertex set the set of maximal totally singular (or isotropic) subspaces of $V$, and two vertices $\alpha, \beta$ are adjacent whenever $\alpha\cap \beta$ has codimension $1$ in both $\alpha$ and $\beta$.

First, let $T$ be one of $\PSp_4( q)$, $\PSU_4( q)$,  $\PSU_5( q)$ and $\mathrm{P\Omega}^{-}_6(q)$. Then by \cite[Theorems~11.3.2, 2.2.19 \& Table~2.1]{BrouwerSRG}, we have $\Ga\cong\Delta(T)$. If $T=\PSU_4(q)$, then by \cite[\textsection{2.7.5}]{BrouwerSRG}, we have $\Delta(\PSU_4(q))\cong\mathrm{\Gamma}(\mathrm{O}^-_6(q))$, and by \cite[Corollary~1.2]{Cameron-3-hom}(vi), $\Ga$ is $3$-homogeneous, as required.

Next we shall prove there exist no $3$-set-homogeneous graphs arising from each of the groups $\PSp_4( q)$,  $\PSU_5( q)$, $\mathrm{P\Omega}^{-}_6(q)$.

If $T=\mathrm{PSp}_4(q)$ with $q\geq 3$, then from~\cite[\textsection{2.6.4}]{BrouwerSRG} we see that $ \Delta(\mathrm{PSp}_4(q))$ is isomorphic to the orthogonal graph $\Gamma(\mathrm{O}_5(q))$ associated with $ \mathrm{P\Omega}_{5}(q)$, which has been proved to be not $(G,3)$-set-homogeneous in Step~5.2.

If $T=\PSU_5(q)$, then $|G|$ is a divisor of $2sq^{10}(q^5+1) (q^4-1) (q^3+1) (q^2-1)$. By~\cite[Theorem~2.2.19~and~Table~2.1]{BrouwerSRG},  we know that $\Ga=\Delta(\PSU_5(q))$ has parameters \[(v,k,\ld,\mu)=((q^5+1)(q^3+1), q^3(q^2+1), q^3-1, q+1).\]
Since $\Ga$ is $(G,3)$-set-homogeneous, by Lemma~\ref{lm:Gaorder}, $\ld=6|G|/(vk)$ divides $12sq^{7}(q^2-1)^2$.
If $q=4$, then $7\mid \ld=q^3-1$ but $7\nmid 12sq^{7}(q^2-1)^2$, a contradiction. If $q\neq 4$, then by Lemma~\ref{primitive-prime}, there exists a Zsigmondy  prime divisor $r$ of $q^3-1=p^{3s}-1$ such that $r$ does not divide $p^i-1$ for $i<3s$. It follows that $r\mid 12s$. However, by Lemma~\ref{primitive-prime}, we have $r\equiv 1\ (\mod 3s)$, a contradiction.

If $T=\mathrm{P\Omega}^{-}_6(q)$, then from~\cite[\textsection{2.7.5}]{BrouwerSRG}, we see that $\Delta(\mathrm{P\Omega}^{-}_6(q))\cong\Gamma(\PSU_4(q))$. However, in Step~5.2, we already proved that $\Gamma(\PSU_4(q))$ is not $3$-set homogeneous, a contradiction.

Second, let $T=\mathrm{P\Omega}^{+}_8(q)$ or $\mathrm{P\Omega}^{+}_{10}( q)$. Then by \cite[Theorem~11.3.2, Table~2.1 \& \textsection{2.2.12}]{BrouwerSRG}, we conclude that $\Ga\cong\Delta_{1/2}(T)$, which is one of the two connected components of the distance-$2$ graph of $\Delta(T)$. Furthermore, if $T=\mathrm{P\Omega}^{+}_8(q)$, then by \cite[\textsection{2.2.12}]{BrouwerSRG}, $\Delta_{1/2}(T)\cong\Gamma(\mathrm{P\Omega}^{+}_8(q))$, which has been proved to be not $3$-set homogeneous in Step~5.2.
Let $T=\mathrm{P\Omega}^{+}_{10}(q)$. Then $|G|$ is a divisor of $2sq^{20}(q^5-1) (q^8-1)(q^6-1)(q^4-1)(q^2-1)$.
By~\cite[Theorem~2.2.20~and~Table~2.1]{BrouwerSRG} we conclude that
$\Ga=\Delta_{1/2}(\mathrm{P\Omega}^{+}_{10}(q))$ has parameters \[(v,k,\ld,\mu)=(\prod_{i=1}^4(q^i+1), \frac{q(q^2+1)(q^5-1)}{(q-1)q(q+1)}, q^2(q+1)(q^2+q+1)+q-1, (q^2+1)(q^2+q+1)).\]
Then $\lambda= (q^2+1)(q^3 + 2q^2 + q - 1)$, and since $\Ga$ is $(G,3)$-set-homogeneous, by Lemma~\ref{lm:Gaorder}, we have $\ld$ divides $6|G|/(vk)$.  It follows that \[(q^2+1)(q^3 + 2q^2 + q - 1)\mid 12sq^{19}(q-1)^5(q+1)^2(q^2+q+1).\]
By Lemma~\ref{primitive-prime}, $q^4-1=p^{4s}-1$ has a  Zsigmondy prime divisor $r$ such that $r\nmid (p^{i}-1)$ for $i<4s$. In particular, $r\nmid (q-1)(q^2-1)(q^3-1)$. It follows that $r\nmid (q-1)^5(q+1)^2(q^2+q+1)$. Since $q^4-1=(q^2+1)(q^2-1)$, one has $r\mid (q^2+1)$, and hence $r\mid 12sq^{19}(q-1)^5(q+1)^2(q^2+q+1)$. Therefore, $r\mid 12s$. However, by Lemma~\ref{primitive-prime}, we have $r\equiv 1\ (\mod 4s)$, a contradiction.

\smallskip
\f{\bf Step~5.4.}\ Case (3) cannot happen.
\smallskip

Suppose by way of contradiction that Case~(3) happens. Then $V(\Ga)$ is an orbit of $T$ acting on the set of nonsingular points of a projective space, and $T$ is one of the groups $\PSU_m( 2)$, $\mathrm{P\Omega}^{\pm}_{2m}(2)$, $\mathrm{P\Omega}^{\pm}_{2m}(3)$ or $\mathrm{P\Omega}_{2m-1}( 3)$, where $m\geq 3$.

Let $T=\PSU_m(2)$ with $m\geq 3$. By \cite[\textsection{3.1.6}]{BrouwerSRG}, we may assume that $\Ga=\overline{NU_m(2)}$, and the vertices of $\overline{NU_m(2)}$ are the nonisotropic $1$-subspaces of $V$, and two vertices are adjacent whenever they are orthogonal. Let $\boldsymbol{\beta}(,)$ be the Hermite form of $V$ associated with $T$.
By~\cite[Proposition~1.5.29]{Bray-Holt-Dougal},  $V$ has a basis $\{e_1,e_2,\ldots,e_m\}$ such that $\boldsymbol{\beta}(e_i,e_i)= 1$ for all $i\in \{1,2,\ldots,m\}$ and $V=\langle e_1\rangle \perp \langle e_2\rangle \perp \cdots \perp \langle e_m\rangle $.
Notice that $V$ is a vector space over $\GF(4)$.
Write $\GF(4)=\{0,1, t, t+1\}$ with $t^2=t+1$.
Let
\[
\alpha_1=\langle e_1 \rangle, \alpha_2=\langle e_1+e_2+e_3 \rangle,  \alpha_3=\langle te_1+e_2+e_3 \rangle,  \alpha_4=\langle e_1+te_2+e_3 \rangle.
\]
Then $\a_1,\a_2,\a_3,\a_4\in V(\Ga)$. Notice that $\boldsymbol{\beta}( e_1+e_2+e_3,e_1)=1$, $\boldsymbol{\beta}(te_1+e_2+e_3,e_1)=t$, $\boldsymbol{\beta}(te_1+e_2+e_3,  e_1+e_2+e_3)=t$, $\boldsymbol{\beta}(e_1+te_2+e_3, e_1 )=1$ and $\boldsymbol{\beta}(e_1+te_2+e_3, e_1+e_2+e_3 )=t$.
Thus $ [\{ \alpha_1,\alpha_2,\alpha_3\}]\cong [\{ \alpha_1,\alpha_2,\alpha_4\}] \cong  3\K_1$.
Since $\langle \alpha_1,\alpha_2,\alpha_3\rangle=\langle e_1,e_2+e_3\rangle$ has dimension $2$ while $\langle \alpha_1,\alpha_2,\alpha_4\rangle=\langle e_1,e_2,e_3\rangle$ has dimension $3$, there is no $g\in G$ such that $\{ \alpha_1,\alpha_2,\alpha_3\}^g=\{ \alpha_1,\alpha_2,\alpha_4\}$, a contradiction.

Let $T=\mathrm{P\Omega}^{\pm}_{2m}(3)$ or $\mathrm{P\Omega}_{2m-1}(3)$.
Let $\boldsymbol{\beta}(,)$ be the  polar form  associated with $Q$.
By \cite[\textsection{3.1.3}--\textsection{3.1.4}]{BrouwerSRG}, $T$ has two orbits on the set of nonsingular points, say $O_1$ and $O_{-1}$, such that all points in $O_{\varepsilon}$ have $Q$-value $\varepsilon$ for $\varepsilon=1$ or $-1$. By \cite[\textsection{3.1.3}--\textsection{3.1.4}]{BrouwerSRG}, we may assume that $\Ga$ is the graph with vertex set $O_{\varepsilon}$, and two vertices are adjacent whenever they are orthogonal. Since $m\geq 3$, there exist two mutually orthogonal hyperbolic pairs $(e_1,f_1)$ and $(e_2,f_2)$.
Let $d_1=e_1+\varepsilon f_1$, $d_2=e_1+\varepsilon f_1+e_2$, $d_3=e_1+\varepsilon f_1-e_2$ and $d_4=e_1+\varepsilon f_1-\varepsilon f_2$.
By a direct computation, we see that $Q(d_i)=\varepsilon$ for all $i\in \{1,2,3,4\}$, $\boldsymbol{\beta}(d_1, d_2)=\boldsymbol{\beta}(d_1, d_3)=\boldsymbol{\beta}(d_1, d_4)=\boldsymbol{\beta}(d_2, d_3)=-1$ and $\boldsymbol{\beta}(d_2, d_4)=\varepsilon$.
Let $\alpha_i=\langle d_i\rangle$ for all $i\in \{1,2,3,4\}$.  Then all $\alpha_i$ are in $V(\Ga)$, and
 $ [\{ \alpha_1,\alpha_2,\alpha_3\}]\cong [\{ \alpha_1,\alpha_2,\alpha_4\}] \cong  3\K_1$.
However, $\langle \alpha_1,\alpha_2,\alpha_3\rangle=\langle e_1+\epsilon f_1,e_2 \rangle$ has dimension $2$ and $\langle \alpha_1,\alpha_2,\alpha_4\rangle=\langle e_1+\epsilon f_1,e_2,f_2\rangle$ has dimension $3$. So there is no $g\in G$ such that $\{ \alpha_1,\alpha_2,\alpha_3\}^g=\{ \alpha_1,\alpha_2,\alpha_4\}$, contradicting that $\Ga$ is $(G,3)$-set-homogeneous.

Let $T= \mathrm{P\Omega}^{\pm}_{2m}(2)$.
Let $\boldsymbol{\beta}(,)$ be the  polar form  associated with $Q$.
By \cite[\textsection{3.1.2}]{BrouwerSRG}, we may assume that $\Ga$ is the graph with vertex set the set of nonsingular points, and two vertices are adjacent whenever they are orthogonal. Since $m\geq 3$, there exist two mutually orthogonal hyperbolic pairs $(e_1, f_1)$ and $(e_2, f_2)$. Let $\alpha_1=\langle e_1+f_1 \rangle$, $\alpha_2=\langle e_1+e_2+f_2 \rangle$, $\alpha_3=\langle f_1+e_2+f_2 \rangle$ and  $\alpha_3'=\langle e_1+e_2+w \rangle$ with $w \in \langle e_1, f_1, e_2, f_2\rangle^\perp $ and $Q(w)=1$. Then $\a_1,\a_2,\a_3,\a_3'\in V(\Ga)$, and  $ [\{ \alpha_1,\alpha_2,\alpha_3\}]\cong [\{ \alpha_1,\alpha_2,\alpha_3'\}] \cong  3\K_1$. It is easy to see that $ \dim( \langle \alpha_1,\alpha_2,\alpha_3\rangle)=2$ and $ \dim( \langle \alpha_1,\alpha_2,\alpha_3'\rangle)=3$. Therefore, there exists no $g\in G$ such that $\{ \alpha_1,\alpha_2,\alpha_3\}^g=\{ \alpha_1,\alpha_2,\alpha_3'\}$, contradicting that $\Ga$ is $(G,3)$-set-homogeneous.

\smallskip
\f{\bf Step~5.5.}\ Case (4) cannot happen.
\smallskip

Suppose by way of contradiction that Case~(4) happens. Then $V(\Ga)$ is an orbit of $T$ acting on the set of nonsingular hyperplanes of a projective space, and $T=\mathrm{P\Omega}_{2m-1}( 4)$ or $\mathrm{P\Omega}_{2m-1}( 8)$, where $m\geq3$ and in the latter case $G= \mathrm{P\Omega}_{2m-1}(8).3$.
Let $V$ be a vector space of dimension $2m-1$ over $\GF(q)$ with $q=4$ or $8$, provided with a non-degenerate quadratic form $Q$. By \cite[\textsection{3.1.4}]{BrouwerSRG}, the set of nonsingular hyperplanes of $\PG(2m-2, q)$ splits into two parts,  say $O_{1}$ and $O_{-1}$ of sizes $ q^{m-1}(q^{m-1}+1)/2$ and $ q^{m-1}(q^{m-1}-1)/2$, respectively. Furthermore, $\Ga$ is a graph with vertex set $O_{1}$ or $O_{-1}$, and two hyperplanes $x, y$ are adjacent when  $Q\cap x\cap y$  is degenerate.

First, let $T=\mathrm{P\Omega}_{2m-1}(4)$. From~\cite[\textsection{3.1.4}]{BrouwerSRG},  we see that $\Ga$ has parameters:
\[\begin{array}{l}
v=4^{m-1}(4^{m-1}+\varepsilon)/2,\ k=(4^{m-2}+\varepsilon)(4^{m-1}-\varepsilon),\\

\lambda=2(4^{2(m-2)}-1)+3\cdot4^{m-2}\varepsilon, \mu= 2\cdot4^{m-2}(4^{m-2} +\varepsilon),
\end{array}
\]
where $\varepsilon=1$ or $-1$.
By~\cite[Table~8.14]{Bray-Holt-Dougal} we see that $G$ contains no graph automorphism when $T=\mathrm{P\Omega}_5(4)\cong \PSp_{4}(4)$. Therefore, $|G|$ divides $|T.2|=2\cdot 4^{(m-1)^2}\prod_{i=1}^{m-1}(4^{2i}-1)$. Since $\Ga$ is $(G,3)$-set-homogeneous, by Lemma~\ref{lm:Gaorder}, $\lambda $ divides $6|G|/(vk)$, which implies $\lambda \mid 6(2^{2(m-2)}-\varepsilon)\prod_{i=1}^{m-3}(2^{4i}-1)$.
Note that
\[\lambda=(2\cdot 4^{2(m-2)}-\varepsilon)(4^{2(m-2)}+2\varepsilon)=2(2^{4m-7}-\varepsilon)(2^{4m-9}+\varepsilon).\]
If $\varepsilon=1$, then by Lemma~\ref{primitive-prime}, there is a Zsigmondy prime divisor $r $ of $2^{2(4m-9)}-1$ such that $r \mid \lambda$   while $r \nmid 6(2^{2(m-2)}-1)\prod_{i=1}^{m-3}(2^{4i}-1)$, a contradiction.
If $\varepsilon=-1$, then there is a Zsigmondy prime divisor $r$ of $2^{2(4m-7)}-1$ such that $r \mid \lambda$ while $r \nmid 12s(2^{2(m-2)}+1)\prod_{i=1}^{m-3}(2^{4i}-1)$,  a contradiction.

Second, let $T=\mathrm{P\Omega}_{2m-1}(8)$ with $m\geq 3$. Note that the action of $\mathrm{P\Omega}_{2m-1}(8)$ on an orbit of nonsingular hyperplanes of $\PG_{2m-2}(8)$ is permutation isomorphic to the action of $\mathrm{Sp}_{2m-2}(8)$ on the
set $\mathcal{Q}$ of quadratic forms on $W$ polarizing into the given symplectic form $\boldsymbol{\beta}(,)$, where $W=\GF(8)^{2m-2}$.
From~\cite[Table~8.14]{Bray-Holt-Dougal} we see that $G$ contains no graph automorphism when $T=\mathrm{Sp}_4(8)$.
This implies that $G\leq \mathrm{\Gamma Sp}_{2m-2}(8)$.

Let $\GF(8)=\{0, 1, t, t^2, \ldots, t^6\}$ with $t+t^2=t^4$, and let $S=\{i+i^2 \mid i \in \GF(8) \}=\{0, t, t^2, t^4\}$. Let $\Sig$ be the set of quadratic forms on $W$. As proved in \cite[p.327]{Inglis1990}, for each pair $(Q,R)$ of quadratic forms in $\Sig$, there exists a unique vector $y(Q,R) \in W$ such that $Q(w)+R(w)=\boldsymbol{\beta}(w,y(Q,R))^2$ for all $w\in W$. Let $\a(Q,R):=Q(y(Q,R))$. By \cite[Lemma~2(iii)]{Inglis1990}, two quadratic forms $Q$ and $R$ are in the same orbits of $\mathrm{Sp}_{2m-2}(8)$ on $\Sig$ if and only if $\a(Q,R)\in S$. From the proof of \cite[Theorem~1]{Inglis1990}, we see that for $Q\in\Sig$, the orbits of the stabiliser $(\mathrm{Sp}_{2m-2}(8))_Q$ on $Q^{\mathrm{Sp}_{2m-2}(8)}$ are: $\{Q\}$ and $\{R\in\mathcal{Q}\mid \alpha(Q,R)=\ld\}$ with $\ld\in S$. Notice that $\mathrm{\Gamma Sp}_{2m-2}(8)=\{g: g\in \mathrm{\Gamma L}_{2m-2}( 8) \mid \boldsymbol{\beta}(w^g, u^g)=\boldsymbol{\beta}(w, u)^{\sigma(g)}\}$, where $\s$ is a linear function from $\mathrm{\Gamma L}_{2m-2}( 8)$ to $\Aut(\GF(8))$ such that its kernel is $\GL_{2m-2}(8)$. As observed in \cite[p.329]{Inglis1990}, for each $Q \in \Sig$ and each $g\in \mathrm{\Gamma Sp}_{2m}(8)$, we have $Q^g(w)=Q(w^{g^{-1}})^{\sigma(g)}$ for all $w \in W$. Furthermore, by~\cite[Lemma~3]{Inglis1990}, we have $\a(Q^g, R^g)=\a(Q,R)^{\s(g)}$ for all $g\in \mathrm{\Gamma Sp}_{2m-2}(8)$. This implies that the orbits of $(\mathrm{\Gamma Sp}_{2m-2}(8))_Q$  on $Q^{\mathrm{Sp}_{2m-2}(8)}$ are
\[\{Q\}, \{R\in\mathcal{Q}\mid \alpha(Q, R)=0\}\ {\rm and}\ \{R\in\mathcal{Q}\mid \alpha(Q, R)\in\{t,t^2,t^4\}\}.\]

Now let $\{e_1, f_1, e_2, f_2, \ldots, e_{m-1}, f_{m-1}\}$ be a basis of $W$, where $(e_i, f_i)$'s are mutually orthogonal hyperbolic pairs.
Let $Q, R$ be two quadratic forms on $W$ such that
\[\begin{array}{l}
Q(e_i)=Q(f_i)=0\ {\rm for\ all}\ 1\leq i\leq m-1,\ {\rm and}\\

R(e_{1})=R(f_{1})=1,\ {\rm and}\ R(e_i)=R(f_i)=0\ {\rm for\ all}\ 2\leq i\leq m-1.
\end{array}
\]
A calculation shows that $\sum_{i=1}^{m-1}Q(e_i)Q(f_i)=0\in S$ and $\sum_{i=1}^{m-1}R(e_i)R(f_i)=1\notin S$. By \cite[p.327]{Inglis1990} or \cite{Dye1978}, $Q^{\mathrm{Sp}_{2m-2}(8)}$ and $R^{\mathrm{Sp}_{2m-2}(8)}$ are the two orbits of $\mathrm{Sp}_{2m-2}(8)$ on the
set of quadratic forms on $W$.

With the above argument, we may let $\Ga$ be a graph with vertex set $V(\Ga)=Q^{\mathrm{Sp}_{2m-2}(8)}$ or $R^{\mathrm{Sp}_{2m-2}(8)}$, and two different quadratic forms $Q', R'\in\mathcal{Q}$ are adjacent if and only if $\alpha(Q', R')=0$.

Assume first that $V(\Ga)=Q^{\mathrm{Sp}_{2m-2}(8)}$. For $w\in W$, let
\[\begin{array}{l}
X(w)=Q(w)+\boldsymbol{\beta}(w,e_1)^2,\\

Y(w)=Q(w)+\boldsymbol{\beta}(w,f_1)^2,\\

Z(w)=Q(w)+\boldsymbol{\beta}(w, e_1+f_2)^2.\\
\end{array}\]
Then $y(Q, X)=e_1$, $y(Q, Y)=f_1$ and $y(Q, Z)=e_1+f_2$. Furthermore, for all $w\in W$, we have $X(w)+Y(w)=\boldsymbol{\beta}(w,e_1+f_1)^2$ and $X(w)+Z(w)=\boldsymbol{\beta}(w,f_2)^2$. It follows that
$y(X, Y)=e_1+f_1$ and $y(X, Z)=f_2$.

Since $Q(e_1)=Q(f_1)=Q(e_1+f_2)=0$, one has $\alpha(Q, X)=\alpha(Q, Y)=\alpha(Q, Z)=0$. It follows that $X,Y,Z\in V(\Ga)$, and $\{Q, X\}, \{Q, Y\}, \{Q,Z\}\in E(\Ga)$. As $X(e_1+f_1)=Q(e_1+f_1)+\boldsymbol{\beta}(e_1+f_1, e_1)^2=0$, it follows that $\alpha(X,Y)=0$ and so $\{X, Y\}\in E(\Ga)$. Similarly, as $X(f_2)=Q(f_2)+\boldsymbol{\beta}(f_2, e_1)^2=0$, it follows that $\a(X,Z)=0$, and so $\{X, Z\}\in E(\Ga)$. Now we have $[\{Q, X, Y\}] \cong [\{Q, X, Z\}]\cong \mathbf{K}_3$. Since $\Ga$ is $(G,3)$-set-homogeneous, it follows that there exists $g\in G\leq \mathrm{\Gamma Sp}_{2m-2}(8)$ such that $\{Q, X, Y\}^g=\{Q, X, Z\}$. Recall that for $L \in \Sig$ and $g\in \mathrm{\Gamma Sp}_{2m-2}(8)$, we have $L^g(w)=L(w^{g^{-1}})^{\sigma(g)}$ for all $w \in W$. It follows that $(Q+X+Y)^g=Q+X+Z$, and hence $Q+X+Y$ and $Q+X+Z$ are in the same orbit of $\mathrm{\Gamma Sp}_{2m-2}(8)$. On the other hand, for $w\in W$, we have $Q(w)+X(w)+Y(w)=Q(w)+\boldsymbol{\beta}(w,e_1+f_1)$ and $Q(w)+X(w)+Z(w)=Q(w)+\boldsymbol{\beta}(w, f_2)$. Then $e_1+f_1=y(Q, Q+X+Y)$ and $f_2=y(Q, Q+X+Z)$.  Since $\a(Q, Q+X+Y)=Q(e_1+f_1)=1\notin S$ and $\a(Q, Q+X+Z)=Q(f_2)=0\in S$, by \cite[Lemma~2 (iii)]{Inglis1990}, $Q+X+Y\notin Q^{\mathrm{\Gamma Sp}_{2m-2}(8)}$ but $Q+X+Z\in Q^{\mathrm{\Gamma Sp}_{2m-2}(8)}$. This is a contradiction.

Assume now that $V(\Ga)=R^{\mathrm{Sp}_{2m-2}(8)}$. Recall that $R(e_{1})=R(f_{1})=1,\ {\rm and}\ R(e_i)=R(f_i)=0\ {\rm for\ all}\ 2\leq i\leq m-1.$  Since $m\geq 3$, for $w\in W$, we may let
\[\begin{array}{l}
X(w)=R(w)+\boldsymbol{\beta}(w, e_2)^2,\\

Y(w)=R(w)+\boldsymbol{\beta}(w, f_2)^2,\\

Z(w)=R(w)+\boldsymbol{\beta}(w, te_2)^2.\\
\end{array}\]
Then $y(R,X)=e_2, y(R,Y)=f_2$ and $y(R, Z)=te_2$. Furthermore, for all $w\in W$, we have $X(w)+Y(w)=\boldsymbol{\beta}(w, e_2+f_2)^2$, and $X(w)+Z(w)=\boldsymbol{\beta}(w, (t+1)e_2)^2$. It follows that
$y(X, Y)=e_2+f_2$ and $y(X, Z)=(t+1)e_2$.

Since $R(e_2)=R(f_2)=R(te_2)=0$, one has $\alpha(R, X)=\alpha(R, Z)=\alpha(R, Y)=0\in S$. It follows that $X,Y,Z\in V(\Ga)$, and $\{R, X\}, \{R, Y\}, \{R,Z\}\in\Ga$. As $X(e_2+f_2)=R(e_2+f_2)+\boldsymbol{\beta}(e_2+f_2, e_2)^2=1+1=0$, it follows that $\alpha(X,Y)=0$ and so $\{X, Y\}\in E(\Ga)$. Similarly, as $X((t+1)e_2)=0$, it follows that $\a(X,Z)=0$, and so $\{X, Z\}\in E(\Ga)$. This implies that $[\{R, X, Y\}] \cong [\{R, X, Z\}]\cong \K_3$. Again, since $\Ga$ is $(G,3)$-set-homogeneous, it follows that there exists $g\in G\leq \mathrm{\Gamma Sp}_{2m-2}(8)$ such that $\{R, X, Y\}^g=\{R, X, Z\}$. This implies that $(R+X+Y)^g=R+X+Z$, and then $R+X+Y$ and $R+X+Z$ are in the same orbit of $\mathrm{\Gamma Sp}_{2m-2}(8)$. On the other hand, for $w\in W$, we have $R(w)+X(w)+Y(w)=R(w)+\boldsymbol{\beta}(w, e_2+f_2)$ and $R(w)+X(w)+Z(w)=R(w)+\boldsymbol{\beta}(w, (t+1)e_2)$. Then $y(R, R+X+Y)=e_2+f_2$ and $y(R, R+X+Z)=(t+1)e_2$.  Since $\a(R, R+X+Y)=R(e_2+f_2)=1\notin S$ and $\a(R, R+X+Z)=R((t+1)e_2)=0\in S$, by \cite[Lemma~2 (iii)]{Inglis1990}, $R+X+Y\notin R^{\mathrm{\Gamma Sp}_{2m-2}(8)}$ but $R+X+Z\in Q^{\mathrm{\Gamma Sp}_{2m-2}(8)}$. A contradiction occurs.\hfill\qed

\subsection{Proof of Theorem~\ref{th:3sethomo}}\label{subsec:proofTh:3sethomo}

Let $\Ga=(V, E)$ be a finite connected $(G,3)$-set-homogeneous graph. By Lemma~\ref{lem:prim-3ch-is-srg}, we see that, up to a complement, either $\Ga$ is isomorphic to one of the graphs: $\C_5, \K_n, \K_{m[b]}$, or $G$ is primitive on $V$ and $\Ga$ is a strongly regular graph of diameter $3$ and girth $3$. For the former, there is nothing further to prove. Assume therefore the latter happens in what follows.

If the socle of $G$ is neither elementary abelian nor nonabelian simple, then by Theorem~\ref{th:rank3groups} and Lemma~\ref{lm:CaseI}, we have $\Ga\cong \K_n\square\K_n$ or $\K_n\times\K_n$.

If the socle of $G$ is elementary abelian, then by Theorem~\ref{th-affine}, up to a complement, $\Ga$ is isomorphic either to $\K_n\square\K_n$ with $n$ a prime power or to one of the graphs in part (1) of our theorem.

If the socle of $G$ is nonabelian simple, then by Theorem~\ref{th-almost-simple}, up to a complement, $\Ga$ is isomorphic to one of the graphs in parts (3)--(4) of our theorem.

\section{Proof of Theorem~\ref{th-main}}\label{set:proof-of-main}

By definition, both 1-homogeneous graphs and 1-set-homogeneous graphs are exactly the vertex-transitive graphs. Assume now that $k\geq 2$. It is obvious by definition that every $k$-homogeneous graph is $k$-set-homogeneous. By Devillers~\cite{Devillers-2002paper}, any finite $5$-set-homogeneous graph is set-homogeneous and hence homogeneous (by \cite{Ronse}), and moreover, a $4$-set-homogeneous graph is either homogeneous or one of the Schl\"{a}fli graph on $27$ vertices and its complement, which are $4$-homogeneous. By Theorems~\ref{th:2sethomo} and \ref{th:3sethomo}, we know that every $k$-set-homogeneous graph is also $k$-homogeneous for $k=2$ or $3$. This proves the first assertion of Theorem~\ref{th-main}.

Now let $\Ga$ be a $k$-homogeneous graph with $k\geq 2$. If $k\geq 5$, then by \cite{Cameron-6-transitive}, $\Ga$ is homogeneous, and then by \cite{Gardiner1978}, $\Ga$ is one of the graphs in part~(1). If $\Ga$ is $4$-homogeneous but not $5$-homogeneous, then by \cite{Buczak}, we know that part (2) holds. If $\Ga$ is $3$-homogeneous but not $4$-homogeneous, then by Theorem~\ref{th:3sethomo} or \cite{Cameron-3-hom}, we conclude that $\Ga$ is one of the graphs in part (3). Finally, if $\Ga$ is $2$-homogeneous but not $3$-homogeneous, then $\Ga$ and its complement $\ov\Ga$ are both arc-transitive. 
This implies that $\Ga$ must be an orbital graph of a primitive permutation group of rank $3$, which is not one of the graphs in parts (1)--(3). \hfill\qed


\vfill
\eject

\section*{Appendix~1}

\begin{center}
Tables~1--2 of~\cite{Vauhkonen}\medskip

\begin{tabular}{cccc}
\hline
&$G$&$H$&subdegrees\\
\hline
1 &$\PGL_2(7)$& $\mathrm{D}_{16}$& 4,8,8\\
2 &$\PSL_2(8)$&$\mathrm{D}_{18}$& 9,9,9\\
3 &$\M_{10}=\PSL_2(9).2$& $\mz_5{:}\mz_4$&5,10,20\\
4 &$\PL_2(9)$&$\mz_{10}{:}\mz_4$&5,10,20\\
5&$\PL_2(16)$&$\mz_{17}{:}\mz_8$&17,34,68\\
6&$\PL_2(16)$&$(\A_5\times2).2$&12,15,40\\
7&$\PSL_2(19)$&$\A_5$&6,20,30\\
8&$\PSL_2(25)$&$\S_5$&10,24,30\\
9&${\rm P\Sigma L}_2(25)$&$\S_5\times 2$&10,24,30\\
10& $\PL_2(32)$&$\mz_{33}{:}\mz_{10}$&165,165,165\\
11&$\A_8=\GL_4(2)$&$(\A_5\times \A_3){:}2$&10,15,30\\
12&$\S_8=\GL_4(2).2$&$\S_5\times \S_3$&10,15,30\\
13&$\PSL_3(4)$&$\GL_3(2)$&21,42,56\\
14 &$\PSL_3(4).2$&$\GL_3(2){:}2$&21,42,56\\
15 &${\rm P\Sigma L}_3(4)$&$\GL_3(2)\times2$&21,42,56\\
16 &$\PSL_3(4).\lg\s\rg$&$\GL_3(2){:}2$&21,42,56\\
17 &${\rm P\Sigma L}_3(4).\lg\s\rg$&$\GL_3(2){:}2\times2$&21,42,56\\
18 &${\rm P\Sigma L}_4(4)=\PO^+_6(4).\lg\phi\rg$&${\rm Sp}_4(4).2$&255,272,480\\
19 &$\Aut(\PSL_4(4))={\rm PSO}^+_6(4)$&${\rm Sp}_4(4).[4]$&255,272,480\\
20 &$\PSL_4(5)=\PO^+_6(5)$&$\PSp_4(5).[a]$&325,600,624\\
&$~\unlhd~G\leq {\rm PO}^+_6(5)$& $a=|G|/|\PSL_4(5)|$&\\
\hline
\end{tabular}
\end{center}

\newpage
\section*{Appendix~2}

Cuypers in \cite{Cuypers} gave a classification of rank at most $5$
primitive groups whose socle is a finite simple group of Lie type, and we restate it as the following theorem (see {\rm\cite[Theorem~1.1]{Cuypers}}).

\begin{theorem}
Let $G$ be a finite group whose socle $S$ is a finite simple group of Lie type.
Suppose $M$ is a maximal proper subgroup of $G$ such that $MS=G$ and the permutation representation
of $G$ acting on the set of left (or right) cosets of $M$ has permutation rank at most five. Then up to conjugation
in $\Aut(S)$ the groups $S$ and $M\cap S$ are as stated in Table~I.
\end{theorem}

\begin{center}
{Table~I}\medskip

\begin{tabular}{|c|c|l|l|}
\hline nr. &  $S=\soc(G)$ &  $S_\alpha=M\cap S$ & restrictions \\
\hline
1 & $\PSL_n(q)$ & stabiliser of a $m$-space & $m\leq {\rm min}(4, {n\over 2}), n\geq 2$\\
2 & $\PSL_2(q)$ & stabiliser of a decomposition of  $V$ as& $q\in\{4,5,7,8,9,16,32\}$\\
&& $V=V_1\oplus V_2$, $\dim(V_1)=1$&\\
3 & $\PSL_2(q)$ & stabiliser of an extension field $F$ of& $q\in\{4,5,7,8,9,16,32\}$\\
&& $\GF(q)$ of order $q^2$&\\
4 & $\PSL_2(q)$ & stabiliser of a subfield of $\GF(q)$ with& $q\in\{4,9,16,25,49,64,$\\
&& index $2$&$81\}$\\
5 & $\PSL_3(4)$ & stabiliser of a subfield of $\GF(q)$ with & \\
&&index $2$&\\
6 & $\PSL_3(4)$ & normaliser of a non-degenerate & \\
&&hermitian form&\\
7 & $\PSL_4(q)$ & normaliser of a non-degenerate & $q\leq 9$\\
&&symplectic form&\\
8 & $\PSL_6(2)$ & normaliser of a non-degenerate & \\
&&symplectic form&\\
9 & $\PSL_2(11)$ & $\mathrm{A}_5$ & \\
10 & $\PSL_2(19)$& $\mathrm{A}_5$&\\
11 & $\PSL_3(4)$& $\mathrm{A}_6$&\\
12 & $\PSL_n(q)$ & stabiliser of pair $\{U,W\}$ with $U\leq W$&$n\geq 3$ and $G$ contains a\\
&& and  $\dim(U)=1$ &  graph  automorphism  \\
13 & $\PSL_n(2)$ & stabiliser of pair $\{U,W\}$ with $U\leq W$ &$n\geq 3$ and  $G$ contains a\\
&&and  $\dim(U)=1$  & graph automorphism\\
\hline

14 & $\PSp_n(q)$ & stabiliser of a singular $1$-space   &$n$ even  \\
15&$\PSp_n(q)$&stabiliser of a singular $1$-space& $n=4$ or $6$\\
16&$\PSp_n(q)$&stabiliser of a maximal totally& $n=4,6$ or $8$\\
&& singular subspace&\\
17&$\PSp_n(2)$&stabiliser of a non-degenerate $2$-space& $n\geq 6$\\
18&$\PSp_4(q)'$&stabiliser of a decomposition  of $V$ & $q\in\{2,3,4,5,7,9,$\\
&&  into a direct sum of two  isomorphic&$16,32\}$ \\
&&  non-degenerate $2$-spaces & \\
19&$\PSp_4(q)$&stabiliser of an extension field $F$ of& $q\in\{3,5,7,9\}$\\
&&   $\GF(q)$ of order $q^2$ & \\
20&$\PSp_n(q)$&normaliser of a non-degenerate & $q\in\{2,4,8,16,32\}$\\
&&  quadratic form $V$ & \\
21&$\PSp_6(q)$& $G_2(q)$ & $q\in\{2,4,8\}$\\
\hline
\end{tabular}
\end{center}

\begin{table}
\begin{center}
\begin{tabular}{|c|c|l|l|}
\hline
nr. &  $S=\soc(G)$ &  $S_\alpha=M\cap S$ & restrictions \\
\hline
22 & $\PSp_4(q)$ & stabiliser of pair $\{U,W\}$ &$q$ is even and  \\
&&with $U\leq W$ and $W$ is totally singular & $G$ contains a\\
&&$\dim(U)=1$ and $\dim(W)=2$ & graph automorphism\\ \hline
23 & $\PSU_n(r)$ & stabiliser of a singular $1$-space &$n\geq 3$  \\
24 & $\PSU_n(r)$ & stabiliser of a totally singular $2$-space &$4\leq n\leq 7$  \\
25 & $\PSU_n(r)$ & stabiliser of a maximal  totally &$6\leq n\leq 9$  \\
&& singular subspace&\\
26 & $\PSU_n(r)$ & stabiliser of a nonsingular $1$-space &$r\in\{2,3,4,8\}$  \\
27 & $\PSU_3(3)$ & stabiliser of an orthogonal bases&  \\
28 & $\PSU_4(2)$ & stabiliser of an orthogonal bases&  \\
29 & $\PSU_4(r)$ & normaliser of $\PSp_4(r)$  & $r\in\{2,3,4,5,7,8,9\}$ \\
30 & $\PSU_6(2)$ & normaliser of $\PSp_6(2)$ &  \\
31 & $\PSU_4(3)$ & $2^4{:}\mathrm{A}_6$ &  \\
32 & $\PSU_3(5)$ & $\mathrm{M}_{10}$ &  \\
33 &$\PSU_3(5)$& $\mathrm{A}_7$&\\
34 &$\PSU_4(3)$& $\PSL_3(4)$&\\
35 &$\PSU_6(2)$& $\PSU_4(3)$&\\
36 &$\PSU_3(3)$& $\PSL_3(2)$&\\
\hline
37 &$\PO_{2m+1}(q)$& stabiliser of a singular $1$-space&\\
38 &$\PO^{\pm}_{2m }(q)$& stabiliser of a singular $1$-space&\\
39 &$\PO^{\pm}_{2m }(q)$& stabiliser of a totally singular $2$-space& $m\leq 3$\\
40 &$\PO^{\pm}_{2m }(q)$& stabiliser of a totally singular $2$-space& $m\leq 4$\\
41 &$\PO_{2m+1}(q)$& stabiliser of a maximal totally  & $m\leq 4$\\
&&singular subspace&\\
42 &$\PO^-_{2m }(q)$& stabiliser of a maximal totally & $m\leq 5$\\
&&singular subspace&\\
43 &$\PO^+_{2m }(q)$& stabiliser of a maximal totally & $m\leq 9$\\
&&singular subspace&\\
44 &$\PO_{2m+1}(q)$& stabiliser of a nonsingular $1$-space & $q\in\{3,5,7,9\}$\\
45 &$\PO^{\pm}_{2m }( q)$& stabiliser of a nonsingular $1$-space & $q\in\{2,3,4,5,7,8,9\}$\\
46 &$\PO^{\pm}_{2m }(2)$& stabiliser of a non-degenerate $2$-space & $m\geq 3$\\
&& of elliptic type&\\
47 &$\PO_7(q)$& $G_2(q)$ & $q\in\{3,5,7,9\}$\\
48 &$\PO_7(3)$& $\PSp_6(2)$ & \\
49 &$\PO^+_{8}(2)$& $\mathrm{A}_9$ & \\
50 &$\PO^+_{8}(3)$& $\PO^+_{8}(2)$ & \\ \hline
51 &${}^2B_2(q)$& maximal parabolic & \\
52 &${}^2C_2(q)$&maximal parabolic&\\
53 &${}^2G_2(q)$&maximal parabolic&\\
54 &$G_2(q)$&maximal parabolic (2 classes)&\\
55 &$G_2(q)$&${\rm SU}_3(q).2$& $q\in\{2,3,4,5,7,8,$\\
&&&$16,32\}$\\
56 &$G_2(q)$&${\rm SL}_3(q).2$& $q\in\{2,3,4,5,8\}$\\
57 &$G_2(4)$&$\mathrm{J}_2$&\\
\hline
58 &${}^3D_4(q)$&maximal parabolic (2 classes)&\\

\hline
59 &${}^2F_4(q)'$&maximal parabolic (2 classes)&\\
60 &${}^2F_4(2)'$&$\PSL_3(3).2$&\\
\hline
\end{tabular}
\end{center}
\end{table}

\begin{table}
\begin{center}
\begin{tabular}{|c|c|l|l|}
\hline nr. &  $S=\soc(G)$ &  $S_\alpha=M\cap S$ & restrictions \\
\hline

61 &$F_4(q)$&maximal parabolic of type $B_3$ or $C_3$&\\
62 &$F_4(2)$&$B_4(2)$&\\
\hline
63 &${}^2E_6(q)$&maximal parabolic of type ${}^2A_5$ or ${}^2D_5$&\\
64 &${}^2E_6(2)$& $F_4(2)$&\\
\hline
65 &$E_6(q)$&maximal parabolic of type $A_5$ or $D_5$&\\
66 &$E_6(2)$&$F_4(2)$& $G$ contains a graph \\
&&& automorphism \\
\hline
67 &$E_7(q)$&maximal parabolic of type $E_6$ or $D_6$&\\
\hline
68 &$E_8(q)$&maximal parabolic of type $E_7$&\\
\hline
\end{tabular}
\end{center}
\end{table}

\end{document}